\documentclass[12pt,reqno]{amsart}
\usepackage{amsmath} 

\usepackage{amsthm}
\usepackage{amssymb}
\usepackage{graphics}
\usepackage{latexsym}

\numberwithin{equation}{section}
\newtheorem{thm}{Theorem}[section]
\newtheorem{prop}[thm]{Proposition}
\newtheorem{lem}[thm]{Lemma}
\newtheorem{cor}[thm]{Corollary}

\theoremstyle{remark}
\newtheorem{rem}[thm]{Remark}
\newtheorem{defn}{Definition}

\newcommand{\BBB}{\mathbb}
\newcommand{\R}{{\BBB R}}
\newcommand{\Z}{{\BBB Z}}
\newcommand{\T}{{\BBB T}}

\newcommand{\N}{{\BBB N}}
\newcommand{\C}{{\BBB C}}

\newcommand{\lec}{{\ \lesssim \ }}

\newcommand{\ga}{\gamma}

\newcommand{\vp}{\varphi}

\newcommand{\e}{\varepsilon}

\newcommand{\p}{\partial}
\newcommand{\la}{\lambda}
\newcommand{\La}{\Lambda}

\newcommand{\de}{\delta}

\newcommand{\supp}{\operatorname{supp}}

\newcommand{\dis}{\displaystyle}
\newcommand{\mT}{\mathcal{T}}

\newcommand{\EQ}[1]{\begin{equation} \begin{split} #1
 \end{split} \end{equation}}
\newcommand{\EQS}[1]{\begin{align} #1 \end{align}}

\newcommand{\EQQ}[1]{\begin{equation*} \begin{split} #1
 \end{split} \end{equation*}}

\newcommand{\ti}{\widetilde}
\newcommand{\ha}{\widehat}

\title[Unconditional well-posdeness for the 4NLS on the torus]
{Unconditional well-posedness for the fourth order nonlinear Sch\"{o}dinger type equations on the torus 
}
\author[T. K.  Kato]{Takamori Kato}
\address[T. K. Kato]{Mathematical Science Course, Faculty of Science and Engineering, Saga University, Saga, 840-8502, 
Japan}
\email[T.K. Kato]{tkkato@cc.saga-u.ac.jp}

\begin{document}

\begin{abstract}
We prove the unconditional well-posedness for the fourth order nonlinear Schr\"{o}dinger type equations
 in $H^s(\T)$ when $s \ge 1$, which includes the non-integrable case. 
This regularity threshold is optimal because the nonlinear terms cannot be defined in the space-time distribution framework for $s<1$. 
The main idea is to employ the normal form reduction and a kind of cancellation property to deal with derivative losses. 
\end{abstract}

\maketitle
\setcounter{page}{001}

\section{Introduction}

In this paper, we consider the Cauchy problem for the fourth order nonlinear Schr\"{o}dinger type equations (4NLS) on $\T:= \R/ 2\pi \Z$:
\EQS{
&i \p_t u +\p_x^4 u = -\frac{3}{8} \la_1|u|^4 u+ \la_2 \bar{u} (\p_x u)^2 + \la_3 u |\p_x u|^2 + \la_4 u^2 \p_x^2 \bar{u} 
+ \la_5 |u|^2 \p_x^2 u, \label{4NLS1}\\
&u(0,x)=\vp(x)\in H^s(\T)\label{initial}
}
where $\la_1,\la_2, \la_3, \la_4$ and $\la_5$ are real constants, $u=u(t,x): \R\times \T\to \C$ is an unknown function and $\vp=\vp(x): \T \to \C$ 
is a given function. 
The 4NLS arises in the context of a motion of vortex filament. More precisely, using the localized induction approximation, 
Da Rios \cite{DR} proposed some equation which approximates the three dimensional motion of an isolated vortex filament embedded 
in an inviscid incompressible fluid filling an infinite region. 
The Da Rios equation is reduced to the cubic nonlinear Schr\"{o}dinger equation 
\begin{equation*}
i \p_t u +\p_x^2 u= - \frac{1}{2} |u|^2 u
\end{equation*}
via the Hasimoto transform \cite{Ha}. 
To describe the motion of actual vortex filament more precisely, some detailed models taking into account the effect of 
higher order corrections to the equation have been introduced by Fukumoto and Moffatt \cite{FM}.  
By using the Hasimoto transform, the Fukumoto-Moffatt equation is rewritten into the generalization of \eqref{4NLS1} as follows:
\begin{align} \label{fm4NLS}
i\p_t u + \p_x^4 u + \mu \p_x^2 u= & -\frac{3}{8} \la_1 |u|^4 u + \la_2 \bar{u} (\p_x u)^2 + \la_3 u |\p_x u|^2 \notag \\
& + \la_4 u^2 \p_x^2 \bar{u} + \la_5 |u|^2 \p_x^2 u   +\la_6 |u|^2 u  
\end{align}
where $\la_1 ,\dots, \la_6$ and $\mu$ are real constants. 
For the physical background, see \cite{FM}. 

We observe some conserved quantities. Define $\int_{\T} f(x) \, dx:= \frac{1}{2 \pi} \int_{0}^{2 \pi} f(x) \, dx$ for a 
$2 \pi$-periodic function $f$. Put
\EQQ{
&E_1(u) (t) :=\int_{\mathbb{T}} |u(t,x)|^2  \, dx, \hspace{0.3cm} E_2(u)(t):= \text{Im} \int_{\T} \bar{u}(t,x) \, \p_x u (t,x) \, dx,  \\
&E_3(u)(t):=\frac{1}{2} \int_{\mathbb{T}}  |\p_x u(t,x)|^2 + \frac{\la_4}{2} |u(t,x)|^4 \, dx, \\
&E_4(u)(t):= \frac{1}{2} \int_{\mathbb{T}}  |\p_x^2 u(t,x)|^2 + \frac{\la_1}{8} |u(t,x)|^6+ (\la_2+\la_4) |u(t,x)|^2 |\p_x u(t,x)|^2 \\
& \hspace{3cm} + \la_4 \text{Re}[ u^2 (t,x) \, (\overline{\p_x u}(t,x) )^2  ] \,  dx, \\
&A_1: \la_5= \la_2+ \la_4, \qquad A_2: \la_2+\la_3= \la_4 +\la_5, \\
& A_3: 8\la_4= 2  \la_2+\la_3, \,\, \la_4 (\la_2+ 2\la_4-2 \la_5)=-3 \la_1/4 . 
}
For each $j=1,2,3$, we have the conservation law
$E_j(u)(t)=E_j(\vp)$ in the formal sense when the assumption $A_j$ holds.
We have $E_4(u)(t)=E_4(\vp)$ if $A_1$ and $A_2$ hold.
In this case, \eqref{4NLS1} can be regarded as the Hamiltonian PDE $\p_t u =i \p_{\bar{u}} E_4(u)$. 
When $A_1$, $A_2$ and $A_3$ hold, that is, 
$\la_1=4\la_4^2$, $\la_2= 3 \la_4$, $\la_{3}=2 \la_4$ and $\la_5=4 \la_4$, 
\eqref{4NLS1} is a completely integrable system and in the Schr\"{o}dinger hierarchy. 
Then, \eqref{4NLS1} has infinitely many conserved quantities including $E_1$, $E_2$, $E_3$ and $E_4$. 
Moreover, on the Schr\"{o}dinger hierarchy, the first equation is the cubic nonlinear Schr\"{o}dinger equation 
\begin{equation*} 
i \p_t u+ \p_x^2 u = \la_4 |u|^2 u  
\end{equation*}
and the second equation is \eqref{4NLS1}.

In this paper, we focus on the well-posedness of the 4NLS on the torus. 
There is a large literature on the well-posedness of the 4NLS in the real line. 
For instance, see \cite{HIT}, \cite{HJ1}, \cite{Se1}, \cite{Se2} and \cite{Se3}. 
Huo and Jia \cite{HJ2} proved the local well-posedness of \eqref{fm4NLS} in $H^s(\R)$ for $s>1/2$ 
by the Fourier restriction norm method, which was introduced by Bourgain in \cite{Bo1}. 
Recently, Chen, Lu and Wang \cite{CLW} studied \eqref{4NLS1} in the integrable case and 
proved the global well-posedness in the modulation space $M_{2, q}^{s} (\R)$ 
for $s \ge 1/2$ and $2 \le q < \infty$. 
In their proof, they employed the $U^p$-$V^p$ spaces and completely integrable structure. 
J. Adams \cite{Ad} considered all equations belonging to the Schr\"{o}dinger hierarchy and demonstrated the local well-posedness in the 
Fourier-Lebesgue space and the modulation space at almost optimal exponent. Furthermore, in $H^s(\R)$, he proved the global well-posedness at the critical exponent where the iteration method works. 
On the other hand, in the periodic setting, known results for the 4NLS are few because 
the linear part does not have any smoothing effects and that makes the problem more difficult. 
Segata \cite{Se4} proved the local well-posedness in $H^s(\T)$ with $s \ge 4$ 
for \eqref{fm4NLS} without any conditions on $\la_1, \dots, \la_6$ and $\mu$. 
In \cite{Se5}, Segata also obtained the global well-posedness in $H^2(\T)$ for \eqref{4NLS1} in the integrable case. 
His proof of these results is based on the modified energy method, which was introduced by Kwon \cite{Kwo}. 
We are especially interested in low regularity solutions and the non-integrable case. 
The main result in this paper is as below. 

\begin{thm} \label{thm_LWP}
Let $s \geq 1$ and $\la_5= \la_2+\la_4$ or $\la_5=0$. Then, for any $\varphi \in H^s (\mathbb{T})$, 
there exists $T=T(\| \varphi \|_{H^{s}}) >0$ such that there exists a unique solution $u \in C([-T,T]: H^{s} (\mathbb{T}))$ to 
\eqref{4NLS1} with \eqref{initial}. 
Moreover, the solution map $H^s(\mathbb{T}) \ni \varphi \mapsto u \in C([-T,T]: H^s (\mathbb{T}))$ is continuous.  
\end{thm}

\begin{rem} \label{rem_opt}
(i) We notice that 
\begin{align} \label{opt1}
& \la_2 \bar{u} (\p_x u)^2+ \la_3 u |\p_x u|^2+ \la_4 u^2 \p_x^2 \bar{u}+ \la_5 |u|^2 \p_x^2 u \notag  \\
&  =(\la_2- \la_5) \bar{u} (\p_x u)^2+ (\la_3-2 \la_4- \la_5) u |\p_x u|^2+ \la_4 \p_x( u^2 \p_x \bar{u})
 +\la_5 \p_x (|u|^2 \p_x u).  
\end{align}
Let $u \in C([-T, T] : H^s(\T))$ be a solution to \eqref{4NLS1}. 
When $s \ge 1$, $\bar{u} (\p_x u)^2$ and $u |\p_x u|^2$ are in $C([-T, T]: H^{s-2} (\T))$. 
Thus, \eqref{4NLS1} makes sense in $C([-T, T]: H^{s-4} (\T))$. 
On the other hand, when $s<1$, $\bar{u} (\p_x u)^2 $ and $u (\p_x u)^2$ cannot be defined even as the space-time distribution. 
Therefore, by \eqref{opt1}, the unconditional uniqueness in Theorem~\ref{thm_LWP} 
for $s=1$ is optimal except $\la_2=\la_5$ and $\la_3= 2 \la_4+ \la_5$.  \\
(ii) In a similar way to the proof of Theorem~\ref{thm_LWP}, under the same assumptions of Theorem~\ref{thm_LWP}, 
we obtain the unconditional well-posedness for \eqref{fm4NLS}. 
\end{rem}

Since \eqref{4NLS1} is $L^2$-subcritical in the sense of the Sobolev scale, by using some conserved quantities, 
we have the following global result as a corollary of Theorem \ref{thm_LWP}.
\begin{cor} \label{cor_GWP}
{\upshape (I)} (Hamiltonian case) 
Assume that $\la_3=2 \la_4$ and $\la_5=\la_2+ \la_4$. 
Then, for any $\vp \in H^2(\T)$, the solution obtained in Theorem~\ref{thm_LWP} can be extended to one on $t \in (- \infty, \infty)$. \\
{\upshape (II)} Assume that $A_1$ and $A_3$ hold, that is, $\la_3= 8 \la_4-2 \la_2$, $\la_5=\la_2+\la_4$ and $\la_1 =8 \la_2 \la_4/3$. 
Then, for any $\vp \in H^1(\T)$, the solution obtained in Theorem \ref{thm_LWP} can be extended to one on $t\in (-\infty,\infty)$.
\end{cor}

The main difficulty in proveing Theorem~\ref{thm_LWP} comes from derivative losses in the nonlinear terms. 
We can extract a kind of smoothing effect for the non-resonant parts in the nonlinear terms. 
Therefore, we need to eliminate the resonant parts with derivative losses in the nonlinear terms. 
We will cancel them by the conservation quantity $E_1$ and the changing variable $x$ to $x+c(t)$ below.
When $\la_5=\la_2 + \la_4$ or $\la_5=0$, it follows that  $\la_5 E_1(\vp) = \la_5 E_1(u) (t) $ 
for any solution $u \in C([-T, T]: H^1(\T))$ of \eqref{4NLS1} in the rigorous sense (see Lemma~\ref{lem_E1}). 
Thus, \eqref{4NLS1} can be rewritten into 
\begin{align}
& \p_t u -i \p_x^4 u + i \la_5 E_1(\vp) \p_x^2 u +(\la_3- 2 \la_2) E_2(u) \p_x u \notag \\ 
& \hspace{1cm} =J_1(u)+ J_2 (u)  +J_3 (u)+J_4(u) \label{4NLS2}
\end{align}
where 
 \begin{align*}
J_1(u)=&   -i \la_2 \bar{u} (\p_x u)^2 -i \la_3 u |\p_x u|^2- i \la_4 u^2 \p_x^2 \bar{u} - i \la_5 |u|^2 \p_x^2 u, \\
J_2(u)=&  i \la_5 E_1(u) \p_x^2 u+ (\la_3-2 \la_2) E_2 (u) \p_x u \\ 
& - i (2\la_4+\la_5 - \la_3 ) \int_{\T} |\p_x u|^2 dx \,  u, \\
J_3(u)= & i (2 \la_4+\la_5-\la_3) \int_{\T} |\p_x u|^2 dx \, u, 
\hspace{0.5cm} J_4(u)=  i \frac{3}{8} \la_1 |u|^4 u. 
\end{align*}
Note that $E_2(u)$ does not depend on $x$. Thus, by the change of variable
\begin{equation*}
 u(t,x) \mapsto u \Big(t, x+ \int_0^t ( \la_3- 2 \la_2) E_2(u) (t') \, dt'\Big), 
\end{equation*}
\eqref{4NLS2} is rewritten into 
\begin{align} \label{4NLS3}
\p_t u- i \p_x^4 u + i \la_5 E_1(\varphi) \p_x^2 u = J_1(u)+ J_2(u) +J_3(u)+J_4(u).  
\end{align}
We only need to show the unconditional local well-posedness of \eqref{4NLS3} with \eqref{initial}. 

In \eqref{4NLS3}, the nonlinear terms include the non-resonant parts or no derivative loss parts. 
We use the normal form reduction to recover derivative losses included in the non-resonant parts. 
It was used to recover derivative losses by Takaoka and Tsutsumi in \cite{TaTs} 
in the study of the well-posedness of the modified KdV equation and by Babin, Ilyin and Titi in \cite{BIT} in the study of 
the unconditional uniqueness of the KdV equation. 
For the studies of the unconditional uniqueness by the normal form reduction, 
see \cite{GKO}, \cite{Kn1}, \cite{Kn2}, \cite{Kn3}, \cite{Kn4}, \cite{Kn5}, \cite{KO}, \cite{KOY}, \cite{MoP} and \cite{MoY}. 
Roughly speaking, by Lemmas~\ref{Le1}, \ref{Le2} and \ref{Le3}, 
only one derivative can be recovered by applying the normal form reduction once. 
The non-resonant parts of $J_1(u)+J_2(u)$ (see the first term of \eqref{EE4}) have two derivative losses. 
Thus, we need to apply the normal form reduction twice. 
After we apply it once, quintic terms with one derivative loss appear. 
We need to remove the resonant parts of them, which is represented by the symbol of $M_{8, \vp}^{(5)}$. 
In Proposition~\ref{prop_res}, we compute the symmetrization of $M_{8, \vp}^{(5)}$ and show it has no derivative losses. 
Therefore, by applying the normal form reduction to the non-resonant parts of quintic terms,  
all derivative losses are recovered. 
When $s \ge 1$, all terms appearing by the normal form reduction can be estimated by fundamental tools such as 
the H\"{o}lder and the Young inequalities. 
We note that each nonlinear term appearing through the normal form reduction requires individual estimates, 
leading to numerous case distinctions. 
This is presumably due to the fact that the standard Fourier restriction norm method does not work for 
the non-resonant parts of \eqref{4NLS3}. 
This is the main idea in this paper. A similar idea was used for the fifth order KdV type equations in \cite{KTT} and 
for the fifth order mKdV type equations in \cite{KTT2}. 
Futhermore, we notice that the iteration argument fails to work for \eqref{4NLS3} when $s<1$ 
because the nonlinear terms of \eqref{4NLS3} have 
the resonant parts of high-high-high interaction (see the second term of \eqref{EE4}).   

Finally, we give some notations.
We write $k_{2i+1, 2i+2,\ldots,2j+1}$ to mean 
\begin{equation*}
\sum_{l=i}^j k_{2l+1}- \sum_{l=i}^{j-1} k_{2l+2}
\end{equation*}
for integers $i$ and $j$ satisfying $i < j$. 
$k_{\max}$ is denoted by $k_{\max}:= \max_{1 \le j \le N} \{ |k_j| \}$ and 
$\text{sec}_{1 \le j \le N} \{  |k_j|  \}$ is denoted by the second largest number among $|k_1|, \dots, |k_N|$. 
We will use $A\lesssim B$ to denote an estimate of the form $A \le CB$ for some positive constant $C$ and write $A \sim B$ 
to mean $A \lesssim B$ and $B \lesssim A$. 
If $a \in \R$, $a+$ will denote a number slightly greater than $a$.  
$\| \cdot \|_{L_T^{\infty} X }$ is denoted by $\| \cdot \|_{L_T^{\infty} X}:= \sup_{t \in [-T, T]} \| \cdot \|_{X}$ for a Banach space $X$. 
For $s \in \R$, $l_s^2$ is denoted by $l_s^2 := \{ f: \Z \to \C: \|  f \|_{l_s^2}:= \| \langle \cdot  \rangle^s f \|_{l^2} < \infty \}$.

\section*{Acknowledgement}
We would like to thank Professor Jun-ichi Segata and Professor Takato Uehara for their valuable comments.

\section{notations and preliminary lemmas}

In this section, we prepare some lemmas to prove main theorem. 
First, we give some notations. 
For a $2 \pi$-periodic function $f$ and a function $g$ on $\Z$, we define the Fourier transform and the inverse Fourier transform by 
\begin{equation*}
(\mathcal{F}_x  f)(k):= \ha{f}(k):= \int_{\T} e^{-ixk} f(x) \, dx, \hspace{0.5cm}
(\mathcal{F}^{-1} g) (x):= \sum_{k \in \Z} e^{ixk} g(k). 
\end{equation*} 
Then, we have 
\begin{equation*}
f= \mathcal{F}^{-1} (\mathcal{F}_x f), \hspace{0.5cm} 
\| f \|_{L^2}:= \Big( \int_{\T} |f(x)|^2 \, dx \Big)^{1/2}= \Big( \sum_{k \in \Z} |\ha{f} (k)|^2 \Big)^{1/2}=\|\ha{f}\|_{l^2}. 
\end{equation*}
For functions $f, g$ on $\Z$, 
\begin{equation*}
f \check{*} g (k):= \sum_{k=k_1-k_2} f(k_1) g(k_2)= \sum_{m \in \Z} f(k+m) g(m).
\end{equation*}
We give some estimates on the phase function $\Phi_{\vp}^{(2N+1)}$ defined as 
\begin{align*}
&\Phi_{\vp}^{(2N+1)} 
:= \phi_{\vp}(k_{1,2, \dots, 2N+1})- \Big\{\sum_{l=1}^{N+1} \phi_{\vp} (k_{2l-1}) - \sum_{l=1}^N \phi_{\vp} (k_{2l}) \Big\}  ,\\
&\phi_{\vp} (k):= k^4+  \la_5 E_1(\vp) k^2
\end{align*}
for $N \in \N$, which plays an important role to recover some derivatives when we estimate non-resonant parts of nonlinear terms.
A simple calculation yields that 
\begin{equation}
\Phi_{f}^{(3)}=  (k_1-k_2) (k_3-k_2) \big\{ (k_1-k_2)^2 + (k_3-k_2)^2 + 3(k_1+k_3)^2 + 2 \la_5 E_1(f)  \big\} \label{2eq2}
\end{equation}
for $f \in L^2(\T)$. We easily prove the following lemmas by the factorization formula \eqref{2eq2}.
\begin{lem} \label{Le1}
Assume that $f,g \in L^2(\T)$, $k_{\max}>16 \max\{1, |\la_5| E_1(f) , |\la_5| E_1(g) \}$ and $(k_1-k_2) (k_3-k_2) \neq 0$. 
Then, the following hold: \\
{\upshape (i)} If 
\begin{equation} \label{rel1}
|k_3| >16 \max\{ |k_1|, |k_2|  \}
\end{equation}
then it follows that 
\begin{align}
& |\Phi_f^{(3)}| \gtrsim |k_1-k_2| \, k_{\max}^3,  \label{ff1} \\
& \Big| \frac{1}{\Phi_{f}^{(3)}}- \frac{1}{\Phi_g^{(3)}} \Big|   \lesssim \frac{|\la_5| |E_1(f) -E_1(g) |}{ k_{\max}^2} 
\min \left\{  \frac{1}{|\Phi_f^{(3)}|}, \frac{1}{|\Phi_g^{(3)}|} \right\}.\label{ff2}
\end{align}
{\upshape (ii)} If either $|k_2| > 16 \max \{ |k_1|, |k_3| \}$ or 
\begin{equation*} 
32 |k_2| < \min\{ |k_1|, |k_3|  \}, \hspace{0.5cm} |k_1|/16 \le |k_3| \le 16 |k_1|
 \end{equation*}
holds, then it follows that \eqref{ff2} and 
\begin{equation}
|\Phi_f^{(3)}| \gtrsim k_{\max}^4. \label{ff3} 
\end{equation}
{\upshape (iii)} If  
\begin{equation*} 
32 |k_1| < \min\{ |k_2|, |k_3|  \}, \hspace{0.5cm} |k_2|/16 \le |k_3| \le 16 |k_2|
\end{equation*}
holds, then it follows that \eqref{ff2} and 
\begin{equation} \label{ff4}
|\Phi_f^{(3)}| \gtrsim |k_3-k_2| \, k_{max}^3.
\end{equation}
{\upshape (iv)} If $|k_1| \sim |k_2| \sim |k_3|$ holds, then it follows that \eqref{ff2} and 
\begin{equation} \label{ff5}
|\Phi_f^{(3)}| \gtrsim |k_1-k_2| |k_3-k_2| \, k_{\max}^2.
\end{equation}
\end{lem}
We need estimates similar to \eqref{ff1} and \eqref{ff3}--\eqref{ff5} for $\Phi_{f}^{(5)}$ to recover the derivative loss. 
However, no factorization formulas are known for $\Phi_f^{(5)}$ and 
the following proposition means that it is impossible to gain $k_{\max}^3$ by $\Phi_f^{(5)}$ under the assumption 
$k_1-k_2+k_3-k_4 \neq 0$ and $|k_5| \gg \max_{j=1,2,3,4} \{ |k_j| \}$.

\begin{prop} \label{prop_co}
For $f \in L^2 (\T)$ and $\la_5 \in \R$, there exists $(k_1, k_2, k_3, k_4, k_5) \in \Z^5$ such that 
\begin{align*}
& |k_5| \gg \max_{1\le j \le 4}\{ |k_j| \}, \hspace{0.5cm} k_1-k_2+k_3-k_4 \neq 0, \hspace{0.5cm} |k_5| \gg |\la_5 | E_1(f) \notag \\
& |\Phi_f^{(5)}| \ll |k_1-k_2+k_3-k_4| |k_5|^3. 
\end{align*}
\end{prop}
\begin{proof}
A simple calculation yields that 
\begin{align} \label{co1}
\Phi_{f}^{(5)}  =&  \Phi_f^{(3)} (k_{1,2,3}, k_4, k_5)+ \Phi_f^{(3)} (k_1, k_2, k_3) \notag \\
= & \{ \Phi_0^{(3)} (k_{1,2,3}, k_4, k_5)+ \Phi_0^{(3)} (k_1, k_2, k_3) \} \notag \\
& +2 \la_5 E_1(f) \{  (k_1-k_2-k_3+k_4) (k_{5}-k_4)+(k_1-k_2) (k_3-k_2)  \} \notag \\
=:& \psi_1(k_1, k_2, k_3, k_3, k_4, k_5) + \la_5 E_1(f) \psi_2(k_1, k_2, k_3, k_4,  k_5). 
\end{align}
Now, we take 
\begin{equation*}
k_5=n^9, \hspace{0.3cm} k_4=n^8, \hspace{0.3cm} k_3=n^{8}+n^3, \hspace{0.3cm}  k_2=n^{5}+n^{3}-1, \hspace{0.3cm}  k_1=n^5
\end{equation*}
for sufficiently large number $n$. Then, it follows that $k_1-k_2+k_3-k_4=1$ and 
\begin{align*}
& \Phi_0^{(3)} (k_{1,2,3}, k_4, k_5)= 4n^{27}-4n^{24} + R_1, \\ 
& \Phi_0^{(3)} (k_1, k_2, k_3) = -4n^{27}+ 4 n^{24} + R_2 
\end{align*}
where $|R_1| \sim n^{18}$ and $|R_2| \sim n^{22}$. Thus, we have 
\begin{equation} \label{co2}
|\psi_1(k_1, k_2, k_3, k_4, k_5)| \sim n^{22}. 
\end{equation}
By taking $n$ sufficiently large, we have $|k_5| \gg |\la_5 | E_1(f)$ and 
\begin{equation} \label{co3}
|\la_5| E_1(f) |\psi_2(k_1, k_2, k_3, k_4, k_5)| \ll n^{20}
\end{equation}
By \eqref{co1}--\eqref{co3}, we obtain $|\Phi_f^{(5)}| \sim n^{22} \ll n^{27} = |k_1-k_2+k_3-k_4| |k_5|^3 $. 
\end{proof}

Therefore, we need the slightly stronger assumption \eqref{rel3}, \eqref{rel4} or \eqref{rel5} instead of \eqref{rel1} in the following lemma. 
\begin{lem} \label{Le2}
Assume $f, g \in L^2(\T)$ and $k_{\max} > 16 \max\{1, |\la_5| E_1(f), |\la_5| E_1(g)  \}$. 
If either 
\begin{align}
& |k_3-k_4|> 16^2 |k_1-k_2| , \hspace{0.5cm} |k_5| > 14 \max_{j=1,2,3,4} \{ |k_j| \}, \label{rel3} \\
& |k_5|^{3/4} > 16^2 \max_{j=1,2,3,4} \{ |k_j| \}, \hspace{0.5cm} k_1-k_2+k_3-k_4 \neq 0  \label{rel4}
\end{align}
or 
\begin{equation} \label{rel5}
|k_5| > 16^2 \max_{j=1,2,3,4} \{ |k_j| \} > 16^3 \, \text{{\upshape sec}}_{j=1,2,3,4} \{ |k_j| \}
\end{equation}
holds, then it follows that 
\begin{align} 
& |\Phi_f^{(5)}| \gtrsim |k_1-k_2+k_3-k_4| |k_5|^3, \label{2eq7} \\
& \Big| \frac{1}{ \Phi_f^{(5)} } - \frac{1}{ \Phi_g^{(5)} }  \Big| \lesssim 
\frac{ |\la_5| | E_1(f) -E_1(g) |  }{k_{\max}} \min \Big\{ \frac{1}{ |\Phi_f^{(5)}| }, \frac{1}{ | \Phi_g^{(5)}  | }  \Big\} \label{2eq8}
\end{align} 
\end{lem}
\begin{proof}
A simple calculation yields that 
\begin{align} \label{L1}
\Phi_f^{(5)} & = \Phi_{f}^{(3)} (k_{1,2,3}, k_4, k_5) + \Phi_{f}^{(3)} (k_1, k_2, k_3) \notag \\
& = \Phi_{0}^{(3)} (k_{1,2,3}, k_4, k_5) + \Phi_{0}^{(3)} (k_1, k_2, k_3) 
+ \la_5 E_1(f) R_1^{(5)}(k_1, k_2, k_3, k_4, k_5)
\end{align}
where  
\begin{equation}  \label{Rem1}
R_1^{(5)}(k_1, k_2, k_3, k_4, k_5) = 2 \{(k_1-k_2+k_3-k_4) (k_5-k_4) + (k_1-k_2) (k_3-k_2)   \}.
\end{equation}
By \eqref{2eq2}, we have
\begin{align}  \label{exc}
& \Phi_0^{(3)} (k_{1,2,3}, k_4, k_5) =  (k_1-k_2+k_3-k_4) (k_5-k_4) \notag \\ 
& \hspace{2cm} \times \{ (k_1-k_2+k_3-k_4)^2 + (k_5-k_4)^2 + 3 (k_5+k_1-k_2-k_4)^2 \}, \notag \\
& \Phi_0^{(3)} (k_1, k_2, k_3) = (k_1-k_2) (k_3-k_2) \{ (k_1-k_2)^2 + (k_3-k_2)^2 + 3(k_1+k_3)^2  \}.
\end{align}
Firstly, we suppose that \eqref{rel3} holds. Then, we easily check that 
\begin{align}
& |\Phi_0^{(3)} (k_{1,2,3}, k_4, k_5) | > |k_1-k_2+k_3-k_4| |k_5|^3,  \notag \\
& |\Phi_0^{(3)} (k_3,  k_4, k_5) | \le 40 |k_1-k_2| \max_{j=1,2,3} \{ |k_j|^3 \} \le 16^{-3} |k_1-k_2+k_3-k_4| |k_5|^3, \notag \\
& |R_1^{(5)} (k_1,  k_2, k_3, k_4, k_5)| < 3 |k_1-k_2+k_3-k_4| |k_5|  \label{L21}
\end{align} 
and by $|k_5| > 16 \max \{1, |\la_5| E_1(f)  \}$
\begin{equation*}
|\la_5| E_1(f) |R_1^{(5)} (k_1, k_2, k_3, k_4, k_5) | <  16^{-1}  |k_1-k_2+k_3-k_4| |k_5|^3. 
\end{equation*}
Thus, by \eqref{L1}, we have 
\begin{equation} \label{L22}
|\Phi_f^{(5)}|> \frac{1}{2} |k_1-k_2+k_3-k_4| |k_5|^3. 
\end{equation}
Since 
\begin{equation} \label{L23}
\frac{1}{\Phi_f^{(5)}}- \frac{1}{\Phi_g^{(5)}}= \la_5 (E_1(g)-E_1(f)) \frac{ R_1^{(5)} (k_1, k_2, k_3, k_4, k_5) }{ \Phi_f^{(5)} \Phi_g^{(5)} }, 
\end{equation}
by \eqref{L21} and \eqref{L22}, we get \eqref{2eq8}. 

Secondly, we suppose that \eqref{rel4} holds. Then, we easily check that 
\begin{align}
& |\Phi_0^{(3)} (k_{1,2,3}, k_4, k_5) | > 2 |k_1-k_2+k_3-k_4| |k_5|^3,  \notag \\
& |\Phi_0^{(3)} (k_3,  k_4, k_5) | \le 80 \max_{j=1,2,3} \{ |k_j|^4 \} \le 16^{-6} |k_5|^3 \le 16^{-6} |k_1-k_2+k_3-k_4| |k_5|^3 , \notag \\
& |R_1^{(5)} (k_1,  k_2, k_3, k_4, k_5)| < 3 |k_1-k_2+k_3-k_4| |k_5|+ 16^{-3} |k_5|^2  \label{L31}
\end{align} 
and by $|k_5| > 16 \max \{1, |\la_5| E_1(f)  \}$
\begin{equation*}
|\la_5| E_1(f) |R_1^{(5)} (k_1, k_2, k_3, k_4, k_5) | <  16^{-1}  |k_1-k_2+k_3-k_4| |k_5|^3. 
\end{equation*}
Thus, by \eqref{L1}, we have 
\begin{equation} \label{L32}
|\Phi_f^{(5)}|>  |k_1-k_2+k_3-k_4| |k_5|^3. 
\end{equation}
By \eqref{L23}, \eqref{L31} and \eqref{L32}, we get \eqref{2eq8}. 

Finally, we suppose that \eqref{rel5} holds. 
By symmetry, we may assume that $|k_3|= \max_{j=1,2,3,4} \{ |k_j| \}$ or $|k_4|= \max_{j=1,2,3,4} \{ |k_j| \}$ holds.  
When $|k_3|=\max_{j=1,2,3,4} \{ |k_j| \} $, we easily check that 
\begin{align}
& |\Phi_0^{(3)} (k_{1,2,3}, k_4, k_5) | > 2 |k_3| |k_5|^3,  \hspace{0.5cm} 
|\Phi_0^{(3)} (k_3,  k_4, k_5) | \le 80 |k_3|^4 < 16^{-4} |k_3| |k_5|^3, \notag \\
& |R_1^{(5)} (k_1,  k_2, k_3, k_4, k_5)| < 4 |k_3| |k_5|  \label{L33}
\end{align} 
and by $|k_5| > 16 \max \{1, |\la_5| E_1(f)  \}$
\begin{equation*}
|\la_5| E_1(f) |R_1^{(5)} (k_1, k_2, k_3, k_4, k_5) | < 16^{-1}  |k_3| |k_5|^3. 
\end{equation*}
Thus, by \eqref{L1} and $|k_3|> \frac{4}{5} |k_1-k_2+k_3-k_4| $, we have 
\begin{equation} \label{L34}
|\Phi_f^{(5)}|> |k_1-k_2+k_3-k_4| |k_5|^3. 
\end{equation}
By \eqref{L23}, \eqref{L33}, \eqref{L34} and $|k_3|> \frac{4}{5} |k_1-k_2+k_3-k_4|$, we get \eqref{2eq8}. 
In a similar manner as above, we obtain \eqref{2eq7} and \eqref{2eq8} when $|k_4|= \max_{j=1,2,3,4} \{ |k_j| \}$. 
\end{proof}
\begin{lem} \label{Le3}
Assume that $f, g \in L^2(\T)$ and $k_{\max} > 16 \{ 1, |\la_5| E_1(f), |\la_5| E_1(g) \}$. 
Then, the following hold: \\ 
{\upshape (i)} If either   
\begin{equation} \label{rel21}
14 \max_{j=1,2,3} \{ |k_j| \} < \min\{ |k_4|, |k_5| , |k_4-k_5| \}, \hspace{0.3cm} |k_4|/16 \le |k_5| \le 16 |k_4|
\end{equation}
or 
\begin{equation} \label{rel22}
16^2 |k_1-k_2| < |k_4-k_5|, ~~
14 \max_{j=1,2,3} \{ |k_j| \}< \min\{ |k_4|, |k_5| \}, ~~|k_4|/16 \le |k_5| \le 16 |k_4|
\end{equation}
holds, then it follows that 
\begin{align} 
& |\Phi_f^{(5)} | \gtrsim |k_4-k_5| \, k_{\max}^3, \label{2eq9} \\
& \Big| \frac{1}{ \Phi_f^{(5)} }- \frac{1}{ \Phi_g^{(5)} }  \Big| \lesssim \frac{ |\la_5| |E_1(f) -E_1(g) |}{ k_{\max}^2 } 
\min \Big\{ \frac{1}{ |\Phi_f^{(5)} | }, \frac{1}{| \Phi_g^{(5)} | }  \Big\}.  \label{2eq10} 
\end{align}
{\upshape (ii)} If either 
\begin{equation} \label{rel23}
14^2 \max_{j=1,2,4,5} \{ |k_j| \} < |k_4|
\end{equation} 
or
\begin{equation} \label{rel24}
14 \max_{j=1,2,4} \{ |k_j| \} < \min\{ |k_3|, |k_5| \}, \hspace{0.5cm} |k_3|/16 \le |k_5| \le 16 |k_3|
\end{equation}
holds, then it follows that  \eqref{2eq10} and 
\begin{equation}
|\Phi_f^{(5)}| \gtrsim k_{\max}^4.  \label{C31} 
\end{equation}
\end{lem}
\begin{proof}
A direct computation yields that 
\begin{align} \label{L41}
\Phi_f^{(5)}=\Phi_0^{(3)} (k_{1,2,3}, k_4, k_5)+ \Phi_0^{(3)} (k_1, k_2,k_3)+ \la_5 E_1 (f) R_1^{(5)}(k_1, k_2, k_3, k_4, k_5)
\end{align}
where $R_1^{(5)}$ is defined by \eqref{Rem1}. 
We notice that \eqref{exc} holds by \eqref{2eq2}. \\
(i) First, we suppose that \eqref{rel21} holds. When $|k_4| \le |k_5|$, we easily check that $|k_5|=k_{\max}$ and 
\begin{align}
& |\Phi_0^{(3)} (k_{1,2,3}, k_4, k_5) |> \frac{6}{5} |k_4-k_5| |k_4| |k_5|^2, \label{L46} \\
& |\Phi_0^{(3)} (k_1,  k_2, k_3) | < 14^{-2} |k_4-k_5| |k_4| |k_5|^2, \notag \\
& |R_1^{(5)} (k_1k_2, k_3, k_4, k_5) | < 3 |k_4-k_5| |k_4| \label{L42}
\end{align} 
and by $|k_5| > 16 \max\{ 1, |\la_5| E_1(f) \}$
\begin{equation*}
|\la_5| E_1(f) |R_1^{(5)} (k_1, k_2, k_3, k_4, k_5) | < 8^{-2} |k_4-k_5| |k_4| |k_5|^2.  
\end{equation*}
Thus, by \eqref{L41} and $|k_5| \le 16 |k_4|$, we have 
\begin{equation} \label{L43}
|\Phi_f^{(5)}| > |k_4-k_5| |k_4| |k_5|^2 \ge \frac{1}{16} |k_4-k_5| |k_5|^3. 
\end{equation}
By \eqref{L23}, \eqref{L42} and \eqref{L43}, we get \eqref{2eq10}. 
In a similar manner as above, we obtain \eqref{2eq9}  and \eqref{2eq10} when $|k_5| \le |k_4|$. 

Next, we suppose that \eqref{rel22} holds. When $|k_4| \le |k_5|$, we easily check that \eqref{L46}, \eqref{L42} and 
\begin{equation*}
|\Phi_0^{(3)} (k_1,  k_2, k_3) | \le 40 |k_1-k_2| \max_{j=1,2,3} \{  |k_j|^3 \} <  14^{-3} |k_4-k_5| |k_4| |k_5|^2. 
\end{equation*}
Thus, we have \eqref{L43}. By \eqref{L23} \eqref{L41} and \eqref{L42}, we get \eqref{2eq10}. 
Similarly, we get \eqref{2eq9}  and \eqref{2eq10} when $|k_5| \le |k_4|$. \\
\noindent
(ii) First, we suppose that \eqref{rel24} holds. Then, we easily check that 
\begin{align}
&|\Phi_0^{(3)} (k_{1,2,3}, k_4, k_5) | > \frac{2}{3} |k_4|^4, \hspace{0.5cm} 
|\Phi_0^{(3)}(k_1, k_2,  k_3)| < 14^{-2} |k_4|^4, \notag \\
& |R_1^{(5)} (k_1, k_2, k_3, k_4, k_5) | < 3|k_4|^2 \label{L52}
\end{align}
and by $|k_4| > 16 \max \{ 1, |\la_5| E_1(f) \}$
\begin{equation*}
|\la_5| E_1(f) |R_1^{(5)} (k_1,k_2, k_3, k_4, k_5) | < 8^{-2} |k_4|^4. 
\end{equation*}
Thus, by \eqref{L41}, we have 
\begin{equation} \label{C312}
|\Phi_f^{(5)}|> \frac{1}{2} |k_4|^4. 
\end{equation}
By \eqref{L23}, \eqref{L52} and \eqref{C312}, we get \eqref{2eq10}.

Next, we suppose that \eqref{rel24} holds. By symmetry, we assume that $|k_3| \le |k_5|$ holds. 
Then, it follows that $|k_5|=k_{\max}$ and 
\begin{equation} \label{rel25}
14^2 \max_{j=1,2,4} \{ |k_j| \} <|k_3| \le |k_5| \le 16|k_3|. 
\end{equation}
A simple calculation yields that 
\begin{align} \label{L5}
\Phi_f^{(5)}= \Phi_{0}^{(3)} (k_3, k_2+k_4-k_1, k_5) - \Phi_{0}^{(3)} (k_2, k_1, k_4) + \la_5 E_1(f) R_2^{(5)}(k_1, k_2, k_3, k_4, k_5)
\end{align}
where  
\begin{equation*} 
R_2^{(5)}(k_1, k_2, k_3, k_4, k_5) = 2 \{(k_3+k_1-k_2-k_4) (k_5+k_1-k_2-k_4) -(k_2-k_1) (k_4-k_1)   \}.
\end{equation*}
By \eqref{2eq2}, 
\begin{align*}
& \Phi_0^{(3)} (k_3, k_2+k_4-k_1, k_5) =  (k_3+k_1-k_2-k_4) (k_5+k_1-k_2-k_4) \notag \\ 
& \hspace{2cm} \times \{ (k_3+k_1-k_2-k_4)^2 + (k_5+k_1-k_2-k_4)^2 + 3 (k_3+k_5)^2 \}, \notag \\
& \Phi_0^{(3)} (k_2, k_1, k_3) = (k_1-k_2) (k_4-k_1) \{ (k_2-k_1)^2 + (k_4-k_1)^2 + 3(k_2+k_4)^2  \}.
\end{align*}
By \eqref{rel25}, we easily check that 
\begin{align}
& |\Phi_0^{(3)} (k_3, k_2+k_4-k_1, k_5)  | > \frac{1}{3} |k_3| |k_5|^3, \notag \\
& |\Phi_0^{(3)} (k_2, k_1, k_4)  | \le 80 \max_{j=1,2,4} \{ |k_j|^4 \} < 14^{-2} |k_3| |k_5|^3, \notag \\
& |R_2^{(5)} (k_1, k_2, k_3, k_4, k_5) | < 4 |k_3| |k_5| \label{L51}
\end{align}
and by $|k_5| > 16 \max \{ 1, |\la_5| E_1(f) \}$
\begin{equation*}
|\la_5| E_1(f) |R_2^{(5)} (k_1,k_2, k_3, k_4, k_5) | < 8^{-2} |k_3| |k_5|^2. 
\end{equation*}
Thus,  by \eqref{L5} and $|k_3| \le 16|k_5|$ , we have
\begin{equation} \label{C311}
|\Phi_f^{(5)}| > \frac{1}{4} |k_3| |k_5|^3 \ge \frac{1}{64} |k_5|^4.  
\end{equation}
Since
\begin{equation*}
\frac{1}{\Phi_{f}^{(5)}}- \frac{1}{ \Phi_g^{(5)} } = \la_5 (E_1(g)-E_1(f)) \, 
\frac{ R_2^{(5)} (k_1, k_2,  k_3, k_4,  k_5) }{ \Phi_f^{(5)} \Phi_g^{(5)}},
\end{equation*}
by \eqref{L51} and \eqref{C311}, we get \eqref{2eq10}. 
\end{proof}

\begin{defn}
For $f \in L^2(\T)$ and a $(2N+1)$-multiplier $m^{(2N+1)} (k_1, \dots, k_{2N+1})$, we define  
\begin{align*}
& \La_f^{(2N+1)} (m^{(2N+1)}, \ha{v}_1, \dots, \ha{v}_{2N+1}) (t,k) \notag  \\
& \hspace{1cm}
:= \sum_{k=k_{1,\dots, 2N+1}} e^{-i t \Phi_{f}^{(2N+1)} } m^{(2N+1)} (k_1, \dots, k_{2N+1}) \prod_{1 \le l \le 2N+1}^{*} \ha{v}_l (k_l)
\end{align*}
where 
\begin{equation*}
\prod_{1 \le l \le 2N+1}^{*} \ha{v}_l (k_i):= \prod_{j=1}^{N+1} \ha{v}_{2j-1} (k_{2j-1}) \, \prod_{i=1}^N \overline{\ha{v}}_{2i} (k_{2i})
\end{equation*} 
for functions $\ha{v}_1, \dots, \ha{v}_{2N+1}$ on $\Z$. 
$\La_f^{(2N+1)} (m^{(2N+1)}, \ha{v}, \dots, \ha{v} )$ may simply be written as $\La_f^{(2N+1)} (m^{(2N+1)}, \ha{v}) $. 
\end{defn}

\begin{defn}
The symmetrization of a $(2N+1)$-multiplier $m^{(2N+1)}$ is the multiplier 
\begin{align*}
& [m^{(2N+1)} ]_{sym1}^{(2N+1)} (k_1, \dots, k_{2N+1}) \\ 
& \hspace{1cm} := \frac{1}{(N+1)!} 
\sum_{\sigma \in S_{N+1}} m^{(2N+1)} (k_{\sigma(1)}, k_2, k_{\sigma(3)}, \dots, k_{2N}, k_{ \sigma(2N+1) }), \\
& [m^{(2N+1)} ]_{sym}^{(2N+1)} (k_1, \dots, k_{2N+1}) \\ 
& \hspace{1cm} := \frac{1}{N! (N+1)!} 
\sum_{\sigma \in S_{N+1}} \sum_{\tau \in \ti{S}_N} m^{(2N+1)} 
(k_{\sigma(1)}, k_{\tau(2)}, k_{\sigma(3)}, \dots, k_{\tau(2N)}, k_{ \sigma(2N+1) }). 
\end{align*}
where $S_{N+1}$ is the group of all permutations on $\{ 1, 3 \dots, 2N+1 \}$ 
and $\ti{S}_N$ is the group of all permutations on $\{ 2,4,\dots, 2N\}$. 

We say that a $(2N+1)$-multiplier $m^{(2N+1)}$ is symmetric if $[ m^{(2N+1)} ]_{sym}^{(2N+1)} = m^{(2N+1)}$ holds. 
We also use $\ti{m}^{(2N+1)}:=[m^{(2N+1)}]_{sym}^{(2N+1)}$ for short. 
\end{defn}
\begin{defn}
We say an $N$-multiplier $m^{(N)}$ is symmetric with $(k_i, k_j)$ if
\begin{equation*}
m^{(N)} (k_1, \dots, k_{i}, \dots, k_{j} \dots, k_{N})=m^{(N)} (k_1, \dots, k_{j}, \dots, k_i, \dots, k_N)
\end{equation*}
holds for $(k_1, \dots, k_N) \in \Z^{N}$. 
\end{defn}
\begin{defn}
We define $(2N+3)$-extension operators of a $(2N+1)$-multiplier $m^{(2N+1)}$ and a $3$-multiplier $m^{(3)}$ by 
\begin{align*}
& [ m^{(2N+1)}  ]_{ext1,1}^{(2N+3)} (k_1, \dots, k_{2N+3}) \\
& \hspace{0.5cm} := m^{(2N+1)} (k_1, \dots, k_{2N-1}, k_{2N}, k_{2N+1, 2N+2, 2N+3}), \\
& [ m^{(2N+1)}  ]_{ext1,2}^{(2N+3)} (k_1, \dots, k_{2N+3}) \\ 
& \hspace{0.5cm} := m^{(2N+1)} (k_1, \dots, k_{2N-1}, k_{2N}-k_{2N+3}+k_{2N+2} , k_{2N+1}), \\
& [m^{(3)} ]_{ext2,1}^{(2N+3)} (k_1, \dots, k_{2N+3})  := m^{(3)} (k_{2N+1}, k_{2N+2}, k_{2N+3}), \\
& [m^{(3)} ]_{ext2,2}^{(2N+3)} (k_1, \dots, k_{2N+3}) := m^{(3)} (k_{2N}, k_{2N+3}, k_{2N+2})
\end{align*} 
and define $(2N+5)$-extension operators of a $(2N+1)$-multiplier $m^{(2N+1)}$ by 
\begin{align*}
& [ m^{(2N+1)}  ]_{ext1,1}^{(2N+5)} (k_1, \dots, k_{2N+5}) \\
& \hspace{0.5cm} := m^{(2N+1)} (k_1, \dots, k_{2N-1}, k_{2N}, k_{2N+1, 2N+2, 2N+3, 2N+4, 2N+5}), \\ 
& [ m^{(2N+1)}  ]_{ext1,2}^{(2N+5)} (k_1, \dots, k_{2N+5}) \\
&  \hspace{0.5cm} := m^{(2N+1)} (k_1, \dots, k_{2N-1}, k_{2N}-k_{2N+3}+k_{2N+2}-k_{2N+5}+k_{2N+4} , k_{2N+1}).
\end{align*}
\end{defn}
\begin{rem}\label{rem_sym}
For any multipliers $m_1, m_2$, it follows that
\begin{equation*}
[m_1 m_2]_{ext1,j}^{(2N+1)}= [m_1]_{ext1,j}^{(2N+1)} [m_2]_{ext1,j}^{(2N+1)}, \hspace{0.5cm} 
[m_1 m_2]_{ext2,j}^{(2N+1)}= [m_1]_{ext2,j}^{(2N+1)} [ m_2 ]_{ext2,j}^{(2N+1)}
\end{equation*}
for each $j=1,2$. 
\end{rem}

For $L>0$, we define some multipliers to restrict summation regions in the Fourier space as follows:
\begin{align*}
& \chi_{\le L}^{(2N+1)}:=
\begin{cases}
\dis 1, \hspace{0.2cm} \text{when} \hspace{0.2cm} \max_{1 \le j \le 2N+1}  \{  |k_j| \} \le L \\
0, \hspace{0.2cm} \text{otherwise}
\end{cases}, \\
& \chi_{> L}^{(2N+1)}:=
\begin{cases}
\dis 1, \hspace{0.2cm} \text{when} \hspace{0.2cm} \max_{1 \le j \le 2N+1}  \{  |k_j| \} > L \\
0, \hspace{0.2cm} \text{otherwise}
\end{cases}, \\
& \chi_{H1,1}^{(2N+1)}:=
\begin{cases}
\dis 1, \hspace{0.2cm} \text{when} \hspace{0.2cm} 16^{N} \max_{1 \le l \le 2N} \{ |k_l|   \} < |k_{2N+1}| \\
0, \hspace{0.2cm} \text{otherwise}
\end{cases}, 
\end{align*}
with $N=1,2$ and 
\begin{align*}
& \chi_{H1,2}^{(3)}:=
\begin{cases}
\dis 1, \hspace{0.2cm} \text{when} \hspace{0.2cm} 16 \max\{ |k_1|, |k_3|  \} < |k_2| \\
0, \hspace{0.2cm} \text{otherwise}
\end{cases}, \\
& \chi_{H2,1}^{(3)}:=
\begin{cases}
\dis 1, \hspace{0.2cm} \text{when} \hspace{0.2cm} 32 |k_1| < \min\{ |k_2-k_3| , |k_2|, |k_3|\}, 
\hspace{0.3cm} |k_2|/16 \le |k_3| \le 16 |k_2| \\
0, \hspace{0.2cm} \text{otherwise}
\end{cases}, \\
& \chi_{H2,2}^{(3)}:=
\begin{cases}
\dis 1, \hspace{0.2cm} \text{when} \hspace{0.2cm} |k_2-k_3| \le 32 |k_1| < \min\{ |k_2|, |k_3| \}, 
\hspace{0.3cm} |k_2|/16 \le |k_3| \le 16 |k_2| \\
0, \hspace{0.2cm} \text{otherwise}
\end{cases}, \\
& \chi_{H2,3}^{(3)}:=
\begin{cases}
\dis 1, \hspace{0.2cm} \text{when} \hspace{0.2cm} 32 |k_2| < \min\{ |k_1|, |k_3|\}, 
\hspace{0.3cm} |k_1|/16 \le |k_3| \le 16 |k_1| \\
0, \hspace{0.2cm} \text{otherwise}
\end{cases}, \\
& \chi_{H3}^{(3)}:=1-([2 \chi_{H1,1}^{(3)} ]_{sym1}^{(3)}+ \chi_{H1,2}^{(3)} + [ 2 \chi_{H2,1}^{(3)} ]_{sym1}^{(3)}
+ [2 \chi_{H2,2}^{(3)}]_{sym1}^{(3)}+ \chi_{H2,3}^{(3)}). 
\end{align*}
Note that $|k_1| \sim |k_2| \sim |k_3|$ for $(k_1, k_2, k_3) \in \supp \chi_{H3}^{(3)}$. 
We put
\begin{align*}
\chi_{NR1}^{(3)}:=
\begin{cases}
1, \hspace{0.2cm} \text{when} \hspace{0.2cm} (k_1-k_2) (k_3-k_2) \neq 0 \\
0, \hspace{0.2cm} \text{otherwise}
\end{cases},
\end{align*}
$\chi_{NR1}^{(5)}:=[\chi_{NR1}^{(3)}]_{ext1,1}^{(5)} [ \chi_{NR1}^{(3)} ]_{ext2,1}^{(5)}$ and 
\begin{align*}
& \chi_{R1}^{(2N+1)}:=
\begin{cases}
\dis 1, \hspace{0.2cm} \text{when} \hspace{0.2cm} k_1-k_2+ \dots +k_{2N-1}-k_{2N}= 0 \\
0, \hspace{0.2cm} \text{otherwise}
\end{cases}
\end{align*}
with $N=1,2$ and 
\begin{align*}
& \chi_{R2}^{(3)}:=
\begin{cases}
\dis 1, \hspace{0.2cm} \text{when} \hspace{0.2cm} k_1=k_2=k_3 \\
0, \hspace{0.2cm} \text{otherwise}
\end{cases}, \\
& \chi_{R3}^{(5)}:=
\begin{cases}
\dis 1, \hspace{0.2cm} \text{when} \hspace{0.2cm}  |k_5|^{3/4} \le 16^2 \max_{1 \le j \le 4} \{ |k_j| \}, \hspace{0.3cm} 
\max_{1 \le j \le 4} \{ |k_j| \} \le 16  \text{sec}_{1 \le j \le 4} \{ |k_j|  \}, \\
\hspace{2cm} |k_3-k_4| \le 16^2 |k_1-k_2| \\
0, \hspace{0.2cm} \text{otherwise}
\end{cases}. 
\end{align*}
We define 
\begin{align*} 
& \chi_{A1}^{(5)}:=
\begin{cases}
1, \hspace{0.2cm} \text{when} \hspace{0.2cm}  8^3 \max \{ |k_1|, |k_2|  \} < |k_{3,4,5}| \\
0, \hspace{0.2cm} \text{otherwise}
\end{cases}, \\
& \chi_{A2}^{(5)}:=
\begin{cases}
1, \hspace{0.2cm} \text{when} \hspace{0.2cm} 16^2 | k_1-k_2 |< |k_3-k_4| \\
0, \hspace{0.2cm} \text{otherwise}
\end{cases}, \\
& \chi_{A3}^{(5)}:=
\begin{cases}
1, \hspace{0.2cm} \text{when} \hspace{0.2cm} 16^2 |k_1-k_2| < |k_4-k_5|  \\
0, \hspace{0.2cm} \text{otherwise}
\end{cases}.  
\end{align*}
Moreover, we define
\begin{align*}
& \chi_{NR(1,1)}^{(5)}:= [ \chi_{H1,1}^{(3)}]_{ext1,1}^{(5)} [ \chi_{H1,1}^{(3)} ]_{ext2,1}^{(5)}, 
\hspace{0.3cm}
\chi_{NR(1,2)}^{(5)}:= \chi_{H1,1}^{(3)} (k_{3,4,5}, k_2, k_1) \chi_{H1,1}^{(3)} (k_3,k_4, k_5), \\
& \chi_{NR(2,1)}^{(5)}:= [ \chi_{H1,1}^{(3)}]_{ext1,1}^{(5)} [ \chi_{H1,2}^{(3)} ]_{ext2,1}^{(5)}, 
\hspace{0.3cm}
\chi_{NR(2,2)}^{(5)}:= \chi_{H1,1}^{(3)} (k_{3,4,5}, k_2, k_1) \chi_{H1,2}^{(3)} (k_3,k_4, k_5), \\
& \chi_{NR(3,1)}^{(5)}:= [ \chi_{H1,1}^{(3)}]_{ext1,1}^{(5)} [ \chi_{H2,1}^{(3)} ]_{ext2,1}^{(5)}, 
 \hspace{0.3cm}
\chi_{NR(3,2)}^{(5)}:= \chi_{H1,1}^{(3)} (k_{3,4,5}, k_2, k_1) \chi_{H2,1}^{(3)} (k_3, k_4,  k_5), \\
& \chi_{NR(4,1)}^{(5)}:= [ \chi_{H1,1}^{(3)}]_{ext1,1}^{(5)} [ \chi_{H2,2}^{(3)} ]_{ext2,1}^{(5)}, 
\hspace{0.3cm}
\chi_{NR(4,2)}^{(5)}:= \chi_{H1,1}^{(3)} (k_{3,4,5}, k_2, k_1) \chi_{H2,2}^{(3)} (k_3, k_4,  k_5), \\
& \chi_{NR(5,1)}^{(5)}:=[ \chi_{H1,1}^{(3)}]_{ext1,1}^{(5)} [ \chi_{H2,3}^{(3)} ]_{ext2,1}^{(5)}, 
\hspace{0.3cm}
\chi_{NR(5,2)}^{(5)}:= \chi_{H1,1}^{(3)} (k_{3,4,5}, k_2, k_1) \chi_{H2,3}^{(3)} (k_3,k_4,  k_5).
\end{align*}

\begin{lem} \label{Le6} 
Let $3$-multipliers $m_1^{(3)}$ and $m_2^{(3)}$ be symmetric. 
Then, it follows that 
\begin{align}
&\big[ [ m_1^{(3)} [2 \chi_{H1,1}^{(3)}]_{sym1}^{(3)} ]_{ext1,1}^{(5)} 
\big[m_2^{(3)} \big([2\chi_{H1,1}^{(3)}]_{sym1}^{(3)} + [2 \chi_{H2,1}^{(3)}]_{sym1}^{(3)} + [2 \chi_{H2,2}^{(3)} ]_{sym1}^{(3)} \big) \big]_{ext2,1}^{(5)}  \big]_{sym}^{(5)}  \notag \\
 & =  \sum_{j=1,3,4} \big\{ \big[ [ m_1^{(3)} ]_{ext1,1}^{(5)} [m_2^{(3)}]_{ext2,1}^{(5)} 2 \chi_{NR(j ,1)}^{(5)} \big]_{sym}^{(5)} 
+  \big[ [ m_1^{(3)} ]_{ext1,1}^{(5)} [m_2^{(3)}]_{ext2,1}^{(5)}  2 \chi_{NR(j, ,2)}^{(5)} \big]_{sym}^{(5)} \big\} 
\label{le211} 
\end{align}
and
\begin{align}
&\big[ [ m_1^{(3)} [2 \chi_{H1,1}^{(3)}]_{sym1}^{(3)} ]_{ext1,1}^{(5)} 
\big[m_2^{(3)} \big( \chi_{H1,2}^{(3)}+ \chi_{H2,3}^{(3)} \big) \big]_{ext2,1}^{(5)}  \big]_{sym}^{(5)}  \notag \\
 & =  \sum_{j=2,5} \big\{  \big[ [ m_1^{(3)} ]_{ext1,1}^{(5)} [m_2^{(3)}]_{ext2,1}^{(5)} \chi_{NR(j ,1)}^{(5)} \big]_{sym}^{(5)} 
+  \big[ [ m_1^{(3)} ]_{ext1,1}^{(5)} [m_2^{(3)}]_{ext2,1}^{(5)}  \chi_{NR(j,2)}^{(5)} \big]_{sym}^{(5)} \big\}.
\label{le212} 
\end{align}
\end{lem}

\begin{proof}
We prove (\ref{le211}). It suffices to show 
\begin{align}
& \big[ [ m_1^{(3)} [ 2\chi_{H1,1}^{(3)} ]_{sym1}^{(3)} ]_{ext1,1}^{(5)}  
[ m_2^{(3)} [ 2 \chi_{H1,1}^{(3)} ]_{sym1}^{(3)} ]_{ext2,1}^{(3)} \big]_{sym}^{(5)} \notag \\
& =\big[ [m_1^{(3)}]_{ext1,1}^{(5)} [m_2^{(3)}]_{ext2,1}^{(5)} 2 \chi_{NR(1,1)}^{(5)} \big]_{sym}^{(5)}
+\big[ [m_1^{(3)}]_{ext1,1}^{(5)} [m_2^{(3)}]_{ext2,1}^{(5)} 2 \chi_{NR(1,2)}^{(5)} \big]_{sym}^{(5)}, \label{le221} \\
& \big[ [ m_1^{(3)} [ 2\chi_{H1,1}^{(3)} ]_{sym1}^{(3)} ]_{ext1,1}^{(5)}  
[ m_2^{(3)} [ 2 \chi_{H2,1}^{(3)} ]_{sym1}^{(3)} ]_{ext2,1}^{(3)} \big]_{sym}^{(5)} \notag \\
& = \big[ [m_1^{(3)}]_{ext1,1}^{(5)} [m_2^{(3)}]_{ext2,1}^{(5)} 2 \chi_{NR(3,1)}^{(5)} \big]_{sym}^{(5)}
+ \big[ [m_1^{(3)}]_{ext1,1}^{(5)} [m_2^{(3)}]_{ext2,1}^{(5)} 2 \chi_{NR(3,2)}^{(5)} \big]_{sym}^{(5)} \label{le222}
\end{align}
and
\begin{align}
& \big[ [ m_1^{(3)} [ 2\chi_{H1,1}^{(3)} ]_{sym1}^{(3)} ]_{ext1,1}^{(5)}  
[ m_2^{(3)} [ 2 \chi_{H2,2}^{(3)} ]_{sym1}^{(3)} ]_{ext2,1}^{(3)} \big]_{sym}^{(5)} \notag \\
& =\big[ [m_1^{(3)}]_{ext1,1}^{(5)} [m_2^{(3)}]_{ext2,1}^{(5)} 2 \chi_{NR(4,1)}^{(5)} \big]_{sym}^{(5)}
+ \big[ [m_1^{(3)}]_{ext1,1}^{(5)} [m_2^{(3)}]_{ext2,1}^{(5)} 2 \chi_{NR(4,2)}^{(5)} \big]_{sym}^{(5)}. \label{le223}
\end{align}
We only prove \eqref{le221} since \eqref{le222} and \eqref{le223} follow in a similar manner. 
Put $M:= [ m_1^{(3)} [ 2\chi_{H1,1}^{(3)} ]_{sym1}^{(3)} ]_{ext1,1}^{(5)} [ m_2^{(3)} [ 2\chi_{H1,1}^{(3)} ]_{sym1}^{(3)}  ]_{ext2,1}^{(5)} $. 
By Remark~\ref{rem_sym}, it follows that 
\begin{align*}
M = & [m_1^{(3)}]_{ext1,1}^{(5)} [ m_2^{(3)}  ]_{ext2,1}^{(5)} [ [2 \chi_{H1,1}^{(3)}]_{sym1}^{(3)} ]_{ext1,1}^{(5)}
[ [2 \chi_{H1,1}^{(3)} ]_{sym1}^{(3)} ]_{ext2,1}^{(5)} \\
=& [m_1^{(3)}]_{ext1,1}^{(5)} [ m_2^{(3)}  ]_{ext2,1}^{(5)}
 \big( \chi_{H1,1}^{(3)} (k_1, k_2, k_{3,4,5})+ \chi_{H1,1}^{(3)} (k_{3,4,5}, k_2,  k_1)  \big) \\
& \hspace{0.5cm} \times \big( \chi_{H1,1}^{(3)} (k_3, k_4, k_5)+ \chi_{H1,1}^{(3)} (k_5, k_4, k_3)   \big).
\end{align*}
Since $m_2^{(3)}$ is symmetric, $[ m_1^{(3)} ]_{ext1,1}^{(5)} [ m_2^{(3)} ]_{ext2,1}^{(5)} $ is symmetric with $(k_3, k_5)$. 
Thus, $[M]_{sym}^{(5)}$ is equal to 
\begin{align*}
&  \big[ [m_1^{(3)}]_{ext1,1}^{(5)} [ m_2^{(3)}  ]_{ext2,1}^{(5)} (\chi_{H1,1}^{(3)} (k_1,  k_2, k_{3,,4,5})   + \chi_{H1,1}^{(3)} (k_{3,4,5}, k_2, k_1) ) 
2 \chi_{H1,1}^{(3)}(k_3, k_4, k_5)   \big]_{sym}^{(5)} \\
& = \big[ [m_1^{(3)} ]_{ext1,1}^{(5)} [m_2^{(3)}]_{ext2,1}^{(5)} 2 \chi_{NR(1,1)}^{(5)} \big]_{sym}^{(5)}
+ \big[ [m_1^{(3)}]_{ext1,1}^{(5)} [ m_2^{(3)}  ]_{ext2,1}^{(5)} 2 \chi_{NR(1,2)}^{(5)} \big]_{sym}^{(5)}
\end{align*}
which implies \eqref{le221}. 

Similarly, we obtain \eqref{le212}.
\end{proof}
Finally, we show variants of Sobolev's inequalities.
\begin{lem} \label{lem_nl2}
Let $s \ge 1$. Then, it follows that 
\begin{align} 
& \Big\| \sum_{k=k_{1,\cdots, 2N+1}} \langle k_{\max}  \rangle^{2} \prod_{l=1}^{2N+1} |\ha{v}_l(k_l)|  \Big\|_{l_{s-2}^2  } 
\lesssim \prod_{l=1}^{2N+1} \| v_l \|_{H^s}, \label{nl11} \\
& \Big\|   \sum_{k=k_{1, \dots, 2N+1}} \langle k_{1, \dots, 2N+1} \rangle^{-1} \langle k_{\max} \rangle^{-1} 
\prod_{l=1}^{2N+1} | \ha{v}_l (k_l) | \Big\|_{l_s^2} 
\lesssim \| v_m \|_{H^{s-2}} \prod_{l \in \{1, \dots, 2N+1 \} \setminus \{ m \}} \| v_l \|_{H^s} \label{nl21}
\end{align}
for any $m \in \{ 1, \dots, 2N+1 \}$.
\end{lem}
\begin{proof}
Firstly, we prove \eqref{nl11}. 
By symmetry, we may assume that $|k_1| \le |k_3| \le \dots \le |k_{2N+1}|$ and 
$|k_2| \le |k_4| \le \dots \le |k_{2N}|$. Then, $k_{\max}=|k_{2N+1}|$ or $k_{\max}= |k_{2N}|$ holds. 

First, we suppose that $|k_{1,\dots ,2N+1}| \sim k_{\max} $ holds.  
When $k_{\max}=|k_{2N+1}|$, 
by the Young inequality and the Schwarz inequality, the left hand side of \eqref{nl11} is bounded by
\begin{align*}
& \big\| \{ (|\ha{v}_1|*|\ha{v}_3| * \dots * |\ha{v}_{2N-1}| ) \check{*} ( |\ha{v}_2|* |\ha{v}_4| * \dots * |\ha{v}_{2N}| ) \} 
* \langle \cdot \rangle^s |\ha{v}_{2N+1}|   \big\|_{l^2} \\
& \lesssim \prod_{l=1}^{2N} \| \ha{v}_l \|_{l^1} \|  \langle \cdot \rangle^s |\ha{v}_{2N+1}| \|_{l^2} 
\lesssim \prod_{l=1}^{2N} \| v_l \|_{H^{1/2+}} \| v_{2N+1}  \|_{H^s}.
\end{align*}  
In a similar manner as above, we have \eqref{nl11} when $k_{\max} = |k_{2N}|$. 

Next, we suppose that $|k_{1,\dots, 2N+1}|  \ll k_{\max}$ holds. 
Then, either $k_{max} \sim |k_{2N+1}| \sim |k_{2N}|$, $k_{\max} \sim |k_{2N+1}| \sim |k_{2N-1}|$ or 
$k_{\max} \sim |k_{2N}| \sim |k_{2N-2}|$ holds. 
When $k_{\max} \sim |k_{2N+1}| \sim |k_{2N}|$, it follows that 
\begin{align*}
\langle k_{1, \dots, 2N+1}  \rangle^{s-2} \langle k_{\max}  \rangle^2 
& \sim \langle k_{1,\dots, 2N+1} \rangle^{s-2} \langle k_{2N} \rangle^s \langle k_{2N+1} \rangle^s \langle k_{\max} \rangle^{-2s+2} \\
& \lesssim  \langle k_{1, \dots, 2N+1} \rangle^{-s} \langle k_{2N} \rangle^s \langle k_{2N+1} \rangle^s
\end{align*}
Thus, by the Young and the H\"{o}lder inequalities, the left hand side of \eqref{nl11} is bounded by
\begin{align*} 
& \big\| \langle \cdot \rangle^{-s} \big\{ ( |\ha{v}_1|* \dots * |\ha{v}_{2N-1}|* \langle \cdot  \rangle^s |\ha{v}_{2N+1}|) 
\check{*} ( |\ha{v}_2|* \dots * |\ha{v}_{2N-2}|* \langle \cdot \rangle^s |\ha{v}_{2N}| ) \big\} \big\|_{l^2} \\
& \lesssim  \big\| |\ha{v}_1|* \dots * |\ha{v}_{2N-1}|* \langle \cdot \rangle^{s} |\ha{v}_{2N+1}|  \big\|_{l^2} \,
\big\| |\ha{v}_2| * \dots * |\ha{v}_{2N-2}|* \langle \cdot \rangle^s | \ha{v}_{2N}| \big\|_{l^2} \\
& \lesssim \prod_{l=1}^{2N-1} \| v_l \|_{H^{1/2+}} \| v_{2N} \|_{H^s} \| v_{2N+1}  \|_{H^s}. 
\end{align*}
Similarly, we have \eqref{nl11} if either $k_{\max} \sim |k_{2N+1}| \sim |k_{2N-1}|$ or $k_{\max} \sim |k_{2N}| \sim |k_{2N-2}|$ holds. 
Therefore, we obtain \eqref{nl11}. 

Secondly, we prove \eqref{nl21}. 
By $\langle k_{1,\dots, 2N+1} \rangle^{-1} \langle k_{\max} \rangle^{-1} \lesssim \langle k_{1,\dots, 2N+1}  \rangle^{-2} 
\langle k_{\max} \rangle^2 \langle k_m  \rangle^{-2}$ and \eqref{nl11}, 
the left hand side of \eqref{nl21} is bounded by 
\begin{align*}
& \sum_{m=1}^N \Big\| \sum_{k=k_{1,\dots, 2N+1}} \langle k_{\max} \rangle^{2} \big( \langle k_m \rangle^{-2} |\ha{v}_{m} (k_m) | \big)
\prod_{l \in \{ 1, \dots, 2N+1 \} \setminus \{ m \}} |\ha{v}_l(k_l)| \Big\|_{l_{s-2}^2} \\
& \lesssim \| \langle \cdot \rangle^{-2} |\ha{v}_m|  \|_{l_s^2} \prod_{l \in \{ 1, \dots, 2N+1  \} \setminus \{ m \}} \| v_l \|_{H^s} 
\lesssim \| v_m \|_{H^{s-2}} \prod_{l \in \{1, \dots, 2N+1 \} \setminus \{ m \}} \| v_l \|_{H^s}. 
\end{align*}
\end{proof}
\begin{lem}\label{lem_go}
Let $f, g \in L^2(\T)$ and a $(2N+1)$-multiplier $m^{(2N+1)}$ satisfy
\begin{equation*} 
\Big\| \sum_{k=k_{1,\ldots,2N+1}} |m^{(2N+1)}(k_1,\ldots,k_{2N+1})| \prod_{l=1}^{2N+1} |\ha{v}_l(t,k_l) | \Big\|_{L^\infty_Tl^2_s} \le C_0 \prod_{l=1}^{2N+1} \|v_l\|_{L^\infty_TH^{s_l}}
\end{equation*}
for any $v_l \in C([-T,T];H^{s_l}(\T))$ with $l=1,2, \ldots, 2N+1$.
Then, for any 
$v_l \in C([-T,T];H^{s_l}(\T))$ with $l=1,2,\ldots,2N+1$, it follows that
\begin{equation*} 
\Big\| 
\Lambda_{f}^{(2N+1)} ( m^{(2N+1)} , \ha{v}_1,\ldots,\ha{v}_{2N+1} ) 
-\Lambda_{g}^{(2N+1)} ( m^{(2N+1)} , \ha{v}_1,\ldots,\ha{v}_{2N+1} )\Big\|_{L^\infty_Tl^2_s} \le C_*
\end{equation*}
where $C_*=C_*(C_0,v_1,\ldots,v_{2N+1}, s_1,\ldots,s_{2N+1} ,|E_1(f)-E_1(g)|,|\la_5|,T) \ge 0$ 
and $C_*\to 0$ when $|E_1(f)-E_1(g)|\to 0$.
\end{lem}
\begin{proof}
By a slight modification of the proof of Lemma 2.11 in \cite{KTT}, we show this lemma.  
\end{proof}

\section{the normal form reduction}

Our aim in this section is to remove the derivative losses in the right-hand side of \eqref{4NLS3} by the normal form reduction, 
that is the differentiation by parts.
We  put
\EQQ{
&  q_1^{(3)}:= \la_2 k_1 k_3- \frac{\la_3}{2} (k_1+k_3) k_2 + \la_4 k_2^2 + \frac{\la_5}{2} (k_1^2+k_3^2), \\
& q_2^{(3)}:= (2 \la_4+\la_5-\la_3) \, [k_1^2 \, \chi_{R1}^{(3)}]_{sym1}^{(3)}, \\
& Q_1^{(3)}:= q_1^{(3)} \chi_{NR1}^{(3)}, \hspace{0.5cm} Q_2^{(3)}:= q_2^{(3)}- q_1^{(3)} \chi_{R2}^{(3)}. 
} 
We notice that all multipliers defined above are symmetric.  

The main proposition in this section is as below: 

\begin{prop} \label{prop_NF2}
Let $s \ge 1$, $\vp \in L^2(\T)$, $L \gg \max \{ 1, |\la_5| E_1(\vp) \}$, $T>0$ and 
$ u \in C([-T, T]:H^s(\T))$ be a solution of \eqref{4NLS3}. 
Then $\ha{v}(t,k):= e^{-it \phi_{\vp}(k)} \ha{u} (t,k) $ satisfies the following equation for each $k \in \Z$:
\begin{align} \label{NF21}
\p_t (\ha{v}(t,k)+ F_{\vp, L} (\ha{v}) (t,k)) = G_{\vp, L} (\ha{v}) (t, k), 
\end{align} 
where
\begin{align*}
F_{\vp, L} (\ha{v}) (t, k) :=
\sum_{j=1}^6 \La_{\vp}^{(3)} \big(  \ti{L}_{j, \vp}^{(3)} \chi_{>L}^{(3)}, \ha{v} (t)  \big) (t,k)
+ \sum_{j=1}^6 \La_{\vp}^{(5)} \big( \ti{L}_{j, \vp}^{(5)} \chi_{>L}^{(5)}, \ha{v} (t)  \big) (t, k)
\end{align*}
and
\begin{align*}
& G_{\vp, L} (\ha{v})(t, k) \\
& \hspace{0.3cm}
:= \sum_{j=1}^6 \La_{\vp}^{(3)} \big( i \ti{L}_{i, \vp}^{(3)} \Phi_{\vp}^{(3)} \chi_{\le L}^{(3)}, \ha{v}(t)   \big)(t, k)
+\sum_{j=1}^6 \La_{\vp}^{(5)} \big( i \ti{L}_{i, \vp}^{(5)} \Phi_{\vp}^{(5)} \chi_{\le L}^{(5)}, \ha{v}(t)   \big)(t, k) \\
& \hspace{0.3cm}
+ \La_{\vp}^{(3)} \big( i \ti{M}_{1,\vp}^{(3)}, \ha{v}(t) \big) (t, k)
 + \sum_{j=1}^{19} \La_{\vp}^{(5)} \big( i \ti{M}_{j,\vp}^{(5)}, \ha{v}(t) \big) (t, k) \\
& \hspace{0.3cm} + \sum_{j=1}^6 \La_{\vp}^{(7)} \big( i \ti{M}_{j,\vp}^{(7)}, \ha{v}(t) \big) (t, k) 
+ \sum_{j=1}^2 \La_{\vp}^{(9)} \big( i \ti{M}_{j, \vp}^{(9)}, \ha{v} (t)  \big) (t,k). 
\end{align*}
{\upshape (i)} The multipliers $\{L_{j, \vp}^{(3)}\}_{j=1}^6$ and $\{L_{j, \vp}^{(5)}\}_{j=1}^{6}$ are defined as below:
\begin{align*}
L_{1, \vp}^{(3)}  & := Q_1^{(3)} \, 2 \chi_{H1,1}^{(2)}/\Phi_{\vp}^{(3)}, \hspace{0.5cm} 
L_{2, \vp}^{(3)}  := Q_1^{(3)} \, \chi_{H1,2}^{(3)}/\Phi_{\vp}^{(3)}, \\ 
L_{3, \vp}^{(3)}  & := Q_1^{(3)} \, 2\chi_{H2,1}^{(3)}/\Phi_{\vp}^{(3)},
\hspace{0.5cm} L_{4, \vp}^{(3)}  := Q_1^{(3)} \, 2\chi_{H2,2}^{(3)}  /\Phi_{\vp}^{(3)}, \\
L_{5,  \vp}^{(3)} & := Q_1^{(3)} \chi_{H2,3}^{(3)}/ \Phi_{\vp}^{(3)}, 
\hspace{0.5cm} L_{6, \vp}^{(3)}:= Q_1^{(3)} \chi_{H3}^{(3)} /\Phi_{\vp}^{(3)}
\end{align*}
and 
\begin{align*}
L_{1, \vp}^{(5)}  &:= 4 \Big[ \frac{q_1^{(3)}}{ \Phi_0^{(3)} } \Big]_{ext1,1}^{(5)} 
[ q_1^{(3)} ]_{ext2,1}^{(5)} \, \chi_{NR1}^{(5)} \chi_{H1,1}^{(5)} (1 -\chi_{R1}^{(5)}) (1- \chi_{R3}^{(5)}) /\Phi_{\vp}^{(5)}, \\
L_{2, \vp}^{(5)}  &:= 4 \Big[ \frac{q_1^{(3)}}{ \Phi_0^{(3)} } \Big]_{ext1,1}^{(5)} 
[ q_1^{(3)} ]_{ext2,1}^{(5)} \, \chi_{NR1}^{(5)} \chi_{NR(1,1)}^{(5)} (1- \chi_{H1,1}^{(5)}) \chi_{A2}^{(5)} / \Phi_{\vp}^{(5)}, \\
L_{3, \vp}^{(5)}&:= 2 \Big[ \frac{q_1^{(3)}}{ \Phi_{\vp}^{(3)} } \chi_{>L}^{(3)}  \Big]_{ext1,1}^{(5)} [ q_1^{(3)} ]_{ext2,1}^{(5)} 
\, \chi_{NR1}^{(5)} \chi_{NR(2,1)}^{(5)} / \Phi_{\vp}^{(5)}, \\
L_{4, \vp}^{(5)}&:=  4 \Big[ \frac{q_1^{(3)}}{ \Phi_{\vp}^{(3)} } \chi_{>L}^{(3)}  \Big]_{ext1,1}^{(5)} [ q_1^{(3)} ]_{ext2,1}^{(5)} 
\, \chi_{NR1}^{(5)} \chi_{NR(3,1)}^{(5)} \chi_{A1}^{(5)} / \Phi_{\vp}^{(5)}, \\
L_{5, \vp}^{(5)}&:=  4 \Big[ \frac{q_1^{(3)}}{ \Phi_{\vp}^{(3)} } \chi_{>L}^{(3)}  \Big]_{ext1,1}^{(5)} 
[ q_1^{(3)} ]_{ext2,1}^{(5)} \, \chi_{NR1}^{(5)} \chi_{NR(4,1)}^{(5)} \chi_{A3}^{(5)} / \Phi_{\vp}^{(5)}, \\
L_{6, \vp}^{(5)}&:=  2 \Big[ \frac{q_1^{(3)}}{ \Phi_{\vp}^{(3)} } \chi_{>L}^{(3)}  \Big]_{ext1,1}^{(5)} [ q_1^{(3)}  ]_{ext2,1}^{(5)} 
\, \chi_{NR1}^{(5)} \chi_{NR(5,1)}^{(5)} \chi_{A1}^{(5)} / \Phi_{\vp}^{(5)}. 
\end{align*}
{\upshape (ii)} The multipliers $M_{1, \vp}^{(3)} $ and $\{M_{j, \vp}^{(5)}\}_{j=1}^{19}$ are defined as below:
\begin{align*}
M_{1,\vp}^{(3)}& := Q_2^{(3)}, \hspace{0.3cm} M_{1, \vp}^{(5)}:= \frac{3}{8} \la_1,  \\
M_{2, \vp}^{(5)} &:= \Big[2  \sum_{j=1}^6 \ti{L}_{j, \vp}^{(3)} \chi_{>L}^{(3)}  \Big]_{ext1,1}^{(5)} [ Q_2^{(3)} ]_{ext2,1}^{(5)}, \hspace{0.3cm}
M_{3, \vp}^{(5)} := -\Big[ \sum_{j=1}^6 \ti{L}_{j, \vp}^{(3)} \chi_{>L}^{(3)}  \Big]_{ext2,1}^{(5)} [ Q_2^{(3)} ]_{ext2,2}^{(5)}, \\
M_{4, \vp}^{(5)} &:= \Big[2  \sum_{j=2,3,5} \ti{L}_{j, \vp}^{(3)} \chi_{>L}^{(3)}  \Big]_{ext1,1}^{(5)} [ Q_1^{(3)} ]_{ext2,1}^{(5)}, \\
M_{5, \vp}^{(5)} & : = -\Big[ \sum_{j=2,3,5} \ti{L}_{j, \vp}^{(3)} \chi_{>L}^{(3)}  \Big]_{ext2,1}^{(5)} [ Q_1^{(3)} ]_{ext2,2}^{(5)}, \\
M_{6, \vp}^{(5)} &:= \Big[2  \sum_{j=4,6} \ti{L}_{j, \vp}^{(3)} \chi_{>L}^{(3)}  \Big]_{ext1,1}^{(5)} [ Q_1^{(3)} ]_{ext2,1}^{(5)}, \hspace{0.3cm} 
M_{7, \vp}^{(5)} := -\Big[ \sum_{j=4,6} \ti{L}_{j, \vp}^{(3)} \chi_{>L}^{(3)}  \Big]_{ext2,1}^{(5)} [ Q_1^{(3)} ]_{ext2,2}^{(5)}, \\
M_{8, \vp}^{(5)} &:= 4 \Big[ \frac{q_1^{(3)}}{ \Phi_0^{(3)}}  \Big]_{ext1,1}^{(5)} [ q_1^{(3)} ]_{ext2,1}^{(5)} \, 
\chi_{NR1}^{(5)}  \chi_{H1,1}^{(5)} \chi_{R1}^{(5)}, \\
M_{9, \vp}^{(5)} & :=4 \Big[ \frac{q_1^{(3)}}{ \Phi_0^{(3)}}  \Big]_{ext1,1}^{(5)} [ q_1^{(3)} ]_{ext2,1}^{(5)} \, 
\chi_{NR1}^{(5)} \chi_{H1,1}^{(5)} (1-\chi_{R1}^{(5)}) \chi_{R3}^{(5)}, \\
M_{10, \vp}^{(5)} &:=4 \Big[ \frac{q_1^{(3)}}{ \Phi_0^{(3)}}  \Big]_{ext1,1}^{(5)} [ q_1^{(3)} ]_{ext2,1}^{(5)} \, 
\chi_{NR1}^{(5)} \chi_{NR(1,1)}^{(5)}(1- \chi_{H1,1}^{(5)}) (1-\chi_{A2}^{(5)}), \\
M_{11 ,\vp}^{(5)} &:= 4 \Big[ \Big( \frac{q_1^{(3)}}{ \Phi_\vp^{(3)}}- \frac{ q_1^{(3)} }{ \Phi_0^{(3)} } \Big) \chi_{>L}^{(3)} \Big]_{ext1,1}^{(5)} 
[ q_1^{(3)} ]_{ext2,1}^{(5)} \, \chi_{NR1}^{(5)} \chi_{NR(1,1)}^{(5)}, \\
M_{12, \vp}^{(5)} &:= 4 \Big[ \Big(- \frac{ q_1^{(3)} }{ \Phi_0^{(3)} } \Big) \chi_{ \le L}^{(3)} \Big]_{ext1,1}^{(5)} 
[ q_1^{(3)} ]_{ext2,1}^{(5)} \, \chi_{NR1}^{(5)} \chi_{NR(1,1)}^{(5)}, \\
M_{13, \vp}^{(5)} &:= 4 \Big[ \frac{q_1^{(3)}}{ \Phi_\vp^{(3)}} \chi_{>L}^{(3)} \Big]_{ext1,1}^{(5)} 
[ q_1^{(3)} ]_{ext2,1}^{(5)} \, \chi_{NR1}^{(5)} \chi_{NR(3,1)}^{(5)} (1-\chi_{A1}^{(5)}) , \\
M_{14, \vp}^{(5)} &:= 4 \Big[ \frac{q_1^{(3)}}{ \Phi_\vp^{(3)}} \chi_{>L}^{(3)} \Big]_{ext1,1}^{(5)} 
[ q_1^{(3)} ]_{ext2,1}^{(5)} \, \chi_{NR1}^{(5)} \chi_{NR(4,1)}^{(5)} (1-\chi_{A3}^{(5)}) , \\
M_{15, \vp}^{(5)} &:= 2 \Big[ \frac{q_1^{(3)}}{ \Phi_\vp^{(3)}} \chi_{>L}^{(3)} \Big]_{ext1,1}^{(5)} 
[ q_1^{(3)} ]_{ext2,1}^{(5)} \, \chi_{NR1}^{(5)} \chi_{NR(5,1)}^{(5)} (1-\chi_{A1}^{(5)}) , \\
M_{16, \vp}^{(5)} &:= 2 \Big[ \frac{q_1^{(3)}}{ \Phi_\vp^{(3)}} \chi_{>L}^{(3)} \Big]_{ext1,1}^{(5)} 
[ q_1^{(3)} ]_{ext2,1}^{(5)} \, \chi_{NR1}^{(5)} (2 \chi_{NR(3,2)}^{(5)}+ 2 \chi_{NR(4,2)}^{(5)}+ \chi_{NR(5,2)}^{(5)}), \\ 
M_{17, \vp}^{(5)} &:= 2 \Big[ \frac{q_1^{(3)}}{ \Phi_\vp^{(3)}} \chi_{>L}^{(3)} \Big]_{ext1,1}^{(5)} 
[ q_1^{(3)} ]_{ext2,1}^{(5)} \, \chi_{NR1}^{(5)} (2 \chi_{NR(1,2)}^{(5)}+ \chi_{NR(2,2)}^{(5)} ) , \\
M_{18, \vp}^{(5)} &:=   \big[2 \ti{L}_{1, \vp}^{(3)} \chi_{>L}^{(3)} \big]_{ext1,1}^{(5)} [Q_1^{(3)} \chi_{H3}^{(3)} ]_{ext2,1}^{(5)}, \\
M_{19, \vp}^{(5)} &:=  - \big[ \ti{L}_{1, \vp}^{(3)}  \chi_{> L}^{(3)} \big]_{ext1,2}^{(5)} [Q_1^{(3)} ]_{ext2,2}^{(5)}.  
\end{align*}
{\upshape (iii)} The multipliers $\{M_{j, \vp}^{(7)}\}_{j=1}^{6}$ and $\{ M_{j, \vp}^{(9)} \}_{j=1}^2$ are defined as below:
\begin{align*}
 M_{1, \vp}^{(7)} &:= \frac{3}{8} \la_1 \Big[ 2 \sum_{j=1}^6 \ti{L}_{j, \vp}^{(3)} \chi_{>L}^{(3)}  \Big]_{ext1,1}^{(7)} 
 ,  \hspace{0.3cm}
M_{2, \vp}^{(7)} :=- \frac{3}{8} \la_1 \Big[ \sum_{j=1}^6 \ti{L}_{j, \vp}^{(3)} \chi_{>L}^{(3)}  \Big]_{ext1,2}^{(7)},  \\
M_{3, \vp}^{(7)} &:= \Big[ 3 \sum_{j=1}^6 \ti{L}_{j, \vp}^{(5)} \chi_{>L}^{(5)} \Big]_{ext1,1}^{(7)} 
[ Q_2^{(3)} ]_{ext2,1}^{(7)}, \hspace{0.3cm}
M_{4, \vp}^{(7)} := -\Big[ 2 \sum_{j=1}^6 \ti{L}_{j, \vp}^{(5)} \chi_{>L}^{(5)} \Big]_{ext2,1}^{(7)} 
[ Q_2^{(3)} ]_{ext2,2}^{(7)}, \\
M_{5, \vp}^{(7)} &:= \Big[ 3 \sum_{j=1}^6 \ti{L}_{j, \vp}^{(5)} \chi_{>L}^{(5)} \Big]_{ext1,1}^{(7)} 
[ Q_1^{(3)} ]_{ext2,1}^{(7)}, \hspace{0.3cm}
M_{6, \vp}^{(7)} := -\Big[ 2 \sum_{j=1}^6 \ti{L}_{j, \vp}^{(5)} \chi_{>L}^{(5)} \Big]_{ext2,1}^{(7)} 
[ Q_1^{(3)} ]_{ext2,2}^{(7)}, \\
M_{1, \vp}^{(9)} &:= \frac{3}{8} \la_1 \Big[ 3 \sum_{j=1}^6 \ti{L}_{j, \vp}^{(5)} \chi_{>L}^{(5)} \Big]_{ext1,1}^{(9)}, \hspace{0.3cm}
M_{2, \vp}^{(9)} := - \frac{3}{8} \la_1 \Big[ 2 \sum_{j=1}^6 \ti{L}_{j, \vp}^{(5)} \chi_{>L}^{(5)} \Big]_{ext1,2}^{(9)}. 
\end{align*}
\end{prop}

\begin{rem}
Precisely speaking, $L_{1, \vp}^{(3)}  := Q_1^{(3)} \, 2 \chi_{H1,1}^{(3)}/\Phi_{\vp}^{(3)}$ means
\begin{align*}
L_{1, \vp}^{(3)} :=
\begin{cases}
Q_1^{(3)} \, 2 \chi_{H1,1}^{(3)}/\Phi_{\vp}^{(3)} \ \ &\text{ when } \Phi_{\vp}^{(3)} \neq 0\\
0\ \ &\text{ when } \Phi_{\vp}^{(3)}= 0.
\end{cases}
\end{align*}
We adapt the same rule in the definition of each multiplier when its denominator is equal to $0$.
\end{rem}

Before we prove Proposition \ref{prop_NF2}, we prepare some propositions and lemmas.
\begin{prop} \label{prop_req1}
Let $s \ge 1$, $\vp \in L^2(\T)$, $T>0$ and $u \in C([-T, T]:H^s(\T))$ be a solution of \eqref{4NLS3}. 
Put $v:= U_\vp(-t)u$. Then, $v \in C([-T, T]: H^s(\T))\cap C^1([-T, T]: H^{s-2}(\T))$ and 
$\ha{v}(t,k)$ satisfies 
\begin{align} 
\p_t \ha{v} (t,k) =&  \La_{\vp}^{(3)}( iQ_1^{(3)} ,\ha{v} (t) )(t,k)
+ \La_{\vp}^{(3)}(i Q_2^{(3)}, \ha{v} (t) )(t,k) 
+ \La_{\vp}^{(5)} \big(i \frac{3}{8} \la_1, \ha{v} (t) \big) (t, k), \label{eq20} \\
\p_t \bar{\ha{v}}(t, k)= &  
\sum_{k=k_{1,2,3}} e^{it \Phi_{\vp}^{(3)}} (-i Q_1^{(3)}) \bar{\ha{v}}(t, k_1) \ha{v}(t, k_2) \bar{\ha{v}}(t, k_3) \notag \\
&+ \sum_{k=k_{1,2,3}} e^{it \Phi_{\vp}^{(3)}} (-i Q_2^{(3)}) \bar{\ha{v}}(t, k_1) \ha{v}(t, k_2) \bar{\ha{v}}(t, k_3) \notag \\
&+ \sum_{k=k_{1,2,3,4,5}} e^{it \Phi_{\vp}^{(5)} } \big( -i \frac{3}{8} \la_1\big) 
\bar{\ha{v}}(t,k_1) \ha{v} (t, k_2) \bar{\ha{v}}(t, k_3) \ha{v} (t, k_4) \bar{\ha{v}} (t, k_5)
\label{eq21}
\end{align}
for each $k \in \Z$, where $U_{\vp}(t)= \mathcal{F}^{-1} \exp(i t \phi_{\vp} (k)) \mathcal{F}_x$. Moreover, we have 
\begin{align}
\|\p_t v\|_{L_T^\infty H^{s-2}} \lesssim \|v\|_{L_T^\infty H^{s}}^3 + \|  v \|_{L_T^{\infty} H^s}^5.  \label{eq_es}
\end{align}
\end{prop}
\begin{proof}
Since $u\in C([-T, T]:H^s(\T))$ satisfies \eqref{4NLS3} and $\ha{v}(t,k)=e^{-i t\phi_{\vp} (k)} \ha{u}(t,k)$, it follows that
$v\in C([-T,T]: H^{s}(\T))$ and
\begin{equation}\label{EE2}
\p_t \ha{v} (t,k) = \sum_{j=1,2,3,4} e^{-i t \phi_{\vp} (k) } \mathcal{F}_x [J_j(u)](t, k)
\end{equation}
By \eqref{nl11}, it follows that $J_1(u), J_2(u), J_3(u), J_4(u) \in C([-T, T]: H^{s-2}(\T))$. 
Thus, we have $v \in C^1([-T,T]: H^{s-2}(\T))$ and \eqref{eq_es}. 
A direct computation yields that 
\begin{align*}
e^{-it \phi_{\vp} (k) } \mathcal{F}_x [J_1(u)](t,k)
& = \sum_{k=k_{1,2,3}} e^{-it \Phi_{\vp}^{(3)}  } (i q_1^{(3)}) \prod_{1 \le l \le 3}^{*} \ha{v}(t, k_l), \\
e^{-it \phi_{\vp} (k) } \mathcal{F}_x [ J_2(u) ](t,k)
&= \sum_{k=k_{1,2,3}} e^{-it \Phi_{\vp}^{(3)}} (-iq_1^{(3)} [2 \chi_{R1}^{(3)}]_{sym1}^{(3)} ) \prod_{1 \le l \le 3}^{*} \ha{v} (t, k_l). 
\end{align*}
Thus, by $1= \chi_{NR1}^{(3)}+ [ 2 \chi_{R1}^{(3)}  ]_{sym1}^{(3)}-\chi_{R2}^{(3)}$, it follows that 
\begin{align} \label{EE4}
& e^{-it \phi_{\vp}(k)} \mathcal{F}_x [J_1(u)+J_2(u)  ]
= \sum_{k=k_{1,2,3}} e^{- it \Phi_{\vp}^{(3)} } \big\{ i q_1^{(3)} (1- [2\chi_{R1}^{(3)} ]_{sym1}^{(3)}  ) \big\}
\prod_{1 \le l \le 3}^{*} \ha{v} (t, k_l) \notag \\
& \hspace{1cm} 
= \sum_{k=k_{1,2,3}} e^{- it \Phi_{\vp}^{(3)} } ( i q_1^{(3)} \chi_{NR1}^{(3)}- i q_1^{(3)} \chi_{R2}^{(3)} \big)
\prod_{1 \le l \le 3}^{*} \ha{v} (t, k_l) \notag \\ 
& \hspace{1cm}
= \La_{\vp}^{(3)} \big(i Q_1^{(3)} , \ha{v} (t) \big) (t, k)+ \La_{\vp}^{(3)} \big(-i q_1^{(3)} \chi_{R2}^{(3)}, \ha{v} (t) \big) (t, k).  
\end{align}
A simple calculation yields that 
\begin{align} 
e^{-it \phi_{\vp} (k) } \mathcal{F}_x [ J_3(u) ] (t,k)
& = \sum_{k=k_{1,2,3}} e^{-it \Phi_{\vp}^{(3)} } (i q_2^{(3)} ) \prod_{1 \le l \le 3}^{*} \ha{v} (t, k_l) \notag \\
& = \La_{\vp}^{(3)} \big( i q_2^{(3)}, \ha{v} (t)  \big) (t, k), \label{EE5} \\
e^{-it \phi_{\vp} (k) } \mathcal{F}_x [ J_4(u) ] (t,k)
& = \sum_{k=k_{1,2,3,4,5}} e^{-it \Phi_{\vp}^{(5)} } \big( i \frac{3}{8} \la_1\big) \prod_{1 \le l \le 5}^{*} \ha{v} (t, k_l) \notag \\
& = \La_{\vp}^{(5)} \big( i \frac{3}{8} \la_1, \ha{v} (t)  \big) (t, k). \label{EE6}
\end{align}
By \eqref{EE4}--\eqref{EE6} and $Q_2^{(3)} =q_2^{(3)}- q_1^{(3)} \chi_{R2}^{(3)}$, we obtain \eqref{eq20}. 
By \eqref{eq20}, we easily check that \eqref{eq21} holds. 
\end{proof}

\begin{prop} \label{prop_NF11}
Let $N =1,2$, $s \ge 1$, $\vp \in L^2(\T)$, $T>0$ and a $(2N+1)$-multiplier $m^{(2N+1)}$ satisfy
\begin{align*}
|m^{(2N+1)} | \lesssim 1, 
\hspace{0.5cm} | m^{(2N+1)} \Phi_{\vp}^{(2N+1)} | \lesssim \langle k_{\max} \rangle^2.
\end{align*} 
If $ u \in C([-T, T]:H^s(\T))$ is a solution of \eqref{4NLS3}, then 
$\ha{v}(t,k)=e^{-i t \phi_{\vp} (k)} \ha{u}(t,k) $ satisfies the following equation for each $k \in \Z$:
\begin{align}
\label{NF11}
& \p_t \Lambda_{\vp}^{(2N+1)} (\ti{m}^{(2N+1)}, \ha{v}(t))(t,k) = 
\Lambda_{\vp}^{(2N+1)} (- i \, \ti{m}^{(2N+1)} \Phi_{\vp}^{(2N+1)}, \ha{v} (t) )(t,k) \notag \\
& \hspace{0.3cm} 
+ \Lambda_{\vp}^{(2N+3)} \big( \big[ [(N+1) \ti{m}^{(2N+1)} ]_{ext1,1}^{(2N+3)} [ i Q_1^{(3)} ]_{ext2,1}^{(2N+3)}  \big]_{sym}^{(2N+3)}, 
\ha{v} (t) \big)(t,k) \notag \\
& \hspace{0.3cm} 
+ \Lambda_{\vp}^{(2N+3)} \big( \big[ [(N+1) \ti{m}^{(2N+1)} ]_{ext1,1}^{(2N+3)} [ i Q_2^{(3)} ]_{ext2,1}^{(2N+3)}  \big]_{sym}^{(2N+3)}, 
\ha{v} (t) \big)(t,k) \notag \\
& \hspace{0.3cm} 
+ \Lambda_{\vp}^{(2N+3)} \big( \big[ [ N \ti{m}^{(2N+1)} ]_{ext1,2}^{(2N+3)} [ -i Q_1^{(3)} ]_{ext2,2}^{(2N+3)}  \big]_{sym}^{(2N+3)}, 
\ha{v} (t) \big)(t,k) \notag \\
& \hspace{0.3cm}
+ \Lambda_{\vp}^{(2N+3)} \big( \big[ [ N \ti{m}^{(2N+1)} ]_{ext1,2}^{(2N+3)} [ - i Q_2^{(3)} ]_{ext2,2}^{(2N+3)}  \big]_{sym}^{(2N+3)}, 
\ha{v} (t) \big)(t,k) \notag \\
& \hspace{0.3cm}
+ \Lambda_{\vp}^{(2N+5)} \big( i \frac{3}{8} \la_1 \big[ [ (N+1) \ti{m}^{(2N+1)} ]_{ext1,1}^{(2N+5)} \big]_{sym}^{(2N+5)}, 
\ha{v} (t) \big)(t,k) \notag \\
& \hspace{0.3cm}
+ \Lambda_{\vp}^{(2N+5)} \big( -i \frac{3}{8} \la_1 \big[ [ N \ti{m}^{(2N+1)} ]_{ext1,2}^{(2N+5)} \big]_{sym}^{(2N+5)}, 
\ha{v} (t) \big)(t,k).
\end{align} 
\end{prop}

Note that any $L_{j, \vp}^{(N)}$ defined in Proposition~\ref{prop_NF2} satisfies 
\begin{align*}
|L_{j, \vp}^{(N)} \chi_{>L}^{(N)} | \lesssim 1,
\hspace{0.5cm}
|L_{j, \vp}^{(N)} \Phi_{\vp}^{(N)} \chi_{>L}^{(N)} | \lesssim \langle k_{\max} \rangle^2
\end{align*}
for $L \gg \max\{ 1, |\la_5 | E_1(\vp)  \}$. Thus, we can apply Proposition~\ref{prop_NF11} 
with $\ti{m}^{(N)}=\ti{L}_{j, \vp}^{(N)} \chi_{>L}^{(N)}$. 
For details, see Remak~\ref{rem_pwb11}.

\begin{proof}[Proof of Proposition~\ref{prop_NF11}] 
At least formally, we have 
\begin{align} \label{NF112}
 \p_t \La_{\vp}^{(2N+1)} (\ti{m}^{(2N+1)}, \ha{v} (t)) (t,k)
&= \sum_{k=k_{1, \dots, 2N+1}} \frac{\p}{\p t} \Big\{ e^{-it  \Phi_{\vp}^{(2N+1)}}  \ti{m}^{(2N+1)} \prod_{1 \le l \le 2N+1}^{*} \ha{v} (t, k_{l}) \Big\} \\
& =: \mathcal{N}_1(t,k)+ \mathcal{N}_2(t, k) + \mathcal{N}_3 (t,k) \notag 
\end{align}
where
\begin{align}
& \mathcal{N}_1(t,k):= 
\sum_{k=k_{1, \dots, 2N+1}} e^{-i t \Phi_{\vp}^{(2N+1)}} (-i \ti{m}^{(2N+1)} \Phi_{\vp}^{(2N+1)}) \prod_{1 \le l \le 2N+1}^{*} \ha{v}(t, k_l), 
\notag \\
& \mathcal{N}_2(t, k):= \sum_{k=k_{1,\dots, 2N+1}} e^{-it \Phi_{\vp}^{(2N+1)} } (N+1) \ti{m}^{(2N+1)} \notag \\
& \hspace{3cm} \times 
\prod_{1\le l \le 2N-1}^{*} \ha{v} (t, k_l) \, \bar{\ha{v}}(t, k_{2N}) \, \p_t \ha{v} (t, k_{2N+1})  , \label{NF113} \\
& \mathcal{N}_3(t, k):=\sum_{k=k_{1,\dots, 2N+1}} e^{-it \Phi_{\vp}^{(2N+1)} } N \ti{m}^{(2N+1)} 
\prod_{1 \le l \le 2N-1}^{*} \ha{v}(t, k_l) \, \p_t \bar{\ha{v}} (t, k_{2N}) \,  \ha{v} (t, k_{2N+1}). 
 \label{NF114}
\end{align}
By $ |\ti{m}^{(2N+1)} \Phi_{\vp}^{(2N+1)} | \le [ |  m^{(2N+1)} \Phi_{\vp}^{(2N+1)}  | ]_{sym}^{(2N+1)} \lesssim \langle k_{\max} \rangle^2 $ 
and \eqref{nl11}, 
\begin{equation*}
\| \mathcal{N}_1 \|_{L_T^{\infty} l_{s-2}^2 } \lesssim \|  v \|_{L_T^{\infty} H^s}^{2N+1}. 
\end{equation*} 
By $|\ti{m}^{2N+1}| \lesssim 1 \lesssim \langle k_{\max}  \rangle^2 \langle k_{2N+1} \rangle^{-2}$, \eqref{nl11} and 
\eqref{eq_es}, 
\begin{align*}
\| \mathcal{N}_2  \|_{L_T^{\infty} l_{s-2}^2 } 
& \lesssim \Big\| \sum_{k=k_{1, \dots, 2N+1}} \langle k_{\max} \rangle^2 \prod_{l=1}^{2N} |\ha{v} (t, k_l)|
 \big( \langle k_{2N+1} \rangle^{-2} |\p_t \ha{v} (t, k_{2N+1})| \big) \Big\|_{L_T^{\infty} l_{s-2}^2} \\
& \lesssim \| v \|_{L_T^{\infty} H^s }^{2N} \| \p_t v \|_{L_T^{\infty} H^{s-2} } 
\lesssim \| v \|_{L_T^{\infty}  H^s }^{2N+3} + \| v \|_{L_T^{\infty} H^s}^{2N+5}  .
\end{align*} 
In a similar manner as above, we have 
$\|  \mathcal{N}_{3}  \|_{L_T^{\infty} l_{s-2}^2  } \lesssim \| v \|_{L_T^{\infty} H^s}^{2N+3}+\| v \|_{L_T^{\infty} H^s(\T)}^{2N+5} $. 
Thus, for each $k \in \Z$,  the convergence of $\mathcal{N}_1(t,k)$, $\mathcal{N}_2(t,k)$ and $\mathcal{N}_3 (t,k)$ 
is absolute and uniform in $t \in [-T, T]$. 
Therefore, changing the sum and the time differentiation in \eqref{NF112} can be verified strictly. 
Substituting \eqref{eq20} into \eqref{NF113}, we have
\begin{align}
\mathcal{N}_2 (t, k)
= & \La_{\vp}^{(2N+3)} \Big( [ (N+1) \ti{m}^{(2N+1)}  ]_{ext1,1}^{(2N+3)} [ i Q_1^{(3)} ]_{ext1,2}^{(2N+3)}, \ha{v} (t)  \Big) (t,k) \notag \\
&+ \La_{\vp}^{(2N+3)} \Big( [ (N+1) \ti{m}^{(2N+1)}  ]_{ext1,1}^{(2N+3)} [ i Q_2^{(3)} ]_{ext1, 2}^{(2N+3)}, \ha{v} (t)  \Big) (t,k) \notag \\
& + \La_{\vp}^{(2N+5)} \Big( i \frac{3}{8} \la_1 [ (N+1) \ti{m}^{(2N+1)}  ]_{ext1,1}^{(2N+5)}, \ha{v} (t)  \Big) (t,k) \notag \\
= & \La_{\vp}^{(2N+3)} \Big( \big[ [ (N+1) \ti{m}^{(2N+1)}  ]_{ext1,1}^{(2N+3)} [ i Q_1^{(3)} ]_{ext1,2}^{(2N+3)} \big]_{sym}^{(2N+3)} , 
\ha{v} (t)  \Big) (t,k) \notag \\
&+ \La_{\vp}^{(2N+3)} \Big( \big[ [ (N+1) \ti{m}^{(2N+1)}  ]_{ext1,1}^{(2N+3)} [ i Q_2^{(3)} ]_{ext1,2}^{(2N+3)} \big]_{sym}^{(2N+3)} , 
\ha{v} (t)  \Big) (t,k) \notag \\
&+ \La_{\vp}^{(2N+5)} \Big(  i \frac{3}{8} \la_1 \big[ [ (N+1) \ti{m}^{(2N+1)}  ]_{ext1,1}^{(2N+5)} \big]_{sym}^{(2N+5)} , 
\ha{v} (t)  \Big) (t,k). \label{NF115}
\end{align}
Similarly, we insert \eqref{eq21} into \eqref{NF114} to have 
\begin{align} \label{NF116}
\mathcal{N}_3 (t, k)
= & \La_{\vp}^{(2N+3)} \Big( \big[ [ N \ti{m}^{(2N+1)}  ]_{ext1,2}^{(2N+3)} [ -i Q_1^{(3)} ]_{ext2,1}^{(2N+3)} \big]_{sym}^{(2N+3)}, 
\ha{v} (t)  \Big) (t,k) \notag \\
&+ \La_{\vp}^{(2N+3)} \Big( \big[ [ N \ti{m}^{(2N+1)}  ]_{ext1,2}^{(2N+3)} [ -i Q_2^{(3)} ]_{ext2,1}^{(2N+3)} \big]_{sym}^{(2N+3)}, 
\ha{v} (t)  \Big) (t,k) \notag \\
&+ \La_{\vp}^{(2N+5)} \Big( - i \frac{3}{8} \la_1 \big[ [ N \ti{m}^{(2N+1)}  ]_{ext1,2}^{(2N+5)} \big]_{sym}^{(2N+5)}, 
\ha{v} (t)  \Big) (t,k). 
\end{align}
By \eqref{NF115}, \eqref{NF116} and $\mathcal{N}_1(t,k)= \La_{\vp}^{(2N+1)} (-i \ti{m}^{(2N+1)} \Phi_{\vp}^{(2N+1)}, \ha{v} (t))(t,k)$, 
we obtain \eqref{NF11}. 
\end{proof}

\begin{lem}\label{Le31}
\begin{align} 
&\partial_t \Lambda_{\vp}^{(3)} (  \ti{L}_{1, \vp}^{(3)} \chi_{>L}^{(3)}, \ha{v} (t) ) (t,k)
=-\Lambda_{\vp}^{(3)} (  i \ti{L}_{1, \vp}^{(3)} \Phi_{\vp}^{(3)}  \chi_{>L}^{(3)}, \ha{v} (t) ) (t,k) \notag \\
&+ \Lambda_{\vp}^{(5)} \big( i \sum_{j=1}^{6} \ti{L}_{j ,\vp}^{(5)} \Phi_{\vp}^{(5)}  \chi_{ > L}^{(5)} + 
i \sum_{j=1}^{6} \ti{L}_{j, \vp}^{(5)} \Phi_{\vp}^{(5)} \chi_{ \le L}^{(5)} +
i \sum_{j=8}^{18} \ti{M}_{j, \vp}^{(5)}, \ha{v} (t) \big) (t,k) \notag \\ 
&+ \La_{\vp}^{(5)}  \big( i \ti{M}_{19, \vp}^{(5)}, \ha{v}(t) \big)(t,k) \notag \\
& + \La_{\vp}^{(5)} \big( i \big[ [ 2 \ti{L}_{1, \vp}^{(3)} \chi_{>L}^{(3)} ]_{ext1,1}^{(5)} 
[Q_2^{(3)} ]_{ext2,1}^{(5)} \big]_{sym}^{(5)}, \ha{v} (t) \big) (t, k) \notag \\
&  + \Lambda_{\vp}^{(5)} 
\big(-i \big[ [  \ti{L}_{1, \vp}^{(3)}  \chi_{>L}^{(3)} ]_{est1,2}^{(5)} [Q_2^{(3)}]_{ext2,2}^{(5)} \big]_{sym}^{(5)}, \ha{v} (t) \big) (t,k) \notag \\
&+ \La_{\vp}^{(7)} \big( i \frac{3}{8} \la_1 \big[ [ 2 \ti{L}_{1, \vp}^{(3)} \chi_{>L}^{(3)} ]_{ext1,1}^{(7)} \big]_{sym}^{(7)}, \ha{v} (t) \big) (t, k) 
\notag \\
& +\La_{\vp}^{(7)} \big(- i \frac{3}{8} \la_1 \big[  [ \ti{L}_{1, \vp}^{(3)} \chi_{>L}^{(3)}]_{ext2,1}^{(7)}  \big]_{sym}^{(7)}, \ha{v}(t)  \big) (t,k). 
\label{eq36}
\end{align}
\end{lem}
\begin{proof}
By Proposition \ref{prop_NF11} with $N=1$ and $\ti{m}^{(2N+1)}=\ti{L}_{1, \vp}^{(3)} \chi_{>L}^{(3)}$, we have
\begin{align*} 
&\partial_t \Lambda_{\vp}^{(3)} (  \ti{L}_{1, \vp}^{(3)} \chi_{>L}^{(3)}, \ha{v} (t) ) (t,k)
=- \Lambda_{\vp}^{(3)} ( i  \ti{L}_{1, \vp}^{(3)} \Phi_{\vp}^{(3)}  \chi_{>L}^{(3)}, \ha{v} (t) ) (t,k)\\
&+ \Lambda_{\vp}^{(5)} \big( i \big[ [2 \ti{L}_{1, \vp}^{(3)} \chi_{>L}^{(3)}  ]_{ext1,1}^{(5)} [ Q_1^{(3)} ]_{ext2,1}^{(5)} \big]_{sym}^{(5)}, \ha{v} (t) \big) (t,k) \notag \\ 
&+ \Lambda_{\vp}^{(5)} \big( i \ti{M}_{19, \vp}^{(5)} , \ha{v} (t) \big) (t,k) \notag \\
& + \Lambda_{\vp}^{(5)} \big( i \big[ [2 \ti{L}_{1, \vp}^{(3)} \chi_{>L}^{(3)}  ]_{ext1,1}^{(5)} [ Q_2^{(3)} ]_{ext2,1}^{(5)} \big]_{sym}^{(5)}, \ha{v} (t) \big) (t,k) \notag \\ 
&+ \Lambda_{\vp}^{(5)} \big( -i \big[ [\ti{L}_{1, \vp}^{(3)} \chi_{>L}^{(3)}  ]_{ext1,2}^{(5)} [ Q_2^{(3)} ]_{ext2,2}^{(5)} \big]_{sym}^{(5)}, \ha{v} (t) \big) (t,k) \notag \\
&+ \La_{\vp}^{(7)} \big( i \frac{3}{8} \la_1 \big[ [ 2 \ti{L}_{1, \vp}^{(3)} \chi_{>L}^{(3)} ]_{ext1,1}^{(7)} \big]_{sym}^{(7)}, \ha{v} (t) \big) (t, k) 
\notag \\
& +\La_{\vp}^{(7)} \big(- i \frac{3}{8} \la_1 \big[  [ \ti{L}_{1, \vp}^{(3)} \chi_{>L}^{(3)}]_{ext2,1}^{(7)}  \big]_{sym}^{(7)}, \ha{v}(t)  \big) (t,k). 
\end{align*}
Thus, we only need to show
\begin{equation} \label{eq37}
\big[ [2 \ti{L}_{1, \vp}^{(3)} \chi_{>L}^{(3)}  ]_{ext1,1}^{(5)} [ Q_1^{(3)}  ]_{ext2,1}^{(5)} \big]_{sym}^{(5)} 
= \sum_{j=1}^{6} \ti{L}_{i, \vp}^{(5)} \Phi_{\vp}^{(5)}+  \sum_{j=8}^{18} \ti{M}_{j, \vp}^{(5)}. 
\end{equation}
By the definition, the left hand side of \eqref{eq37} is equal to 
\begin{equation*}
\ti{M}_{18, \vp}^{(5)}+ 
\big[ [2 \ti{L}_{1, \vp}^{(3)} \chi_{>L}^{(3)}  ]_{ext1}^{(5)} 
[ Q_1^{(3)} (1-\chi_{H3}^{(3)} ) ]_{ext2,1}^{(5)} \big]_{sym}^{(5)}. 
\end{equation*}
By $1-\chi_{H3}^{(3)}=[2 \chi_{H1,1}^{(3)}]_{sym1}^{(3)} +\chi_{H1,2}^{(3)}+[2 \chi_{H2,1}^{(3)}]_{sym1}^{(3)}+ [2 \chi_{H2,2}^{(3)} ]_{sym1}^{(3)} + \chi_{H2,3}^{(5)}$ and Lemma~\ref{Le6} with $m_1^{(3)} =Q_1^{(3)} \chi_{>L}^{(3)} / \Phi_{\vp}^{(3)}$ and $m_2^{(3)} = Q_1^{(3)}$, 
we have 
\begin{align}
& \big[  [2 \ti{L}_{1,\vp}^{(3)} \chi_{>L}^{(3)} ]_{ext1,1}^{(3)} [Q_1^{(3)}(1- \chi_{H3}^{(3)}) ]_{ext2,1}^{(5)} \big]_{sym}^{(5)}\notag\\
&= 2 \Big[ \Big[ \frac{Q_1^{(3)} }{ \Phi_{\vp}^{(3)}}  \chi_{>L}^{(3)} [ 2 \chi_{H1,1}^{(3)} ]_{sym1}^{(3)}  \Big]_{ext1,1}^{(5)} 
\big[ Q_1^{(3)} \big([2 \chi_{H1,1}^{(3)}]_{sym1}^{(3)}+ [ 2\chi_{H2,1}^{(3)} ]_{sym1}^{(3)} +[2 \chi_{H2,2}^{(3)}]_{sym1}^{(3)} \big) \big]_{ext2,1}^{(5)} \Big]_{sym}^{(5)}  \notag \\
&\hspace{0.4cm} +  2 \Big[ \Big[ \frac{Q_1^{(3)} }{ \Phi_{\vp}^{(3)}}  \chi_{>L}^{(3)} [ 2 \chi_{H1,1}^{(3)} ]_{sym1}^{(3)}  \Big]_{ext1,1}^{(5)} 
\big[ Q_1^{(3)} \big( \chi_{H2,1}^{(3)}+ \chi_{H2,3}^{(3)} \big) \big]_{ext2,1}^{(5)} \Big]_{sym}^{(5)}  \notag \\
& =\sum_{j=1,3,4} 2 \Big\{ \Big[ \Big[ \frac{Q_1^{(3)} }{\Phi_{\vp}^{(3)} } \chi_{>L}^{(3)}  \Big]_{ext1,1}^{(5)} [Q_1^{(3)}  ]_{ext2,1}^{(5)}  2 \chi_{NR(j,1)}^{(5)}  \Big]_{sym}^{(5)} \notag \\ 
& \hspace{3cm} + \Big[ \Big[ \frac{Q_1^{(3)} }{\Phi_{\vp}^{(3)} } \chi_{>L}^{(3)}  \Big]_{ext1,1}^{(5)} [Q_1^{(3)}  ]_{ext2,1}^{(5)}
 2 \chi_{NR(j, 2)}^{(5)}  \Big]_{sym}^{(5)} \Big\} \notag \\
& \hspace{0.4cm} + \sum_{j=2,5} 2 \Big\{ \Big[ \Big[ \frac{Q_1^{(3)} }{\Phi_{\vp}^{(3)} } \chi_{>L}^{(3)}  \Big]_{ext1,1}^{(5)} [Q_1^{(3)}  ]_{ext2,1}^{(5)}  \chi_{NR(j,1)}^{(5)}  \Big]_{sym}^{(5)} \notag \\
& \hspace{3cm} + \Big[ \Big[ \frac{Q_1^{(3)} }{\Phi_{\vp}^{(3)} } \chi_{>L}^{(3)}  \Big]_{ext1,1}^{(5)} [Q_1^{(3)}  ]_{ext2,1}^{(5)}
 \chi_{NR(j, 2)}^{(5)}  \Big]_{sym}^{(5)} \Big\}. \label{eq38}
\end{align} 
By Remark~\ref{rem_sym} and $\chi_{NR1}^{(5)}= [\chi_{NR1}^{(3)} ]_{ext1,1}^{(5)} [\chi_{NR1}^{(3)}]_{ext2,1}^{(5)}$, 
\begin{equation} \label{eq381}
\Big[  \frac{Q_1^{(3)}}{ \Phi_{\vp}^{(3)} } \chi_{>L}^{(3)} \Big]_{ext1,1}^{(5)} [Q_1^{(3)}]_{ext2,1}^{(5)}
=\Big[  \frac{q_1^{(3)}}{ \Phi_{\vp}^{(3)} } \chi_{>L}^{(3)} \Big]_{ext1,1}^{(5)} [q_1^{(3)}]_{ext2,1}^{(5)} \chi_{NR1}^{(5)}. 
\end{equation}
By \eqref{eq38} and \eqref{eq381}, 
\begin{align}
& \big[  [2 \ti{L}_{1, \vp}^{(3)} \chi_{>L}^{(3)} ]_{ext1,1}^{(5)} [ Q_1^{(3)} (1 -\chi_{H3}^{(3)}) ]_{ext2,1}^{(5)}  \big]_{sym}^{(5)} \notag \\
& = 4 \Big[ \Big[ \frac{q_1^{(3)}  }{ \Phi_{\vp}^{(3)} } \chi_{>L}^{(3)}   \Big]_{ext1,1}^{(5)} [ q_1^{(3)} ]_{ext2,1}^{(5)}
 \, \chi_{NR1}^{(5)} \chi_{NR(1,1)}^{(5)} \Big]_{sym}^{(5)} \notag \\
& + 4 \Big[ \Big[ \frac{q_1^{(3)}  }{ \Phi_{\vp}^{(3)} } \chi_{>L}^{(3)}   \Big]_{ext1,1}^{(5)} [ q_1^{(3)} ]_{ext2,1}^{(5)}
 \, \chi_{NR1}^{(5)} \chi_{NR(3,1)}^{(5)} \Big]_{sym}^{(5)} \label{eq392} \\
& + 4 \Big[ \Big[ \frac{q_1^{(3)}  }{ \Phi_{\vp}^{(3)} } \chi_{>L}^{(3)}   \Big]_{ext1,1}^{(5)} [ q_1^{(3)} ]_{ext2,1}^{(5)}
 \, \chi_{NR1}^{(5)} \chi_{NR(4,1)}^{(5)} \Big]_{sym}^{(5)} \label{eq393} \\
& +2 \Big[ \Big[ \frac{q_1^{(3)}  }{ \Phi_{\vp}^{(3)} } \chi_{>L}^{(3)}   \Big]_{ext1,1}^{(5)} [ q_1^{(3)} ]_{ext2,1}^{(5)}
 \, \chi_{NR1}^{(5)} \chi_{NR(5,1)}^{(5)} \Big]_{sym}^{(5)} \label{eq394} \\
& + \ti{L}_{3, \vp}^{(5)}\Phi_{\vp}^{(5)} + \sum_{j=16}^{17} \ti{M}_{j, \vp}^{(5)}. \notag 
\end{align}
By the  definition, \eqref{eq392} is equal to $\ti{L}_{4, \vp}^{(5)} \Phi_{\vp}^{(5)} + \ti{M}_{13, \vp}^{(5)}$, 
\eqref{eq393} is equal to $\ti{L}_{5, \vp}^{(5)} \Phi_{\vp}^{(5)} + \ti{M}_{14, \vp}^{(5)}$ and 
\eqref{eq394} is equal to $\ti{L}_{6, \vp}^{(5)} \Phi_{\vp}^{(5)} + \ti{M}_{15, \vp}^{(5)}$. 
Therefore, we only need to show
\begin{align} \label{eq39}
4 \Big[ \Big[ \frac{q_1^{(3)} }{\Phi_{\vp}^{(3)} } \chi_{>L}^{(3)}  \Big]_{ext1,1}^{(5)} [q_1^{(3)} ]_{ext2,1}^{(5)} \, \chi_{NR1}^{(5)}  
\chi_{NR(1,1)}^{(5)}  \Big]_{sym}^{(5)} 
=\sum_{j=1}^{2} \ti{L}_{j, \vp}^{(5)} \Phi_{\vp}^{(5)}+\sum_{j=8}^{12} \ti{M}_{j, \vp}^{(5)}.
\end{align}
By the definition,
\begin{align} \label{eq33}
 &4  \Big[ \Big[ \frac{q_1^{(3)} }{\Phi_{\vp}^{(3)} } \chi_{>L}^{(3)}  \Big]_{ext1,1}^{(5)} [q_1^{(3)}]_{ext2,1}^{(5)} \,
 \chi_{NR1}^{(5)} \chi_{NR(1,1)}^{(5)}  \Big]_{sym}^{(5)} \notag \\
& = \sum_{j=11}^{12} \ti{M}_{j, \vp}^{(5)} 
+ 4 \Big[ \Big[ \frac{q_1^{(3)} }{\Phi_{0}^{(3)} }  \Big]_{ext1,1}^{(5)} [q_1^{(3)}  ]_{ext2,1}^{(5)} \, 
\chi_{NR1}^{(5)} \chi_{NR(1,1)}^{(5)}  \Big]_{sym}^{(5)}.
\end{align}
Since $\supp \chi_{H1,1}^{(5)} \subset \supp \chi_{NR(1,1)}^{(5)}$, it follows  that
\begin{align*}
\chi_{NR(1,1)}^{(5)}& = \chi_{NR(1,1)}^{(5)}(1- \chi_{H1,1}^{(5)})+ \chi_{H1,1}^{(5)} \\
& = \chi_{NR(1,1)}^{(5)} (1- \chi_{H1,1}^{(5)}) \chi_{A2}^{(5)} + \chi_{NR(1,1)}^{(5)} (1-\chi_{H1}^{(5)}) (1- \chi_{A2}^{(5)}) \\
& \hspace{0.4cm} + \chi_{H1,1}^{(5)} \chi_{R1}^{(5)}+ \chi_{H1,1}^{(5)} (1-\chi_{R1}^{(5)}) \chi_{R3}^{(5)}+ 
\chi_{H1,1}^{(5)} (1-\chi_{R1}^{(5)}) (1-\chi_{R3}^{(5)})
\end{align*}
which leads that 
\begin{equation} \label{eq34}
4 \Big[ \Big[ \frac{q_1^{(3)} }{\Phi_{0}^{(3)} }   \Big]_{ext1,1}^{(5)} [q_1^{(3)}  ]_{ext2,1}^{(5)} \, \chi_{NR1}^{(5)}
\chi_{NR(1,1)}^{(5)}  \Big]_{sym}^{(5)} 
= \sum_{j=1}^2 \ti{L}_{j, \vp}^{(5)} \Phi_{\vp}^{(5)} + \sum_{j=8}^{10} \ti{M}_{j ,\vp}^{(5)}. 
\end{equation}
Collecting \eqref{eq39}--\eqref{eq34}, we obtain \eqref{eq39}. 
\end{proof}
Now, we prove Proposition~\ref{prop_NF2}. 
\begin{proof}[Proof of Proposition \ref{prop_NF2}]
By $ [2 \chi_{H1,1}^{(3)}]_{sym1}^{(3)}+ \chi_{H1,2}^{(3)}+[2 \chi_{H2,1}^{(3)}]_{sym1}^{(3)}+ [ 2  \chi_{H2,2}^{(3)}]_{sym1}^{(3)}+ \chi_{H2,3}^{(3)}+ \chi_{H3}^{(3)}=1 $, 
\begin{equation*}
Q_1^{(3)}
= \sum_{j=1}^6 \ti{L}_{j, \vp}^{(3)} \Phi_{\vp}^{(3)} \chi_{>L}^{(3)}+ \sum_{j=1}^6 \ti{L}_{j, \vp}^{(3)} \Phi_{\vp}^{(3)} \chi_{\le L}^{(3)}. 
\end{equation*}
Thus, by Proposition \ref{prop_req1}, we have
\begin{align} \label{eq370}
\p_t \ha{v} (t,k) =&   \La_{\vp}^{(3)} \Big(i \sum_{j=1}^6 \ti{L}_{j, \vp}^{(3)} \Phi_{\vp}^{(3)} \chi_{> L}^{(3)}
+ i \sum_{j=1}^6 \ti{L}_{j, \vp}^{(3)} \Phi_{\vp}^{(3)} \chi_{\le L}^{(3)} , \ha{v}(t) \Big)(t,k) \notag \\
& + \La_{\vp}^{(3)} \big( i \ti{M}_{1, \vp}^{(3)}, \ha{v}(t) \big) (t,k)
+\La_{\vp}^{(5)} \big( i \ti{M}_{1, \vp}^{(5)}  , \ha{v} (t) \big)(t,k). 
\end{align}
By Proposition~\ref{prop_NF11} with $N=1$ and $\ti{m}^{(2N+1)}=\sum_{j=2}^6 \ti{L}_{j, \vp}^{(3)} \chi_{>L}^{(3)}$, we have
\begin{align} \label{eq371}
&\p_t \La_{\vp}^{(3)} \Big(\sum_{j=2}^6 \ti{L}_{j, \vp}^{(3)} \chi_{>L}^{(3)}, \ha{v} (t) \Big) (t,k) 
=\La_{\vp}^{(3)} \Big(- i \sum_{j=2}^6 \ti{L}_{j, \vp}^{(3)} \Phi_{\vp}^{(3)} \chi_{>L}^{(3)} , \ha{v}(t) \Big) (t,k) \notag \\
& + \sum_{j=4}^7 \La_{\vp}^{(5)} \big( i \ti{M}_{j, \vp}^{(5)}, \ha{v}(t) \big) (t,k) 
 + \La_{\vp}^{(5)} \Big(i \big[ [ 2 \sum_{j=2}^6 \ti{L}_{j, \vp}^{(3)} \chi_{>L}^{(3)}]_{ext1,1}^{(5)} 
[ Q_2^{(3)} ]_{ext2,1}^{(5)} \big]_{sym}^{(5)}, \ha{v} (t) \Big) (t,k) \notag \\
& + \La_{\vp}^{(5)} \Big( -i \big[ [ \sum_{j=2}^6 \ti{L}_{j, \vp}^{(3)} \chi_{>L}^{(3)}]_{ext1,2}^{(5)} 
[Q_2^{(3)}]_{ext2,2}^{(5)} \big]_{sym}^{(5)}, \ha{v} (t) \Big) (t,k) \notag \\
& + \La_{\vp}^{(7)} \Big(i \frac{3}{8} \la_1 \big[ [ 2 \sum_{j=2}^6 \ti{L}_{j, \vp}^{(3)} \chi_{>L}^{(3)}]_{ext1,1}^{(7)} \big]_{sym}^{(7)}, 
\ha{v} (t) \Big) (t,k) \notag \\
& + \La_{\vp}^{(7)} \Big( -i \frac{3}{8} \la_1 \big[ [ \sum_{j=2}^6 \ti{L}_{j, \vp}^{(3)} \chi_{>L}^{(3)}]_{ext1,2}^{(7)} \big]_{sym}^{(7)}, 
\ha{v} (t) \Big) (t,k).
\end{align}
By Proposition~\ref{prop_NF11} with $N=2$ and 
$\ti{m}^{(2N+1)}= \sum_{j=1}^6 \ti{L}_{j, \vp}^{(5)} \chi_{>L}^{(5)}$, we have
\begin{align}
& \p_t \La_{\vp}^{(5)}  \Big( \sum_{j=1}^{6} \ti{L}_{j, \vp}^{(5)} \chi_{>L}^{(5)}, \ha{v}(t) \Big) (t,k) 
=\Lambda_{\vp}^{(5)}  \Big(- i \sum_{j=1}^6 \ti{L}_{j, \vp}^{(5)} \Phi_{\vp}^{(5)} \chi_{>L}^{(5)}, \ha{v}(t) \Big) (t,k) \notag \\
&  \hspace{0.5cm} + \sum_{j=1}^6  \Lambda_{\vp}^{(7)} (i \ti{M}_{j, \vp}^{(7)} , \ha{v}(t) ) (t,k)
+ \sum_{j=1}^2 \La_{\vp}^{(9)} (i \ti{M}_{j, \vp}^{(9)}, \ha{v} (t) )(t, k).    \label{eq31}
\end{align}
By \eqref{eq370}--\eqref{eq31} and Lemma \ref{Le31}, we conclude \eqref{NF21}.
\end{proof}


\section{cancellation property}

In Lemma~\ref{lem_pwb2} and Lemmas \ref{lem_mle}--\ref{lem_nl10}, 
we show all multipliers $\{ M_{j, \vp}^{(N)}  \}$ except $M_{8, \vp}^{(5)}$ have no derivative loss. 
On the other hand, $M_{8, \vp}^{(5)}$ has one derivative loss. More precisely, for 
$(k_1, k_2, k_3, k_4, k_5) \in \supp \chi_{NR1}^{(5)} \chi_{H1,1}^{(5)} \chi_{R1}^{(5)}$, we have 
$16^2 \max_{j=1,2,3,4} |k_j| < |k_5| $ and
\begin{equation*}
| M_{8, \vp}^{(5)}  | \sim  |k_1-k_2|^{-1} |k_5|. 
\end{equation*}
However, since $M_{8, \vp}^{(5)}$ is the resonant part, we cannot apply the normal form reduction to it. 
To overcome this difficulty, we use a kind of cancellation property. 
In the following proposition, we compute the symmetrization of $M_{8, \vp}^{(5)} $ and show  that it has no derivative loss.

\begin{prop} \label{prop_res}
It follows that 
\begin{equation} \label{res1}
|\ti{M}_{8, \vp}^{(5)}| \lesssim \Big[ \frac{ \max_{j=1,2,3,4} \{ |k_j| \}  }{|k_1-k_2|} \, \chi_{NR1}^{(5)} \chi_{H1,1}^{(5)} \chi_{R1}^{(5)}  \Big]_{sym}^{(5)}. 
\end{equation}
\end{prop}
\begin{proof}
Put $M:= \chi_{NR1}^{(5)} \chi_{H1,1}^{(5)} \chi_{R1}^{(5)}$. 
For $(k_1, k_2, k_3, k_4,  k_5) \in \supp \chi_{H1,1}^{(5)} $, it follows that 
\begin{equation*}
\chi_{NR1}^{(5)} \chi_{R1}^{(5)} = 
\begin{cases}
1 \hspace{0.2cm} \text{when} \hspace{0.2cm} (k_1-k_2) (k_3-k_4) \neq 0, ~ k_1-k_2+k_3-k_4=0 \\
0 \hspace{0.2cm} \text{otherwise}
\end{cases}
\end{equation*}
which leads 
\begin{equation*}
M(k_1, k_2, k_3, k_4, k_5) = M(k_3, k_4, k_1, k_2, k_5). 
\end{equation*}
Thus, it follows that 
\begin{align}
& \ti{M}_{8, \vp}^{(5)}= \Big[ \frac{q_1^{(3)}}{ \Phi_0^{(3)} } (k_1, k_2,  k_{3,4,5}) q_1^{(3)} (k_3, k_4, k_5) \, 4 M  \Big]_{sym}^{(5)} \notag \\
&= \Big[ \big\{ \frac{q_1^{(3)}}{ \Phi_0^{(3)} } (k_1, k_2, k_{3,4,5} ) q_1^{(3)} (k_3, k_4, k_5)+ 
\frac{ q_1^{(3)} }{ \Phi_0^{(3)} }(k_3,  k_4, k_1-k_2+k_5) q_1^{(3)} (k_1, k_2, k_5) \big\} 2 M   \Big]_{sym}^{(5)}\notag \\
& =  \Big[ \frac{L(k_1, k_2, k_3, k_4, k_5) +L(k_3, k_4, k_1, k_2, k_5) }{ \Phi_0^{(3)} (k_1, k_2, k_{3,4,5}) \Phi_0^{(3)}
 (k_3,  k_4, k_{1}-k_2+k_5) } 
\, M  \Big]_{sym}^{(5)} \label{res11}
\end{align} 
where 
\begin{equation*}
L(k_1, k_2, k_3, k_4, k_5)= 2q_1^{(3)} (k_1, k_2, k_{3,4,5}) q_1^{(3)} (k_3,  k_4,  k_5) \Phi_{0}^{(3)} (k_3,  k_4, k_1-k_2+k_5). 
\end{equation*}
By a direct computation, it follows that 
\begin{equation*}
L(k_1, k_2, k_3, k_4, k_5)= 2  \la_5^2 (k_3-k_4) k_5^{7}+ \sum_{l=1}^7 (k_3-k_4) r_l(k_1, k_2, k_3,  k_4) k_{5}^{7-l}
\end{equation*}
where each $r_l$ is a polynomial of degree $l$ for $l=1, \dots, 7$ and 
\begin{align} \label{res14}
& L(k_1, k_2, k_3, k_4, k_5)+ L(k_3,k_4, k_1,  k_2,  k_5)  = 2  \la_5^2 (k_1-k_2+k_3-k_4) k_5^{7}  \notag \\
& \hspace{1.5cm} + \sum_{l=1}^7 \{ (k_3-k_4) r_l(k_1, k_2, k_3,  k_4) + (k_1-k_2) r_l(k_3, k_4, k_1, k_2) \} k_{5}^{7-l}.
\end{align} 
Since $(k_1, k_2, k_3,  k_4,  k_5) \in \supp M$ leads that $2 \la_5^2 (k_1-k_2+k_3-k_4) k_5^7=0$, 
by \eqref{res14}, we have
\begin{equation} \label{res12}
 \big| (L(k_1, k_2, k_3, k_4, k_5) + L(k_3, k_4, k_1, k_2, k_5) ) M \big|  
\lesssim |k_3-k_4| \max_{j=1,2,3,4} \{ |k_j| \} |k_5|^6 \,  M.
\end{equation}
By Lemma~\ref{Le1}, 
\begin{equation} \label{res13}
| \Phi_0^{(3)} (k_1, k_2, k_{3,4,5}) \Phi_0^{(3)}  (k_3, k_4,k_1-k_2+k_5) | \, M 
\gtrsim |k_1-k_2| |k_3-k_4| |k_5|^6 \, M.
\end{equation}
Collecting \eqref{res11}, \eqref{res12} and \eqref{res13}, we obtain \eqref{res1}. 
\end{proof}

\section{multilinear estimates}
In this section, we present some multilinear estimates in order to prove main estimates stated in Section 7. 

\begin{lem} \label{lem_cnl1}
Let $s \ge 1$. Then, for any $m \in \{1,2,3\}$ and $L>0$, it follows that 
\begin{align}
& \Big\| \sum_{k=k_{1,2,3}} \big[ \langle k_2-k_3  \rangle^{-1} \langle k_{\max}  \rangle^{-1} \chi_{H2,2}^{(3)} \big]_{sym1}^{(3)} 
\prod_{l=1}^3  |\ha{v}_l(k_l)|  \Big\|_{l_s^2} 
\lesssim \| v_m \|_{H^{s-2}} \prod_{l \in \{ 1,2,3 \} \setminus \{ m \}} \| v_l  \|_{H^s}, \label{cnl1} \\
& \Big\| \sum_{k=k_{1,2,3}} \big[ \langle k_1-k_2  \rangle^{-1} \langle k_{\max}  \rangle^{-1} \chi_{H1,1}^{(3)} \big]_{sym1}^{(3)} 
\prod_{l=1}^3  |\ha{v}_l(k_l)|  \Big\|_{l_s^2} 
\lesssim  \| v_2 \|_{H^{s-2}} \| v_1 \|_{H^s} \| v_3  \|_{H^s}, \label{cnl4} \\
& \Big\| \sum_{k=k_{1,2,3}} \langle k_1-k_2 \rangle^{-1} \langle k_3-k_2  \rangle^{-1} \, \chi_{H3}^{(3)}
 \prod_{l=1}^3  |\ha{v}_l(k_l)|  \Big\|_{l_s^2} 
\lesssim \| v_m \|_{H^{s-2}} \prod_{l \in \{ 1,2,3 \} \setminus \{ m \}} \| v_l  \|_{H^s}, \label{cnl2} \\
& \Big\| \sum_{k=k_{1,2,3}} \langle k_1-k_2  \rangle^{-1} \langle k_3-k_2  \rangle^{-1} \, \chi_{H3}^{(3)} \chi_{>L}^{(3)} 
\prod_{l=1}^3  |\ha{v}_l(k_l)|  \Big\|_{l_s^2} 
\lesssim L^{-2} \prod_{i=1}^3 \| v_l  \|_{H^s}. \label{cnl3} 
\end{align}
\end{lem}
\begin{proof}
First, we prove \eqref{cnl1}. It suffices to show 
\begin{equation} \label{cnl11}
\Big\|  \sum_{k=k_{1,2,3}} \langle k_{1,2,3}  \rangle^{s} \langle k_2-k_3  \rangle^{-1} \langle k_{\max} \rangle 
\prod_{l=1}^3 |\ha{v}_l(k_l)|   \Big\|_{l^2} \lesssim \prod_{l=1}^3 \|  v_l \|_{H^s}
\end{equation}
for any $\{ v_l \}_{l=1}^3 \subset H^s(\T)$. Since 
\begin{equation*}
\langle k_{1,2,3} \rangle^{s} \langle k_2-k_3  \rangle^{-1} \langle k_{\max} \rangle \, \chi_{H2,2}^{(3)}
\lesssim \langle k_1 \rangle^s \langle k_2-k_3  \rangle^{-1} \langle k_2 \rangle^{1/2} \langle k_3 \rangle^{1/2},
\end{equation*}
by the Young and the H\"{o}lder inequalities, we have \eqref{cnl11}. 

Next, we prove \eqref{cnl4}. It suffices to show 
\begin{equation} \label{cnl41}
\Big\|  \sum_{k=k_{1,2,3}} \langle k_{1,2,3}  \rangle^{s} \langle k_1-k_2 \rangle^{-1} \langle k_2 \rangle^2 \langle k_3 \rangle^{-1} \, \chi_{H1,1}^{(3)}
\prod_{l=1}^3 |\ha{v}_l(k_l)|  \Big\|_{l^2} \lesssim \prod_{l=1}^3 \| v_l \|_{H^s}
\end{equation}
for any $\{ v_l \}_{l=1}^3 \subset H^s(\T)$. Since 
\begin{equation*}
\langle k_{1,2,3} \rangle^{s} \langle k_1-k_2  \rangle^{-1} \langle k_2 \rangle^2 \langle k_3  \rangle^{-1} \, \chi_{H1,1}^{(3)}
 \lesssim \langle k_1-k_2 \rangle^{-1} \langle k_2  \rangle \langle k_3  \rangle^s, 
\end{equation*}
by the H\"{o}lder and the Young inequalities, the left hand side of \eqref{cnl41} is bounded by 
\begin{align*}
\big\| \{ \langle \cdot  \rangle^{-1} (  |\ha{v}_1| \check{*} \langle \cdot \rangle |\ha{v}_2|  )  \} * (\langle \cdot \rangle^s |\ha{v}_3|) \big\|_{l^2}  &
\lesssim  \| |\ha{v}_1| \check{*} \langle \cdot  \rangle |\ha{v}_2| \|_{l^2} \| \langle \cdot \rangle^s |\ha{v}_3| \|_{l^2} \\
&  \lesssim \| v_1 \|_{H^{1/2+}} \| v_2 \|_{H^1} \| v_3 \|_{H^s}.
\end{align*}

Next, we prove \eqref{cnl2}. It suffices to show 
\begin{equation} \label{cnl21}
\Big\| \sum_{k=k_{1,2,3}} \langle k_{1,2,3} \rangle^{s} \langle k_1-k_2  \rangle^{-1} \langle k_3-k_2 \rangle^{-1} \langle k_{\max} \rangle^2
\, \chi_{H3}^{(3)} 
\prod_{l=1}^3 |\ha{v_l} (k_l) | \Big\|_{l^2} \lesssim \prod_{l=1}^3 \| v_l \|_{H^s} 
\end{equation}
for any $\{ v_l  \}_{l=1}^3 \subset H^s(\T)$. 
By symmetry, we assume that $|k_1-k_2| \le |k_3-k_2|$ holds. 
Then, it follows that 
\begin{equation*}
\langle k_{1,2,3} \rangle^{s} \langle k_1-k_2 \rangle^{-1} \langle k_3-k_2 \rangle^{-1} \langle k_{\max} \rangle^2 \, \chi_{H3}^{(3)} 
\lesssim \langle k_1-k_2  \rangle^{-2} \langle k_1 \rangle \langle k_2 \rangle  \langle k_3 \rangle^s
\end{equation*}
so that \eqref{cnl21} holds by the Young and H\"{o}lder inequalities. 

Furthermore, by $\chi_{>L}^{(3)} \lesssim L^{-2} \langle k_{\max} \rangle^2 $ and \eqref{cnl21}, 
the left hand side of \eqref{cnl3} is bounded by 
\begin{equation*}
L^{-2} \big\| \sum_{k=k_{1,2,3}} \langle k_1-k_2 \rangle^{-1} \langle k_3-k_2 \rangle^{-1} 
\langle k_{\max} \rangle^2 \, \chi_{H3}^{(3)} \prod_{l=1}^3 |\ha{v}_l(k_l)| \big\|_{l_s^2}
 \lesssim L^{-2} \prod_{l=1}^3 \| v_l \|_{H^s}. 
\end{equation*}
\end{proof}

\begin{lem} \label{lem_cnl3}
Let $s \ge 1$. Then, for $\{ v_l \}_{l=1}^3 \subset H^{s} (\T)$, it follows that 
\begin{align} 
& \Big\|  \sum_{k=k_{1,2,3}} [ \langle k_1 \rangle^2 \chi_{R1}^{(3)} ]_{sym1}^{(3)}  \prod_{l=1}^3 |\ha{v}_l(k_l)|  \Big\|_{l_{s}^2} 
\lesssim \prod_{l=1}^3 \| v_l \|_{H^s}, \label{cnl5} \\
& \Big\|  \sum_{k=k_{1,2,3}} \langle k_{\max} \rangle^2 \chi_{R2}^{(3)} \prod_{l=1}^3 |\ha{v}_l(k_l)|  \Big\|_{l_{s}^2} 
\lesssim \prod_{l=1}^3 \| v_l \|_{H^s}. \label{cnl6}
\end{align}
\end{lem}
\begin{proof}
Firstly, we prove \eqref{cnl5}. It suffices to show 
\begin{equation} \label{cnl51}
\Big\| \sum_{k=k_{1,2,3}} \langle k_1 \rangle^2 \langle k_3 \rangle^s \chi_{R1}^{(3)} \prod_{l=1}^3 |\ha{v}_l(k_l)| \Big\|_{l^2} 
\lesssim \prod_{l=1}^3 \| v_l \|_{H^s}. 
\end{equation}
By the Schwarz inequality, the left hand side of \eqref{cnl51} is bounded by 
\begin{equation*}
\Big\| \Big( \sum_{k_1 \in \Z} \langle k_1  \rangle |\ha{v}_1 (k_1)| \langle k_1 \rangle |\ha{v}_2 (k_1)| \Big) \, 
\langle k  \rangle^s |\ha{v}_3 (k)| \Big\|_{l^2}
\leq  \| v_1 \|_{H^1} \| v_2 \|_{H^1} \| v_3 \|_{H^s}. 
\end{equation*}
Secondly, we prove \eqref{cnl6}. 
By the H\"{o}lder inequality and the continuous embedding $l^2 \hookrightarrow l^6$, the left hand side of \eqref{cnl6} is bounded by 
\begin{equation*}
\big\| \langle k \rangle^{s+2}  \prod_{l=1}^3  |\ha{v}_l(k)|  \big\|_{l^2} 
\le \prod_{l=1}^3 \| \langle k \rangle^{s/3+2/3} |\ha{v}_l(k)| \|_{l^6}
\lesssim \prod_{l=1}^3 \| v_l \|_{H^{s/3+2/3}} \le \prod_{l=1}^3 \| v_l \|_{H^s}. 
\end{equation*}
Here we used $s  \ge 1$ in the last equality. 
\end{proof}

\section{pointwise upper bounds}
In this section, we state the pointwise upper bounds of some multipliers $L_{j, \vp}^{(2N+1)}$ and $M_{j, \vp}^{(2N+1)}$ 
defined in Proposition~\ref{prop_NF2}. Now, we put 
\begin{align*}
& J_1=\{ (1,2), (1,3), (1,5), (2,1), (2,2), (2,3), (2,4), (2,5), (2,6) \},  \\
& J_2=\{  (1,4), (1,6)   \},  \hspace{0.3cm} J_3=\{ (1,1) \}. 
\end{align*}
\begin{lem} \label{lem_pwb1}
Let $f \in L^2(\T)$ 
and $L_{j, f}^{(2N+1)}$ with $(N, j) \in J_1 \cup J_2 \cup J_3 $ be as in Proposition~\ref{prop_NF2}. 
Then, the following hold for $L \gg \max \{1, | \la_5| E_1(f) \}$: \\
(I) When $(N, j) \in J_1$, we have
\begin{align} 
|L_{j,f}^{(2N+1)} \chi_{>L}^{(2N+1)} | \lesssim \langle k_{1, \dots, 2N+1} \rangle^{-1} \langle k_{\max}  \rangle^{-1} 
\chi_{>L}^{(2N+1)}. \label{pwb1} 
\end{align}
(II) It follows that 
\begin{align} 
& |L_{1, f}^{(3)} \chi_{>L}^{(3)} | \lesssim \langle k_1-k_2 \rangle^{-1} \langle k_{\max} \rangle^{-1} \chi_{H1,1}^{(3)} \chi_{>L}^{(3)}, 
\label{pwb01}  \\
& |L_{4, f}^{(3)} \chi_{>L}^{(3)} | \lesssim \langle k_2-k_3 \rangle^{-1} \langle k_{\max} \rangle^{-1} \chi_{H2,2}^{(3)} \chi_{>L}^{(3)}, 
\label{pwb02}  \\
& |L_{6, f}^{(3)} \chi_{>L}^{(3)} | \lesssim  \langle k_1-k_2  \rangle^{-1} \langle k_3-k_2  \rangle^{-1} \chi_{H3}^{(3)}  \chi_{>L}^{(3)}. \label{pwb03} 
\end{align}
\end{lem}

\begin{rem} \label{rem_pwb11}
By Lemma~\ref{lem_pwb1} and the definition of the multipliers, we can easily check that 
\begin{equation*} 
|L_{j, f}^{(2N+1)} \chi_{>L}^{(2N+1)} | \lesssim 1, \hspace{0.5cm} 
|L_{j, f}^{(2N+1)} \Phi_f^{(2N+1)} \chi_{>L}^{(2N+1)}| \lesssim \langle k_{\max} \rangle^2 
\end{equation*}
with $(N,j) \in J_1 \cup J_2 \cup J_3 $ for $L \gg \max\{1, |\la_5| E_1(f)  \}$. 
Thus, we can apply Proposition~\ref{prop_NF11} with $\ti{m}^{(2N+1)}=\ti{L}_{j, f}^{(2N+1)} \chi_{>L}^{(2N+1)}$. 
\end{rem}
\begin{proof}[Proof of Lemma~\ref{lem_pwb1}] 
(I) We prove \eqref{pwb1} for $(N, j ) \in J_1$. \\
(Ia) Estimates of $L_{2, f}^{(3)} \chi_{>L}^{(3)}$, $L_{3,f}^{(3)} \chi_{>L}^{(3)}$ and $L_{5, f}^{(3)} \chi_{>L}^{(3)}$: 
By $|Q_1^{(3)}|  \lesssim \langle k_{\max} \rangle^{2} \chi_{NR1}^{(3)}$ and (ii) of Lemma~\ref{Le1}, we have  
\begin{equation*}
|L_{2, f}^{(3)} \chi_{>L}^{(3)}| \lesssim \langle k_{\max} \rangle^{-2}\chi_{>L}^{(3)} , \hspace{0.5cm} 
|L_{5,f}^{(3)} \chi_{>L}^{(3)} | \lesssim \langle k_{\max} \rangle^{-2} \chi_{>L}^{(3)}.
\end{equation*}
Since $(k_1, k_2, k_3) \in \supp \chi_{H2,1}^{(3)} $ implies that $|k_2-k_3| \sim |k_{1,2,3}|$, 
by (iii) of Lemma~\ref{Le1}, we have 
$|L_{3, f}^{(3)} \chi_{>L}^{(3)} | \lesssim \langle k_{1,2,3} \rangle^{-1} \langle k_{\max} \rangle^{-1} \chi_{>L}^{(3)} $. 

Let $\chi^{(5)}$ be a characteristic function such that $\supp \chi^{(5)} \subset \supp [\chi_{H1,1}^{(3)}]_{ext1,1}^{(5)} $. 
Then, by $ \chi^{(5)}= [\chi_{H1,1}^{(3)}]_{ext1,1}^{(5)}  \chi^{(5)}$, 
$\chi_{NR1}^{(5)}=[ \chi_{NR1}^{(3)} ]_{ext1,1}^{(5)} \chi_{NR1}^{(5)}$, Remark~\ref{rem_sym} and Lemma~\ref{Le1},  
it follows that 
\begin{align}
& \Big| \Big[  \frac{q_1^{(3)}}{ \Phi_{f}^{(3)} } \chi_{>L}^{(3)} \Big]_{ext1,1}^{(5)} [ q_1^{(3)} ]_{ext2,1}^{(5)} \chi_{NR1}^{(5)} \chi^{(5)} \Big| 
\notag \\
& = \Big[ \Big| \frac{q_1^{(3)}}{ \Phi_{f}^{(3)} } \chi_{NR1}^{(3)} \chi_{H1,1}^{(3)} \chi_{>L}^{(3)} \Big| \Big]_{ext1,1}^{(5)} [| q_1^{(3)} |]_{ext2,1}^{(5)} \chi_{NR1}^{(5)} \chi^{(5)} \notag \\
& \lesssim 
\langle k_1-k_2 \rangle^{-1} \langle k_{3,4,5} \rangle^{-1} \max_{j=3,4,5} \{ |k_j|^2 \} \, \chi_{NR1}^{(5)} \chi^{(5)} .\label{pwes11} 
\end{align}
Similarly, we have 
\begin{equation}
\Big| \Big[  \frac{q_1^{(3)}}{ \Phi_{0}^{(3)} } \Big]_{ext1,1}^{(5)} [ q_1^{(3)} ]_{ext2,1}^{(5)} \chi_{NR1}^{(5)} \chi^{(5)} \Big| 
\lesssim 
\langle k_1-k_2 \rangle^{-1} \langle k_{3,4,5} \rangle^{-1} \max_{j=3,4,5} \{ |k_j|^2 \} \, \chi_{NR1}^{(5)} \chi^{(5)}. \label{pwes12}
\end{equation}
(Ib) Estimate of $L_{1, f}^{(5)} \chi_{>L}^{(5)}$: 
$(k_1, k_2,  k_3, k_4, k_5) \in \supp \chi_{H1,1}^{(5)} (1-\chi_{R1}^{(5)}) (1-\chi_{R3}^{(5)})$ leads that 
either \eqref{rel3}, \eqref{rel4} or \eqref{rel5} holds and $|k_{3,4,5}| \sim  |k_5|=k_{\max}$. 
Thus, by  Lemma~\ref{Le2}, it follows that 
\begin{align*}
& |\Phi_f^{(5)} \chi_{NR1}^{(5)} \chi_{H1,1}^{(5)} (1-\chi_{R1}^{(5)}) (1-\chi_{R3}^{(5)}) \chi_{>L}^{(5)} | \\
& \hspace{1cm} \gtrsim |k_1-k_2+k_3-k_4| |k_5|^3 \chi_{NR1}^{(5)} \chi_{H1,1}^{(5)} (1-\chi_{R1}^{(5)}) (1-\chi_{R3}^{(5)}) \chi_{>L}^{(5)}. 
\end{align*}
Hence, by \eqref{pwes12}, we have $|L_{1, f}^{(5)} \chi_{>L}^{(5)}| \lesssim \langle k_{\max} \rangle^{-2} \chi_{>L}^{(5)} $. \\
(Ic) Estimate of $L_{2, f}^{(5)} \chi_{>L}^{(5)}$: 
$(k_1, k_2, k_3, k_4, k_5) \in \supp \chi_{NR(1,1)}^{(5)} (1-\chi_{H1,1}^{(5)}) \chi_{A2}^{(5)}$ implies that 
\eqref{rel3} holds and $|k_{3,4,5}| \sim |k_5| = k_{\max}$. 
Thus, by \eqref{pwes12} and Lemma~\ref{Le2}, we have 
$|L_{2, f}^{(5)} \chi_{>L}^{(5)}  | \lesssim \langle k_{\max} \rangle^{-2} \chi_{>L}^{(5)}$. \\
(Id) Estimate of $L_{4, f}^{(5)} \chi_{>L}^{(5)} $: 
$(k_1, k_2, k_3, k_4, k_5) \in \supp \chi_{NR(3,1)}^{(5)} \chi_{A1}^{(5)}$ leads that 
\eqref{rel21} holds and $|k_4-k_5| \sim |k_{3,4,5}| \sim |k_{1,2,3,4,5}|$.
Thus, by \eqref{pwes11} and Lemma~\ref{Le3}, we have 
$|L_{4, f}^{(5)} \chi_{>L}^{(5)}  | \lesssim \langle k_{1,2,3,4,5}  \rangle^{-2} \langle k_{\max} \rangle^{-1} \chi_{>L}^{(5)}$. \\
(Ie) Estimate of $L_{5, f}^{(5)} \chi_{>L}^{(5)} $: For $(k_1, k_2, k_3, k_4, k_5) \in \supp \chi_{NR(4,1)}^{(5)} \chi_{A3}^{(5)} $, 
it follows that
\begin{equation*}
16 \max\{ |k_1|, |k_2| \} < |k_{3,4,5}| \le 33 |k_3| < \frac{33}{32} \min\{ |k_4|, |k_5|  \}
\end{equation*}
which leads \eqref{rel22} and $|k_{1,2,3,4,5}| \sim |k_{3,4,5}|$. 
Thus, by \eqref{pwes11} and Lemma~\ref{Le3}, we have 
$|L_{5, f}^{(5)} \chi_{>L}^{(5)}  | \lesssim \langle k_{1,2,3,4,5}  \rangle^{-1} \langle k_{\max} \rangle^{-1} \chi_{>L}^{(5)}$. \\
(If) Estimates of $L_{3, f}^{(5)} \chi_{>L}^{(5)}$ and $L_{6, f}^{(5)} \chi_{>L}^{(5)}$: 
By \eqref{pwes11} and (ii) of Lemma~\ref{Le3}, 
we have $|L_{j, f}^{(5)} \chi_{>L}^{(5)} | \lesssim \langle k_{\max} \rangle^{-2} \chi_{>L}^{(5)} $ with $j=3,6$. 

Therefore, we obtain \eqref{pwb1} for $(N, j) \in J_1$. \\
(II) By $|Q_1^{(3)}| \lesssim \langle k_{\max} \rangle^{2} \chi_{NR1}^{(3)} $ and Lemma~\ref{Le1}, 
we immediately get \eqref{pwb01}--\eqref{pwb03}. 
\end{proof}

Next, we give pointwise upper bounds of some multipliers $M_{j, \vp}^{(2N+1)}$ defined in Proposition~\ref{prop_NF2}. 

\begin{lem} \label{lem_pwb2}
Let $f \in L^2(\T)$, $s \ge 1$ and $M_{j, f}^{(2N+1)}$ be as in Proposition~\ref{prop_NF2}. 
Then, the following hold for $L \gg \max\{ 1, |\ga| E_1(f) \}$: 
\begin{align} 
& \langle k_{1,2,3,4,5}  \rangle^s |\ti{M}_{8, f}^{(5)}| \notag \\
&  \lesssim 
\big[ \min\{ \langle k_1-k_2 \rangle^{-1}, \langle k_3-k_4  \rangle^{-1}  \} \langle \max_{1 \le j \le 4}\{ |k_j| \} \rangle^{1/2} 
\langle \text{sec}_{1\le j \le 4}\{ |k_j| \}  \rangle^{1/2}  \langle k_5  \rangle^s \big]_{sym}^{(5)} , \label{pwb11} \\
& \langle k_{1,2,3,4,5}  \rangle^s |M_{9, f}^{(5)}| \notag \\
& \lesssim 
\min\{ \langle k_1-k_2 \rangle^{-1}, \langle k_3-k_4  \rangle^{-1}  \} \langle \max_{1 \le j \le 4}\{ |k_j| \} \rangle^{2/3} 
\langle \text{sec}_{1\le j \le 4}\{ |k_j| \}  \rangle^{2/3}  \langle k_5  \rangle^s, \label{pwb12} \\
& \langle k_{1,2,3,4,5}  \rangle^s |M_{10, f}^{(5)}| \lesssim \min\{ \langle k_1-k_2 \rangle^{-1}, \langle k_3-k_4  \rangle^{-1}  \} 
 \max_{1 \le j \le 4}\{ |k_j| \}  \langle k_5  \rangle^s,  \notag \\
\end{align}
$|M_{j, f}^{(5)} | \lesssim L $ with $j=11, 12$ and 
\begin{align}
& \langle k_{1,2,3,4,5} \rangle^s  |M_{13, f}^{(5)}| 
\lesssim \langle k_1-k_2 \rangle^{-1} \langle \max \{ |k_1|, |k_2|  \} \rangle^s \langle k_{3,4,5}  \rangle^{-1}
\langle k_4 \rangle \langle k_5  \rangle,  \label{pwb14} \\
& \langle k_{1,2,3,4,5} \rangle^s  |M_{15, f}^{(5)}| 
\lesssim \langle k_1-k_2 \rangle^{-1} \langle \max \{ |k_1|, |k_2|  \} \rangle^s \langle k_{3,4,5}  \rangle^{-1}
\langle k_3 \rangle \langle k_5 \rangle, \label{pwb15} \\
& \langle k_{1,2,3,4,5} \rangle^{s} |M_{14, f}^{(5)}| 
\lesssim \langle k_3 \rangle^s \langle k_4-k_5 \rangle^{-2} \langle k_4 \rangle \langle k_5 \rangle,  \label{pwb16} \\
& \langle k_{1,2,3,4,5} \rangle^s  |M_{16, f}^{(5)}| 
\lesssim \langle k_1\rangle^{s-1} \langle k_2-k_3+k_4-k_5 \rangle^{-1} 
\max\{ \langle k_3 \rangle \langle k_5 \rangle, \langle k_4 \rangle \langle k_5 \rangle \}, \label{pwb17} \\
& \langle k_{1,2,3,4,5} \rangle^s  |M_{17, f}^{(5)}| 
\lesssim \langle k_1\rangle^{s} \langle k_2-k_3+k_4-k_5 \rangle^{-1} 
\max\{ \langle k_4 \rangle, \langle k_5 \rangle \}. \label{pwb18}
\end{align}
\end{lem}

\begin{proof}
Put $n_1= \max_{1\le j \le 4} \{ |k_j| \}$ and $n_2=\text{sec}_{1\le j \le 4} \{ |k_j| \}$. \\
(i) Estimate of $\ti{M}_{8, f}^{(5)}$: Since $(k_1, k_2, k_3, k_4, k_5) \in \supp \chi_{H1,1}^{(5)} \chi_{R1}^{(5)}$ implies 
$ |k_{1,2,3,4,5}| \sim |k_5| =k_{\max} $, $n_1 \sim n_2$ and $|k_1-k_2|=|k_3-k_4|$, 
by Proposition~\ref{prop_res}, we have
\begin{equation*}
\langle k_{1,2,3,4,5} \rangle^{s} |\ti{M}_{8, f}^{(5)}| \lesssim \big[ \min\{ \langle k_1-k_2  \rangle^{-1}, \langle k_3-k_4  \rangle^{-1}  \} 
n_1^{1/2} n_2^{1/2} \langle k_5 \rangle^s \big]_{sym}^{(5)}. 
\end{equation*} 
(ii) Estimate of $M_{9, f}^{(5)}$: For $(k_1, k_2, k_3, k_4, k_5) \in \supp \chi_{H1,1}^{(5)} \chi_{R3}^{(5)}$, it follows that 
\begin{equation*}
|k_{1,2,3,4,5}| \sim |k_{3,4,5}| \sim |k_5|=k_{\max} \lesssim n_1^{4/3}, \hspace{0.3cm} 
n_1 \sim n_2, \hspace{0.3cm} |k_3-k_4| \lesssim |k_1- k_2|.
\end{equation*}
Thus, by \eqref{pwes12}, we have
\begin{equation*}
\langle k_{1,2,3,4,5} \rangle^{s} |M_{9, f}^{(5)}| 
\lesssim \min\{ \langle k_1-k_2 \rangle^{-1}, \langle k_3-k_4 \rangle^{-1}   \}  n_1^{2/3} n_2^{2/3} \langle k_5 \rangle^s. 
\end{equation*}
(iii) Extimate of $M_{10, f}^{(5)}$: For $(k_1, k_2, k_3, k_4, k_5) \in \supp \chi_{NR(1,1)}^{(5)} (1-\chi_{H1,1}^{(5)}) (1-\chi_{R2}^{(5)})$, 
it follows that $ |k_{1,2,3,4,5}| \sim |k_{3,4,5}| \sim |k_5| =k_{\max} \lesssim n_1$ and $|k_3-k_4| \lesssim |k_1-k_2|$ . 
Thus, by \eqref{pwes12}, we have 
\begin{equation*}
\langle k_{1,2,3,4,5} \rangle^{s} |M_{10, f}^{(5)}| \lesssim \min\{ \langle k_1-k_2  \rangle^{-1}, \langle k_3-k_4  \rangle^{-1}  \} 
\, n_1 \langle k_5 \rangle^s. 
\end{equation*} 
(iv) Estimate of $M_{11, f}^{(5)}$: By $\chi_{NR(1,1)}^{(5)}= [ \chi_{H1,1}^{(3)} ]_{ext1,1}^{(5)} \chi_{NR(1,1)}^{(5)}$, 
$\chi_{NR1}^{(5)}= [ \chi_{NR1}^{(3)} ]_{ext1,1}^{(5)} \chi_{NR1}^{(5)}$, Remark~\ref{rem_sym} and (i) of Lemma~\ref{Le1}, 
we have  
\begin{align*}
& \Big| \Big[ \Big( \frac{q_1^{(3)}}{ \Phi_{f}^{(3)} }- \frac{q_1^{(3)}}{\Phi_0^{(3)}} \Big) \chi_{>L}^{(3)} \Big]_{ext1,1}^{(5)} \chi_{NR1}^{(5)} \chi_{NR(1,1)}^{(5)}  \Big| \\
&= \Big[ \Big| \Big( \frac{q_1^{(3)}}{ \Phi_{f}^{(3)} }- \frac{q_1^{(3)}}{\Phi_0^{(3)}} \Big) \chi_{NR1}^{(3)} \chi_{H1,1}^{(3)} \chi_{>L}^{(3)} \Big| \Big]_{ext1,1}^{(5)} \chi_{NR1}^{(5)} \chi_{NR(1,1)}^{(5)} \\
& \lesssim  |\la_5| E_1(f) [ \langle k_{3}  \rangle^{-3} \chi_{H1,1}^{(3)} ]_{ext1,1}^{(5)} \chi_{NR1}^{(5)} \chi_{NR(1,1)}^{(5)} \\
& = |\la_5| E_1(f) \langle k_{3,4,5}  \rangle^{-3} \chi_{NR1}^{(5)} \chi_{NR(1,1)}^{(5)}
\lesssim |\la_5| E_1(f) \langle k_5 \rangle^{-3} \chi_{NR(1,1)}^{(5)}. 
\end{align*}
Thus, by $L \gtrsim |\la_5| E_1(f)$, we have $|M_{11, f}^{(5)}| \lesssim L$. \\
(v) Estimate of $M_{12,f}^{(5)}$: Since $(k_1, k_2, k_3, k_4, k_5) \in \supp [ \chi_{ \le L}^{(3)} ]_{ext1,1}^{(5)} \chi_{NR(1,1)}^{(5)}$ leads that 
$k_{\max}= |k_{5}| \le 8 L/7$, we have $|M_{12, f}^{(5)}| \lesssim L $. \\
(vi) Estimate of $M_{13, f}^{(5)}$ and $M_{15, f}^{(5)}$: 
Since $(k_1, k_2, k_3, k_4, k_5) \in \supp \chi_{NR(3,1)}^{(5)} (1-\chi_{A1}^{(5)}) $ leads that 
\begin{equation*}
|k_{1,2,3,4,5}| \sim |k_{3,4,5}| \lesssim \max \{ |k_1|, |k_2| \}, \hspace{0.5cm} |k_4| \sim |k_5| \sim k_{\max},
\end{equation*}
by \eqref{pwes11}, we have \eqref{pwb14}. 
In a similar manner as above, we get \eqref{pwb15}. \\
(vii) Estimate of $M_{14, f}^{(5)}$: For $(k_1, k_2, k_3, k_4, k_5) \in \supp \chi_{NR(4,1)}^{(5)} (1-\chi_{A3}^{(5)})$ implies that 
\begin{equation*}
|k_{1,2,3,4,5}| \sim |k_{3,4,5}| \lesssim |k_3| \lesssim |k_4| \sim |k_5|, \hspace{0.5cm} 
|k_4-k_5| \lesssim \min \{ |k_1-k_2|,  |k_3| \}. 
\end{equation*} 
Thus, by \eqref{pwes11}, we have \eqref{pwb16}. \\
(viii) Estimate of $M_{16, f}^{(5)}$ and $M_{17, f}^{(5)}$: 
For $j=1,2,3,4$ and $5$, 
$(k_1, k_2, k_3, k_4, k_5) \in \supp \chi_{NR(j,2)}^{(5)}$ leads that $|k_{1,2,3,4,5}| \sim |k_1| $. 
By $\chi_{NR(2,j)}^{(5)}= \chi_{H1,1}^{(3)} (k_{3,4,5}, k_2, k_1) \chi_{NR(j,2)}^{(5)} $, 
$\chi_{NR1}^{(5)}= \chi_{NR1}^{(3)} (k_{3,4,5}, k_2, k_1) \chi_{NR1}^{(5)}$, Remark~\ref{rem_sym} and Lemma~\ref{Le1}, 
we have 
\begin{align*}
& \langle k_{1,2,3,4,5}  \rangle^{s} 
\Big| \Big[ \frac{q_1^{(3)}}{ \Phi_{f}^{(3)} } \chi_{>L}^{(3)} \Big]_{ext1,1}^{(5)} [q_1^{(3)}]_{ext1,2}^{(5)} 
\chi_{NR1}^{(5)} \chi_{NR(j,2)}^{(5)} \Big| \\
& \hspace{0.3cm} \lesssim \langle k_1 \rangle^{s-1} \langle k_2-k_3+k_4-k_5  \rangle^{-1} \max_{j=3,4,5}\{ |k_j|^2 \} \, 
\chi_{NR(j,2)}^{(5)}. 
\end{align*}
Thus, by the definition of $\{ \chi_{NR(j,2)}^{(5)} \}_{j=1}^5$, we get \eqref{pwb17} and \eqref{pwb18}. 
\end{proof}

In the following remark, we state a property of multipliers $L_{j, \vp}^{(2N+1)}$ and $ M_{j, \vp}^{(2N+1)}$ which is needed to 
prove main estimates as in Section 7. 

\begin{rem} \label{rem_pwb3}
For $f, g \in L^2(\T)$ and $L \gg \max \{1,  |\la_5| E_1(f), |\la_5| E_1(g)  \}$, 
by Lemmas~\ref{Le1}, \ref{Le2}, \ref{Le3} and the definition of $\{  L_{j, f}^{(2N+1)} \}$ and $\{ M_{j, f}^{(2N+1)} \}$ as in Proposition~\ref{prop_NF2}, 
each $L_{j, f}^{(2N+1)}$ satisfies 
\begin{align} \label{pwb21}
& |(L_{j, f}^{(2N+1)}-L_{j, g}^{(2N+1)} ) \chi_{>L}^{(2N+1)} | \notag \\
& \hspace{0.3cm}  \lesssim |\la_5|  |E_1(f) -E_1(g)| \, (|L_{j, f}^{(2N+1)} \chi_{>L}^{(2N+1)} |+ |L_{j, g}^{(2N+1)} \chi_{>L}^{(2N+1)} | ). 
\end{align}
and each $M_{j, f}^{(2N+1)}$ satisfies 
\begin{align} \label{pwb22}
|M_{j, f}^{(2N+1)} - M_{j, g}^{(2N+1)}| \lesssim |\la_5| |E_1(f)- E_1(g)| |M_{j, f}^{(2N+1)}|. 
\end{align}
In particular, for $(N, j) \in \{ (1,1), (2,1), (2,8), (2,9), (2,10), (2,12) \}$, it follows that 
$M_{j, f}^{(2N+1)}-M_{j, g}^{(2N+1)}=0$. 
\end{rem}

\section{main estimates}
The main estimates of the proof  of Theorem~\ref{thm_LWP} is as below. 


\begin{prop} \label{prop_main1}
Let $s \ge 1$, $f, g \in L^2(\T)$ and $F_{f, L}$, $G_{f, L}$ be as in Proposition~\ref{prop_NF2}. 
Then, for $v, w \in C([-T, T]: H^s(\T))$ and $L \gg \max\{1, |\la_5| E_1(f), |\la_5| E_1(g)\}$, 
we have 
\begin{align}
& \big\| F_{f, L} (\ha{v})- F_{f, L} (\ha{w})  \big\|_{L_T^{\infty} l_s^2} \lesssim 
L^{-1} (1+\| v \|_{L_T^{\infty} H^s} + \| w \|_{L_T^{\infty} H^s}  )^4 \| v-w \|_{L_T^{\infty} H^s}, \label{mes1} \\
& \big\| G_{f, L} (\ha{v})- G_{f, L} (\ha{w})  \big\|_{L_T^{\infty} l_s^2} \lesssim 
L^{2} (1+\| v \|_{L_T^{\infty} H^s} + \| w \|_{L_T^{\infty} H^s}  )^ 8\| v-w \|_{L_T^{\infty} H^s}. \label{mes2} 
\end{align}
Moreover, it follows that 
\begin{align}
& \big\| F_{f, L} (\ha{v})- F_{g, L} (\ha{v})  \big\|_{L_T^{\infty} l_s^2} \le C_*, \label{mes3} \\
& \big\|  G_{f, L} (\ha{v})- G_{g, L} (\ha{v}) \big\|_{L_T^{\infty} l_s^2} \le C_{*} \label{mes4}
\end{align}
where $C_{*}=C_{*}(v, s, |E_1(f)-E_1(g)|, |\la_5|, T, L) \ge 0 $, $C_{*} \to 0$ when $|E_1(f)-E_1(g)| \to 0$ and 
$C_{*}=0$ when $E_1(f)= E_1(g)$. 
\end{prop}

As a corollary of Propoistion~\ref{prop_main1}, we get the following estimates. 

\begin{cor} \label{cor_mainest}
Let $s \ge 1$. Then there exists a constant $C \ge 1$ such that the following estimates holds for 
$\vp_1, \vp_2 \in H^s(\T)$, $L \gg \max\{1, |\la_5| E_1(\vp_1), |\la_5| E_1(\vp_2) \}$ and 
any solution $u_1 \in C([-T, T]: H^s(\T))$ (resp. $u_2 \in C([-T, T]: H^s(\T))$) to \eqref{4NLS3} with initial data $\vp_1$ (resp. $\vp_2$): 
\begin{align}
\| u_1 \|_{L_T^{\infty} H^s} & \le  \| \vp_1 \|_{H^s}+ CL^{-1} (1 + \| u_1 \|_{L_T^{\infty} H^s })^4 \| u_1 \|_{L_T^{\infty} H^s}  \notag \\
& + C T L^2 (1+\| u_1 \|_{L_T^{\infty} H^s})^8 \| u_1 \|_{L_T^{\infty} H^s(\T)}, \label{mes21} \\
\| u_1-u_2 \|_{L_T^{\infty} H^s(\T)} & \le   \| u_1-u_2 \|_{H^s} + C_*(1+T) \notag \\
&+CL^{-1} (1+ \| u_1 \|_{L_T^{\infty} H^s } + \| u_2 \|_{L_T^{\infty}H^s } )^4 (\| u_1-u_2 \|_{L_T^{\infty}H^s}+C_{*} ) \notag \\
&+CT L^2 (1+ \| u_1 \|_{L_T^{\infty} H^s } + \| u_2 \|_{L_T^{\infty}H^s } )^8 (\| u_1-u_2 \|_{L_T^{\infty}H^s}+C_{*} ) \label{mes22}
\end{align}
where $C_{*}=C_{*}(u_1, s, |E_1(\vp_1)-E_1(\vp_2)|, |\la_5|, T, L) \ge 0 $, $C_{*} \to 0$ when $|E_1(\vp_1)-E_1(\vp_2)| \to 0$ and 
$C_{*}=0$ when $E_1(\vp_1)= E_1(\vp_2)$. 
\end{cor}
\begin{proof}[Proof of Corollary \ref{cor_mainest}]
By proposition~\ref{prop_NF2}, 
$\ha{v}_j(t,k)=e^{-it \phi_{\vp_j}(k)} \ha{u}_j (t,k)$ satisfies (\ref{NF21}) with initial data $\ha{\vp}_j$ 
for each $k \in \Z$ and any $t \in [-T, T]$. Thus, it follows that 
\EQQ{
& \Big[ \ha{v}_1(t',k )+ F_{\vp_1,L} (\ha{v}_1) (t',k)  \Big]_0^t= \int_0^t G_{\vp_1, L} (\ha{v}_1) (t',k) \, dt', \\
& \Big[(\ha{v}_1-\ha{v}_2)(t',k)+ \big(F_{\vp_1,L}(\ha{v}_1 ) - F_{\vp_2,L}(\ha{v}_2  ) \big) (t',k)  \Big]_0^t \\
& \hspace{2.5cm} =\int_0^t \big( G_{\vp_1,L}(\ha{v}_1)-G_{\vp_2,L}(\ha{v}_2)  \big) (t',k)\, dt'
}
for each $k \in \Z$ and any $t \in [-T, T]$, which leads that 
\EQS{
& \| v_1 \|_{L^{\infty}_T H^s} \leq \| \vp_1 \|_{H^s} + 2\| F_{\vp_1, L} (\ha{v}_1) \|_{L_T^{\infty} l_s^2}
+ T \| G_{\vp_1,L}(\ha{v}_1) \|_{L_T^{\infty} l_s^2 },  \label{6es1}\\
& \|v_1-v_2\|_{L^\infty_T H^s} \leq \|\vp_1-\vp_2\|_{H^s} \notag \\ 
& \hspace{0.3cm} +2 \big( \|F_{\vp_1,L}(\ha{v}_1)-F_{\vp_2,L}( \ha{v}_1)\|_{L_T^{\infty} l_s^2} 
+ \| F_{\vp_2, L} (\ha{v}_1) - F_{\vp_2, L} (\ha{v}_2) \|_{L_T^{\infty} l_s^2} \big) \notag \\ 
& \hspace{0.3cm} +T \big( \|G_{\vp_1,L} (\ha{v}_1)-G_{\vp_2,L}(\ha{v}_1) \|_{L_T^{\infty} l_s^2} 
+ \|  G_{\vp_2, L} (\ha{v}_1)- G_{\vp_2, L} (\ha{v}_2)  \|_{L_T^{\infty} l_s^2}  \big). \label{6es2}
}
Applying \eqref{mes1}--\eqref{mes2} to \eqref{6es1} and \eqref{mes1}--\eqref{mes4} to \eqref{6es2},  
we find a constant $C \ge 1$ 
such that 
\EQS{
& \| v_1 \|_{L^\infty_T H^s}  \le  \| \vp_1 \|_{H^s}
+ C L^{-1} (1 +\| v_1\|_{L_T^{\infty} H^s } )^ 4\| v_1\|_{L_T^{\infty} H^s } \notag \\
& \hspace{2cm} +C T L^2(1+\| v_1\|_{L_T^{\infty }H^s } )^8 \| v_1 \|_{L_T^{\infty} H^s}, \label{6es3} \\
& \| v_1- v_2 \|_{L_T^{\infty} H^s}   \le \| \vp_1- \vp_2 \|_{H^s} +(1+T) C_{*} \notag \\
& \hspace{0.5cm} 
+ CL^{-1} (1+ \| v_1\|_{L_T^{\infty} H^s} + \| v_2\|_{L_T^{\infty} H^s})^4 \| v_1 -v_2 \|_{L_T^{\infty} H^s} \notag \\
& \hspace{0.5cm}
+ CT L^2  (1+ \| v_1\|_{L_T^{\infty} H^s} + \| v_2\|_{L_T^{\infty} H^s})^8 \| v_1 -v_2 \|_{L_T^{\infty} H^s} . \label{6es4}
}
By $\| u_1 \|_{L_T^{\infty} H^s}= \|  v_1 \|_{L_T^{\infty} H^s }$ and \eqref{6es3}, we have \eqref{mes21}. 
We notice that it follows that 
\begin{equation} \label{6es5}
\| u_1-u_2 \|_{L_T^{\infty} H^s} \le \| v_1 -v_2 \|_{L_T^{\infty} H^s}+C_{*}, \hspace{0.5cm} 
\| v_1-v_2 \|_{L_T^{\infty}H^s} \le \| u_1- u_2  \|_{L_T^{\infty}H^s} + C_{*} 
\end{equation}
by uniform $l_s^2$-continuity of $\ha{u}_1$ and $\ha{v}_1$ and the dominated convergence theorem. 
By \eqref{6es4}, \eqref{6es5} and $\| u_j \|_{L_T^{\infty} H^s } = \|  v_j \|_{L_T^{\infty} H^s }$ with $j=1,2$, 
we obtain \eqref{mes22}. 
\end{proof}
Before we prove Proposition~\ref{prop_main1}, we prepare some lemmas. 
The proofs of them are based on the multilinear  estimates and the pointwise upper bounds of multipliers 
stated in Section 5 and 6. 

\begin{lem} \label{lem_nes0} 
Let $s \ge 1$, $f, g \in L^2(\T)$ and $L_{j, f}^{(2N+1)}$ with $(N,j) \in J_1 \cup J_2 \cup J_3$ be as in Proposition~\ref{prop_NF2}. 
Then, for any $\{ v_l \}_{l=1}^{2N+1} \subset C([-T, T]: H^s(\T))$ and $L \gg \max \{1, |\la_5| E_{1} (f) , |\la_5| E_1(g) \}$, we have
\begin{align}
& \Big\| \sum_{k=k_{1, \dots, 2N+1}} 
|\ti{L}_{j, f}^{(2N+1)} \chi_{>L}^{(2N+1)} | \, \prod_{l=1}^{2N+1} |\ha{v}_l(t, k_l)| \Big\|_{L_T^{\infty} l_2^s} 
\lesssim L^{-1} \prod_{l=1}^{2N+1} \| v_l \|_{L_T^{\infty} H^s}, \label{ne01} \\
& \big\|  \La_f^{(2N+1)} ( \ti{L}_{j, f}^{(2N+1)} \chi_{>L}^{(2N+1)}, \ha{v}_1) - 
\La_{g}^{(2N+1)} ( \ti{L}_{j, g}^{(2N+1)} \chi_{>L}^{(2N+1)}, \ha{v}_1 )  \big\|_{L_T^{\infty}l_s^2} 
\le C_* \label{ne03}
\end{align}
where $C_*=C_*(v_1,s,|E_1(f)-E_1(g)|,|\la_5|,T) \ge 0$, $C_*\to 0$ when $|E_1(f)-E_1(g)|\to 0$ and 
$C_{*} =0$ when $E_1(f)=E_1(g)$. 
\end{lem}
\begin{proof}
First, we prove \eqref{ne01}. 
By Lemma~\ref{lem_pwb1}, it follows that 
\begin{equation*}  
| \ti{L}_{j, f}^{(2N+1)} \chi_{>L}^{(2N+1)} | \le [ | L_{j, f}^{(2N+1)} \chi_{>L}^{(2N+1)}|  ]_{sym}^{(2N+1)}  
\lesssim \langle k_{\max} \rangle^{-1} \chi_{>L}^{(2N+1)} \lesssim L^{-1}
\end{equation*}
for $(N, j) \in J_1 \cup J_2 \cup J_3 \setminus \{  (1,6) \}$. 
By \eqref{cnl3} and \eqref{pwb03}, we have 
\begin{equation*}
 \Big\| \sum_{k=k_{1,2,3}} | \ti{L}_{6, f}^{(3)} \chi_{>L}^{(3)} | \, \prod_{l=1}^3 |\ha{v}_l(t, k_l)| \Big\|_{L_T^{\infty} l_s^2} 
 \lesssim L^{-2} \prod_{l=1}^3 \| v_l \|_{L_T^{\infty} H^s}.  
 \end{equation*}
Thus, we obtain \eqref{ne01}. 
Next, we prove (\ref{ne03}). A direct computation yields that  
\begin{align*}
& [ \La_{f}^{(2N+1)} (\ti{L}_{j, f}^{(2N+1)} \chi_{>L}^{(2N+1)}, \ha{v}_1 )
- \La_{g}^{(2N+1)} (\ti{L}_{j, g}^{(2N+1)} \chi_{>L}^{(2N+1)}, \ha{v}_1)] (t,k) \\
& = \La_{f}^{(2N+1)} ( ( \ti{L}_{j, f}^{(2N+1)}-\ti{L}_{j,g}^{(2N+1)} ) \chi_{>L}^{(2N+1)}, \ha{v}_1 ) (t,k) \\
& \hspace{0.4cm}  + [ \La_f^{(2N+1)} (\ti{L}_{j,g}^{(2N+1)} \chi_{>L}^{(2N+1)}, \ha{v}_1) 
- \La_{g}^{(2N+1)} (\ti{L}_{j, g}^{(2N+1)} \chi_{>L}^{(2N+1)}, \ha{v}_1 )] (t,k) \\
& =: J_{1,1}^{(2N+1)}(\ha{v}_1) (t,k) + J_{1,2}^{(2N+1)} (\ha{v}_1) (t,k).
\end{align*}
By $|L_{j, f}^{(2N+1)} \chi_{>L}^{(2N+1)} | \lesssim 1$ and \eqref{pwb21}, it follows that 
\begin{equation*}
| ( \ti{L}_{j, f}^{(2N+1)} - \ti{L}_{j,g}^{(2N+1)}) \chi_{>L}^{(2N+1)} | \lesssim |\la_5| |E_1(f)-E_1(g)|
\end{equation*}
which leads that 
\begin{equation*}
\| J_{1,1}^{(2N+1)} (\ha{v}_1) \|_{L_T^{\infty} l_s^2} \lesssim |\la_5| |E_1(f) - E_1(g)| \, \| v_1 \|_{L_T^\infty H^s}^{2N+1} \le C_*.
\end{equation*}
By \eqref{ne01} and Lemma~\ref{lem_go}, we have $\| J_{1,2}^{(2N+1)} (\ha{v}_1)  \|_{L_T^{\infty} l_s^2} \le C_*$.
Therefore, we obtain (\ref{ne03}). 
\end{proof}

Next, we give multilinear estimates for multipliers $M_{j, \vp}^{(2N+1)}$ defined in 
Proposition~\ref{prop_NF2}. We now put
\begin{align*}
K_1=&  \{ (1,1), (2,1), (2,8), (2,9) \}, \\
K_2=&\{  (2,10), (2,11), (2,12), (2,13), (2,14) , (2,15), (2,16), (2,17)  \}, \\
K_3=& \{ (2,4), (2,5), (2,6), (2,7), (3,5), (3,6) \}, \\
K_4=& \{ (2,18), (2,19) \}, \hspace{0.3cm} K_5= \{ (2,2), (2,3), (3,3), (3,4)  \}, \\
K_6=& \{ (3,1), (3,2), (4,1), (4,2) \}.
\end{align*}
\begin{lem} \label{lem_mle}
Let $s \ge 1$, $f, g \in L^2(\T)$ and $M_{j, \vp}^{(2N+1)}$ with $(N, j) \in K_1 \cup K_2$ as be in Proposition~\ref{prop_NF2}. 
Then, for $L \gg \max\{1, |\la_5| E_1(f), |\la_5| E_1(g) \}$ and $\{ v_l \}_{l=1}^{2N+1}  \subset C([-T, T]: H^s(\T))$, 
we have  
\begin{align}
& \Big\| \sum_{k=k_{1, \cdots, 2N+1}} |\ti{M}_{j, f}^{(2N+1)}| \, \prod_{l=1}^{2N+1} |\ha{v}_l(t, k_l)| \Big\|_{L_T^{\infty} l_s^2 } 
\lesssim L \prod_{l=1}^{2N+1} \| v_l \|_{L_T^{\infty} H^s}, \label{mle1} \\
& \big\| \La_f^{(2N+1)} (\ti{M}_{j, f}^{(2N+1)}, \ha{v}_1) - \La_g^{(2N+1)} (\ti{M}_{j, g}^{(2N+1)}, \ha{v}_1) \big\|_{L_T^{\infty} l_s^2} 
\le C_{*} \label{mle2}
\end{align}
where $C_{*}=C_{*}(v_1, s, |E_1(f)-E_1(g) |, |\la_5|, L, T) \ge 0$, $C_{*} \to 0$ when $|E_1(f) -E_1(g)| \to 0$ 
and $C_{*} =0$ when $E_1(f)=E_1(g)$. 
\end{lem}

\begin{proof}
Firstly, we prove \eqref{mle1}. 
Put $n_1= \max_{1 \le j \le 4} \{ |k_j| \}$ and $n_2= \text{sec}_{1\le j \le 4} \{ |k_j| \}$. \\ 
(I) Estimate of $\ti{M}_{8, f}^{(5)}$: 
By \eqref{pwb11}, we need to show 
\begin{equation} \label{mle41}
\Big\| \sum_{k=k_{1,2,3,4, 5}} \min\{ \langle k_1-k_2  \rangle^{-1}, \langle k_3-k_4 \rangle^{-1} \} n_1^{1/2}  n_2^{1/2} 
\langle k_5 \rangle^{s} \, \prod_{l=1}^5 |\ha{v}_l (k_l)|  \Big\|_{l^2} \lesssim \prod_{l=1}^5 \| v_l \|_{H^s}
\end{equation}
for $\{ v_l \}_{l=1}^5 \subset H^{s}(\T)$. We may assume that $|k_1| \le |k_3|$ and $|k_2| \le |k_4|$ hold. 
First, we suppose that $ n_1=|k_4|$ and $n_2=|k_2|$ hold. 
Then, by the H\"{o}lder and the Young inequalities, the left hand side of \eqref{mle41} is bounded by 
\begin{align*}
& \big\| \{ \langle \cdot \rangle^{-1/2} ( |\ha{v}_1| \check{*} \langle \cdot \rangle^{1/2} |\ha{v}_2| ) \}*   
\{ \langle \cdot \rangle^{-1/2} ( |\ha{v}_3| \check{*} \langle \cdot \rangle^{1/2} |\ha{v}_4| ) \}* ( \langle \cdot \rangle^s |\ha{v}_5| )  
\big\|_{l^2} \\
& \lesssim \| |\ha{v}_1| \check{*} \langle \cdot \rangle^{1/2+} |\ha{v}_2|  \|_{l^2} 
\|  |\ha{v}_3| \check{*} \langle \cdot \rangle^{1/2+} |\ha{v}_4) \|_{l^2} 
\| \langle \cdot \rangle^s |\ha{v}_5| \|_{l^2} 
 \lesssim \prod_{l=1}^4 \| v_l \|_{H^{1/2+}} \| v_5 \|_{H^s}. 
\end{align*} 
Similarly, we obtain \eqref{mle41} when $n_1=|k_3|$ and $n_2=|k_1|$. 
Next, we suppose that $n_1= |k_4|$ and $n_2=|k_3|$ hold. 
Then, by the H\"{o}lder and the Young inequalities, the left hand side of \eqref{mle41} is bounded by 
\begin{align*}
& \big\|  ( |\ha{v}_1| \check{*} |\ha{v}_2| ) )*   
\{ \langle \cdot \rangle^{-1} ( \langle \cdot \rangle^{1/2} |\ha{v}_3| \check{*} \langle \cdot \rangle^{1/2} |\ha{v}_4| ) \}
* ( \langle \cdot \rangle^s |\ha{v}_5| ) \big\|_{l^2} \\
& \lesssim \| \ha{v}_1 \|_{l^1} \| \ha{v}_2 \|_{l^1} 
\|  \langle \cdot \rangle^{1/2+} |\ha{v}_3| \check{*} \langle \cdot \rangle^{1/2+} |\ha{v}_4| \|_{l^{\infty}} 
\| \langle \cdot \rangle^s |\ha{v}_5| \|_{l^2} 
 \lesssim \prod_{l=1}^4 \| v_l \|_{H^{1/2+}} \| v_5 \|_{H^s}. 
\end{align*} 
Similarly, we obtain \eqref{mle41} when $n_1=|k_3|$ and $n_2=|k_4|$. 

It suffices to show that 
\begin{equation} \label{mle3}
\Big\| \sum_{k=k_{1,\dots, 2N+1}} |M_{j, f}^{(2N+1)} | \, \prod_{l=1}^{2N+1} |\ha{v}_l (k_l) | \Big\|_{l_s^2} 
\lesssim L \prod_{l=1}^{2N+1} \| v_l \|_{H^s}
\end{equation}
for $\{ v_l \}_{l=1}^{2N+1} \subset H^s(\T)$ and $(N, j) \in K_1 \cup K_2 \setminus \{ (2,8) \}$. \\
(IIa) Estimate of $M_{1,f}^{(3)}$ and $M_{1, f}^{(5)}$: By \eqref{cnl5} and \eqref{cnl6}, we get \eqref{mle3} for $(N,j)=(1,1)$. 
Obviously, we have \eqref{mle3} for $(N, j)=(2,1)$.  \\
(IIb) Estimate of $M_{9, f}^{(5)}$: In a similar manner as the case (I), 
by \eqref{pwb12} and the H\"{o}lder and the Young inequalities, the left hand side of \eqref{mle3} is bounded by 
\begin{align*}
& \Big\| \sum_{k=k_{1,2,3,4,5}} \min\{ \langle k_1-k_2 \rangle^{-1}, \langle k_3-k_4  \rangle^{-1}  \} n_1^{2/3} n_2^{2/3} 
\langle k_5 \rangle^s \, \prod_{l=1}^5  |\ha{v}_l(k_l)|    \Big\|_{l^2} \\ 
& \hspace{0.3cm} \lesssim \prod_{l=1}^4 \| v_l \|_{H^{2/3+}} \| v_5 \|_{H^s}.
\end{align*}
(IIc) By Lemma~\ref{lem_pwb2} and the H\"{o}lder and the Young inequalities, we have \eqref{mle3} for $(N, j) \in K_2$.

Therefore, we obtain \eqref{mle3} for $(N,j) \in K_1 \cup K_2 \setminus \{  (2,8) \}$. 

Secondly, we prove \eqref{mle2}. By a direct computation, it follows that 
\begin{align*}
& \big[ \La_f^{(2N+1)} (\ti{M}_{j, f}^{(2N+1)}, \ha{v}_1) -\La_g^{(2N+1)} (\ti{M}_{j, g}^{(2N+1)}, \ha{v}_1) \big] (t,k) \\
& =\La_f^{(2N+1)} (\ti{M}_{j, f}^{(2N+1)}- \ti{M}_{j, g}^{(2N+1)} , \ha{v}_1) (t,k) \\
& \hspace{0.4cm} +\big[ \La_f^{(2N+1)} (\ti{M}_{j, g}^{(2N+1)}, \ha{v}_1) -\La_g^{(2N+1)} (\ti{M}_{j, g}^{2N+1}, \ha{v}_1)  \big] (t,k) \\
&= :J_{2,1}^{(2N+1)}  (\ha{v}_1)(t,k)+ J_{2,2}^{(2N+1)} (\ha{v}_1) (t,k). 
\end{align*}
By \eqref{pwb22}, $\ti{M}_{8, f}^{(5)}- \ti{M}_{8,g}^{(5)}=0$ and a slight modification of \eqref{mle3}, we have 
\begin{equation*}
\| J_{2,1}^{(2N+1)} (\ha{v}_1) \|_{L_T^{\infty} l_s^2} \lesssim |\la_5| |E_1(f)-E_1(g)| \| v_1 \|_{L_T^{\infty} H^s}^{2N+1} \le C_{*}. 
\end{equation*}
By \eqref{mle1} and Lemma~\ref{lem_go}, we have $\| J_{2,2}^{(2N+1)} (\ha{v}_1) \|_{L_T^{\infty} H^s } \le C_{*}$. 
Therefore, we obtain \eqref{mle2}.  
\end{proof}

\begin{lem} \label{lem_nl10}
Let $s \ge 1$, $f , g \in L^2(\T)$ and $M_{j, f}^{(2N+1)}$ with 
$(N,j) \in K_3 \cup K_4 \cup K_5 \cup K_6$ be as in Proposition~\ref{prop_NF2}. 
Then, for $\{ v_l \}_{l=1}^{2N+1} \subset C([-T, T]: H^s(\T)) $ and $L \gg \max\{ 1, |\la_5| E_1(f), |\la_5| E_1(g) \}$, 
we have
\begin{align}
& \big\| \La_{f}^{(2N+1)} \big(\ti{M}_{j, f}^{(2N+1)}, \ha{v}_1, \dots, \ha{v}_{2N+1} \big) \big\|_{L_T^{\infty} l_s^2} 
\lesssim \prod_{l=1}^{2N+1} \| v_l \|_{L_T^{\infty} H^s}, \label{nes11} \\
& \big\|  \La_{f}^{(2N+1)} (\ti{M}_{j, f}^{(2N+1)}, \ha{v}_1) - \La_g^{(2N+1)} (\ti{M}_{j, g}^{(2N+1)} , \ha{v}_1) \big\|_{L_T^{\infty} l_s^2} 
\le C_* \label{nes13} 
\end{align}
where $C_*=C_*(v_1, s ,|E_1(f)-E_1(g)|, |\la_5|, T) \ge 0$, $C_*\to 0$ when $|E_1(f)-E_1(g)|\to 0$ and 
$C_{*} =0$ when $E_1(f)=E_1(g)$. 
\end{lem}

\begin{proof} 
Put 
\begin{align*}
& \mathcal{N}_{1,f}^{(3)} (\ha{w}_1, \ha{w}_2, \ha{w}_3) (t,k):= 
\La_f^{(3)} (i Q_1^{(3)} , \ha{w}_1, \ha{w}_2, \ha{w}_3 )(t, k), \\
& \mathcal{N}_{2,f}^{(3)} (\ha{w}_1, \ha{w}_2, \ha{w}_3) (t,k):= 
\La_f^{(3)} \big( i Q_1^{(3)} \chi_{H3}^{(3)}  , \ha{w}_1, \ha{w}_2 , \ha{w}_3 \big)(t, k), \\
& \mathcal{N}_{3,f}^{(3)} (\ha{w}_1, \ha{w}_2, \ha{w}_3) (t,k):= 
\La_f^{(3)} \big(i Q_2^{(3)} , \ha{w}_1, \ha{w}_2 , \ha{w}_3 \big)(t, k)
\end{align*}
and
\begin{equation*}
\mathcal{N}_{i, f}^{(3)}(\ha{w}_1) (t, k):= \mathcal{N}_{i, f}^{(3)} (\ha{w}_1, \ha{w}_1, \ha{w}_1) (t, k)
\end{equation*}
for $i=1,2,3$. 
By $|Q_1^{(3)} | \lesssim \langle k_{\max} \rangle^2$ and \eqref{nl11}, it follows that 
\begin{equation} \label{nest11}
\big\| \sum_{k=k_{1,2,3}} |Q_1^{(3)} | \, \prod_{l=1}^3 |\ha{w}_l(t, k_l)| \big\|_{L_T^{\infty} l_{s-2}^2 } 
\lesssim \prod_{l=1}^3 \| w_l \|_{L_T^{\infty} H^s}
\end{equation}
for $\{ w_l \}_{l=1}^3 \subset C([-T, T]: H^s(\T))$, which leads that  
\begin{equation} \label{nest12}
\big\| \mathcal{N}_{1,f}^{(3)} (\ha{w}_1, \ha{w}_2, \ha{w}_3)  \big\|_{L_T^{\infty} l_{s-2}^2 } 
\lesssim \prod_{l=1}^3 \| w_l \|_{L_T^{\infty} H^s}. 
\end{equation}
By \eqref{nest11} and Lemma~\ref{lem_go}, we have 
\begin{equation} \label{nest13}
\| \mathcal{N}_{1,f}^{(3)} (\ha{w}_1)- \mathcal{N}_{1,g}^{(3)} (\ha{w}_1)  \|_{L_T^{\infty} l_{s-2}^2} \le C_{*}(w_1, s, |E_1(f) -E_1(g) |, |\la_5|, T).
\end{equation} 
Moreover, we have 
\begin{equation}
\Big\| \sum_{k=k_{1,2,3}} |Q_1^{(3)}| \chi_{H3}^{(3)} \, \prod_{l=1}^3 |\ha{w}_l (t, k_l)| \Big\|_{L_T^{\infty} l_{s-1}^2} 
\lesssim \prod_{l=1}^3 \| w_l \|_{L_T^{\infty} H^s} \label{nest14} 
\end{equation}
for $\{ w_l \}_{l=1}^3 \subset C([-T, T]: H^s(\T)) $. 
In fact, since
\begin{equation*}
\langle k_{1,2,3}  \rangle^{s-1} |Q_1^{(3)}| \chi_{H3}^{(3)} \lesssim \langle k_1 \rangle^{s-1/3} 
\langle k_2 \rangle^{2/3} \langle k_3 \rangle^{2/3},
\end{equation*}
by the Plancherel theorem and the H\"{o}lder and the Sobolev inequalities, we have
\begin{align*}
& \Big\| \sum_{k=k_{1,2,3}} |Q_1^{(3)}| \chi_{H3}^{(3)} \,  \prod_{l=1}^3 |\ha{w}_l(t, k_l)|  \Big\|_{l_{s-1}^2} \\
& \lesssim \big\|\{ (\langle \cdot \rangle^{s-1/3} |\ha{w}_1(t) |) * ( \langle \cdot \rangle^{2/3} |\ha{w}_3(t) |) \} \check{*} 
(\langle \cdot \rangle^{2/3} |\ha{w}_2 (t)| )  \big\|_{l^2} \\
& = \big\| \mathcal{F}^{-1}[ \langle \cdot  \rangle^{s-1/3} |\ha{w}_1 (t) | ] \, \mathcal{F}^{-1}[ \langle \cdot \rangle^{2/3} |\ha{w}_2 (t)| ] \, \overline{\mathcal{F}^{-1} [ \langle \cdot  \rangle^{2/3} |\ha{w}_3(t)| ] }  \big\|_{L^2} \\
& \le \big\|   \mathcal{F}^{-1}[ \langle \cdot  \rangle^{s-1/3} |\ha{w}_1(t)| ] \big\|_{L^6} 
\big\|  \mathcal{F}^{-1}[ \langle \cdot  \rangle^{2/3} |\ha{w}_2 (t)| ] \big\|_{L^6} 
\big\|  \mathcal{F}^{-1}[ \langle \cdot \rangle^{2/3} |\ha{w}_3 (t)| ]  \big\|_{L^6} \\
& \lesssim \| w_1(t)  \|_{H^s} \|  w_2(t) \|_{H^1} \| w_3 (t) \|_{H^1} 
\end{align*}
for any $t \in [-T, T]$. This leads \eqref{nest14} and 
\begin{equation} \label{nest15}
\big\| \mathcal{N}_{2,f}^{(3)} (\ha{w}_1, \ha{w}_2, \ha{w}_3)  \big\|_{L_T^{\infty} l_{s-1}^2} 
\lesssim \prod_{l=1}^3 \| w_l  \|_{L_T^{\infty} H^s}.
\end{equation}
By \eqref{cnl5} and \eqref{cnl6}, it follows that 
\begin{equation}
\Big\| \sum_{k=k_{1,2,3}} |Q_2^{(3)}|  \, \prod_{l=1}^3 |\ha{w}_l (t, k_l)| \Big\|_{L_T^{\infty} l_{s}^2} 
\lesssim \prod_{l=1}^3 \| w_l \|_{L_T^{\infty} H^s} \label{nest16} 
\end{equation}
for $\{ w_l  \}_{l=1}^3 \subset C([-T, T] : H^s(\T))$, which leads that 
\begin{equation} \label{nest17}
\big\| \mathcal{N}_{3,f}^{(3)} (\ha{w}_1, \ha{w}_2, \ha{w}_3)  \big\|_{L_T^{\infty} l_s^2 } \lesssim \prod_{l=1}^3 \| w_l \|_{L_T^{\infty} H^s}.
\end{equation}
By \eqref{nest14}, \eqref{nest16} and Lemma~\ref{lem_go}, we have 
\begin{align}
& \big\| \mathcal{N}_{2,f}^{(3)} (\ha{w}_1)- \mathcal{N}_{2,g}^{(3)} (\ha{w}_1) \big\|_{L_T^{\infty} l_{s-1}^2} 
\le C_{*} (w_1, s, |E_1(f)-E_1(g)|, |\la_5|, T), \label{nest18} \\
& \big\| \mathcal{N}_{3,f}^{(3)} (\ha{w}_1)- \mathcal{N}_{3,g}^{(3)} (\ha{w}_1) \big\|_{L_T^{\infty} l_{s}^2} 
\le C_{*} (w_1, s, |E_1(f)-E_1(g)|, |\la_5|, T).  \label{nest19}
\end{align}
Here we introduce the following notations:
\begin{align*}
& \mathcal{I}_{i,f}^{(N,j)} (\ha{v}_1) (t,k) 
= \La_f^{(2N+1)} \big( \ti{L}_{j, f}^{(2N+1)} \chi_{>L}^{(2N+1)}, \ha{v}_1 , \dots, \ha{v}_1, 
\mathcal{N}_{i,f}^{(3)} (\ha{v}_{1} ) \big) (t,k), \\
& \mathcal{K}_{i,f}^{(N,j)} (\ha{v}_1) (t,k)
 := \La_f^{(2N+1)} \big( \ti{L}_{j, f}^{(2N+1)} \chi_{>L}^{(2N+1)}, \ha{v}_1, \dots, \ha{v}_{1}, \mathcal{N}_{i,f}^{(3)} (\ha{v}_1 ), \ha{v}_{1} \big) (t,k). 
\end{align*}
(I) Firstly, we prove \eqref{nes11} and \eqref{nes13} for $(N, j) \in K_3$. By the definition, 
\begin{align*}
& \La_f^{(5)} (i \ti{M}_{4, f}^{(5)}, \ha{v}_1) = 2 \sum_{j=2,3,5} \mathcal{I}_{1,f}^{(1,j)} (\ha{v}_1), \hspace{0.3cm}
\La_f^{(5)} ( i \ti{M}_{5,f}^{(5)}, \ha{v}_1) = \sum_{j=2,3,5} \mathcal{K}_{1, f}^{(1,j)} (\ha{v}_1), \\
& \La_f^{(5)} (i \ti{M}_{6, f}^{(5)}, \ha{v}_1) = 2 \sum_{j=4,6} \mathcal{I}_{1,f}^{(1,j)} (\ha{v}_1),  \hspace{0.3cm} 
\La_f^{(5)} (i \ti{M}_{7,f}^{(5)}, \ha{v}_1)= \sum_{j=4,6} \mathcal{K}_{1, f}^{(1,j)} (\ha{v}_1), \\
& \La_f^{(7)} (i \ti{M}_{5, f}^{(7)}, \ha{v}_1) = 3 \sum_{j=1}^6 \mathcal{I}_{1,f}^{(2,j)} (\ha{v}_1),  \hspace{0.3cm}
\La_f^{(7)} (i \ti{M}_{6,f}^{(7)}, \ha{v}_1) = 2 \sum_{j=1}^6 \mathcal{K}_{1, f}^{(2,j)} (\ha{v}_1).
\end{align*}
Thus, it suffices to show that for $i=1$, $(N, j) \in J_1 \cup J_2$ and $\{  v_l \}_{l=1}^{2N+3} \subset C([-T, T]: H^s(\T))$ 
\begin{align}
& \Big\| \sum_{k=k_{1, \dots, 2N+1}} |\ti{L}_{j, f}^{(2N+1)} \chi_{>L}^{(2N+1)} | \prod_{l=1}^{2N} |\ha{v}_l (t, k_l)| \, 
\big| \mathcal{N}_{i, f}^{(3)} (\ha{v}_{2N+1}, \ha{v}_{2N+2}, \ha{v}_{2N+3}) (t, k_{2N+1}) \big| \Big\|_{L_T^{\infty}l_s^2} \notag \\
& \hspace{0.5cm} \lesssim  \prod_{l=1}^{2N+1} \| v_l \|_{L_T^{\infty} H^s}, \label{nest21} \\
& \Big\| \sum_{k=k_{1, \dots, 2N+1}} |\ti{L}_{j, f}^{(2N+1)} \chi_{>L}^{(2N+1)} | \prod_{l=1}^{2N-1} |\ha{v}_l (t, k_l)| \,
 |\ha{v}_{2N+1} (t, k_{2N+1})| \notag \\
& \hspace{1cm} \times \big| \mathcal{N}_{i, f}^{(3)} (\ha{v}_{2N}, \ha{v}_{2N+3}, \ha{v}_{2N+2}) (t, k_{2N})  \big| 
\Big\|_{L_T^{\infty}l_s^2} \lesssim \prod_{l=1}^{2N+3} \|  v_l \|_{L_T^{\infty} H^s}, \label{nest22} \\
&  \big\| \mathcal{I}_{i,f}^{(N,j)} (\ha{v}_1)- \mathcal{I}_{i,g}^{(N, j)} (\ha{v}_1) \big\|_{L_T^{\infty} l_s^2 } \le C_{*}, \label{nest23} \\
& \big\| \mathcal{K}_{i,f}^{(N,j)} (\ha{v}_1)- \mathcal{K}_{i,g}^{(N, j)} (\ha{v}_1) \big\|_{L_T^{\infty} l_s^2 } \le C_{*}. \label{nest24} 
\end{align}
Let $(N, j ) \in J_1 $. Then, by \eqref{pwb1}, it follows that 
\begin{equation} \label{nest25}
|\ti{L}_{j, f}^{(2N+1)} \chi_{>L}^{(2N+1)} | \le [ |L_{j, f}^{(2N+1)} \chi_{>L}^{(2N+1)} | ]_{sym}^{(2N+1)} 
\lesssim \langle k_{1,\dots, 2N+1} \rangle^{-1} \langle k_{\max} \rangle^{-1}. 
\end{equation}
Thus, by \eqref{nl21} and \eqref{nest12}, the left hand side of \eqref{nest21} is bounded by 
\begin{equation*}
\prod_{l=1}^{2N} \| v_l \|_{L_T^{\infty} H^s} 
\big\| \mathcal{N}_{1,f}^{(3)} (\ha{v}_{2N+1}, \ha{v}_{2N+2}, \ha{v}_{2N+3} ) \big\|_{L_T^{\infty} l_{s-2}^2} 
\lesssim \prod_{l=1}^{2N+3} \| v_l \|_{L_T^{\infty} H^s}. 
\end{equation*}
In a similar manner as above, we have \eqref{nest22}. 

By \eqref{cnl1}, \eqref{cnl2}, \eqref{pwb02}, \eqref{pwb03} and \eqref{nest12}, we get \eqref{nest21} and \eqref{nest22} 
for $(N, j) \in J_2$. 

Next, we prove \eqref{nest23} and \eqref{nest24}. We only show \eqref{nest23} since \eqref{nest24} follows in a similar manner. 
A direct computation yields that 
\begin{align*}
& \big[ \mathcal{I}_{1,f}^{(N,j)} (\ha{v}_1)- \mathcal{I}_{1,g}^{(N,j)} (\ha{v}_1)  \big] (t,k) \\
& = \La_f^{(2N+1)} \big( (\ti{L}_{j, f}^{(2N+1)}- \ti{L}_{j, g}^{(2N+1)} ) \chi_{>L}^{(2N+1)}, \ha{v}_1, \dots, \ha{v}_1, 
\mathcal{N}_{1,f}^{(3)} (\ha{v}_1) \big) (t, k) \\ 
& + \big[ \La_f^{(2N+1)} \big( \ti{L}_{j, g}^{(2N+1)} \chi_{>L}^{(2N+1)}, \ha{v}_1, \dots, \ha{v}_1, 
\mathcal{N}_{1,f}^{(3)} (\ha{v}_1) \big) \\
& \hspace{0.5cm} - \La_g^{(2N+1)} \big( \ti{L}_{j, g}^{(2N+1)} \chi_{>L}^{(2N+1)}, \ha{v}_1, \dots, \ha{v}_1, 
\mathcal{N}_{1,f}^{(3)} (\ha{v}_1) \big) \big](t, k) \\
& + \La_g^{(2N+1)} \big( \ti{L}_{j, g}^{(2N+1)} \chi_{>L}^{(2N+1)}, \ha{v}_1, \dots, \ha{v}_1, 
\mathcal{N}_{1,f}^{(3)} (\ha{v}_1)- \mathcal{N}_{1,g}^{(3)}(\ha{v}_1) \big) (t, k) \\
&::= J_{3,1}^{(2N+3)} (\ha{v}_1) (t,k) + J_{3,2}^{(2N+3)} (\ha{v}_1) (t,k) + J_{3,3}^{(2N+3)} (\ha{v}_1) (t,k). 
\end{align*}
By \eqref{pwb21} and a slight modification of \eqref{nest21}, we have 
\begin{equation*}
\| J_{3,1}^{(2N+3)} (\ha{v}_1) \|_{L_T^{\infty} l_s^2 } \lesssim |\la_5| |E_1(f)-E_1(g)| \| v_1 \|_{L_T^{\infty} H^s}^{2N+3} \le C_{*}. 
\end{equation*}
By \eqref{nest21} and Lemma~\ref{lem_go}, we get $\| J_{2,2}^{(2N+3)} (\ha{v}_1)  \|_{L_T^{\infty} l_s^2} \le C_{*}$. 
By \eqref{nl21}, \eqref{cnl1}, \eqref{cnl2}, \eqref{pwb02}, \eqref{pwb03}, \eqref{nest25} and \eqref{nest13}, we have 
\begin{equation*}
\| J_{2,3}^{(2N+3)} (\ha{v}_1) \|_{L_T^{\infty} l_s^2} 
\lesssim \| v_1 \|_{L_T^{\infty} H^s}^{2N} 
\big\| \mathcal{N}_{1,f}^{(3)} (\ha{v}_1)- \mathcal{N}_{1,g}^{(3)} (\ha{v}_1) \big\|_{L_T^{\infty} l_{s-2}^2} \le C_{*}. 
\end{equation*} 
Therefore, we obtain \eqref{nest23}. \\
(II) Secondly, we prove \eqref{nes11} and \eqref{nes13} for $(N,j) \in K_4$. Since
\begin{equation*}
\La_f^{(5)} (i \ti{M}_{18,f}^{(5)}, \ha{v}_1)= 2 \mathcal{I}_{2,f}^{(1,1)} (\ha{v}_1), \hspace{0.3cm} 
\La_f^{(5)} (i \ti{M}_{19,f}^{(5)}, \ha{v}_1)= \mathcal{K}_{1,f}^{(1,1) } (\ha{v}_1), 
\end{equation*}
it suffices to show that
\begin{align}
& \Big\| \sum_{k=k_{1,2,3} } |\ti{L}_{1,1}^{(3)} \chi_{>L}^{(3)} | |\ha{v}_1(t, k_1)| |\ha{v}_2 (t, k_2)| 
\big| \mathcal{N}_{2,f}^{(3)} (\ha{v}_3, \ha{v}_4, \ha{v}_5) (t, k_3)|  \Big\|_{L_T^{\infty} l_s^2} \lesssim \prod_{l=1}^5 \| v_l  \|_{L_T^{\infty} l_s^2}, \label{nest31} \\
& \Big\| \sum_{k=k_{1,2,3} } |\ti{L}_{1,1}^{(3)} \chi_{>L}^{(3)} | |\ha{v}_1(t, k_1)| |\ha{v}_3 (t, k_3)| 
\big| \mathcal{N}_{1,f}^{(3)} (\ha{v}_2, \ha{v}_5, \ha{v}_4) (t, k_2)|  \Big\|_{L_T^{\infty} l_s^2} 
\lesssim \prod_{l=1}^5 \| v_l  \|_{L_T^{\infty} l_s^2}, \label{nest32} \\
& \big\| \mathcal{I}_{2,f}^{(1,1)} (\ha{v}_1)-\mathcal{I}_{2,g}^{(1,1)} (\ha{v}_1) \big\|_{L_T^{\infty} l_s^2} \le C_{*}, \label{nest33} \\
& \big\|  \mathcal{K}_{1,f}^{(1,1)}(\ha{v}_1) - \mathcal{K}_{1,g}^{(1,1)} (\ha{v}_1) \big\|_{L_T^{\infty} l_s^2} \le C_{*} \label{nest34}
\end{align}
for $\{ v_l \}_{l=1}^5 \subset C([-T, T]: H^s(\T))$. By \eqref{pwb01}, it follows that 
\begin{equation*}
\langle k_{1,2,3} \rangle^s |\ti{L}_{1,f}^{(3)}  \chi_{>L}^{(3)}| \lesssim [ \langle k_1-k_2 \rangle^{-1} \langle k_3 \rangle^{s-1} 
\chi_{H1,1}^{(3)} ]_{sym}^{(3)} \lesssim \langle k_{\max}  \rangle^{s-1}. 
\end{equation*}
Thus, by \eqref{nest12}, the left hand side of \eqref{nest31} is bounded by 
\begin{equation*}
\| v_1 \|_{L_T^{\infty} H^s } \| v_2 \|_{L_T^{\infty} H^s} 
\big\| \mathcal{N}_{2,f}^{(3)} (\ha{v}_3, \ha{v}_4, \ha{v}_5) \big\|_{L_T^{\infty} l_{s-1}^2} 
\lesssim \prod_{l=1}^5 \| v_l \|_{L_T^{\infty} H^s}.
\end{equation*}
By \eqref{cnl4}, \eqref{pwb01} and \eqref{nest12}, the left hand side of \eqref{nest32} is bounded by
\begin{equation*}
\| v_1 \|_{L_T^{\infty}H^s } \| v_3 \|_{L_T^{\infty} H^s} 
\big\| \mathcal{N}_{1,f}^{(3)} (\ha{v}_2, \ha{v}_5, \ha{v}_4) \big\|_{L_T^{\infty} l_{s-2}^2 } \lesssim \prod_{l=1}^5 \| v_l \|_{L_T^{\infty} H^s}.  
\end{equation*} 
Therefore, we obtain \eqref{nest31} and \eqref{nest32}. 
In a similar way to the proof of \eqref{nest23}, 
by \eqref{pwb21}, \eqref{nest13}, \eqref{nest18}, Lemma~\ref{lem_go} and a slight modification of \eqref{nest31} and \eqref{nest32}, 
we get \eqref{nest33} and \eqref{nest34}. \\
(III) Thirdly, we prove \eqref{nes11} and \eqref{nes13} for $(N,j) \in K_5$. 
By the definition, it suffices to show \eqref{nest21}--\eqref{nest24} for $i=3$, $(N, j) \in J_1 \cup J_2 \cup J_3$ 
and $\{  v_l \}_{l=1}^{2N+3} \subset C([-T, T]: H^s(\T))$. By Lemma~\ref{lem_pwb1}, it follows that 
\begin{equation}
| \ti{L}_{j, f}^{(2N+1)} \chi_{>L}^{(2N+1)}  | \lesssim 1, \label{nest41} 
\end{equation}
for $(N, j) \in J_1 \cup J_2 \cup J_3$. 
Thus, by \eqref{nest17}, we have \eqref{nest21} and \eqref{nest22}. 
Moreover, by \eqref{pwb21}, \eqref{nest19}, Lemma~\ref{lem_go} 
and a slight modification of \eqref{nest21} and \eqref{nest22} for $i=3$ and $(N, j) \in J_1 \cup J_2 \cup J_3$, 
we get \eqref{nest23} and \eqref{nest24}. \\
(IV) By the definition, \eqref{nest41} and Lemma~\ref{lem_go}, 
we easily check that \eqref{nes11} and \eqref{nes13} hold for $(N,j) \in K_6$. 

Therefore, we obtain \eqref{nes11} and \eqref{nes13} for $(N,j) \in K_3 \cup K_4 \cup K_5 \cup K_6$. 
\end{proof}

Now, we prove Proposition~\ref{prop_main1}. 
\begin{proof}[Proof of Proposition~\ref{prop_main1}] 
By the definition of $F_{f, L}$ and Lemma~\ref{lem_nes0}, we have (\ref{mes1}) and (\ref{mes3}).  
For $(N,j) \in J_1 \cup J_2 \cup J_3$, we can easily check
\begin{align} \label{mes20}
|L_{j, f}^{(2N+1)} \Phi_{f}^{(2N+1)} \chi_{ \le L}^{(2N+1)} | \lesssim L^2. 
\end{align} 
Thus, we have 
\begin{align} \label{mes200}
\Big\| \sum_{k=k_{1, \dots, 2N+1}} |\ti{L}_{j, f}^{(2N+1)} \Phi_{f}^{(2N+1)} \chi_{\le L}^{(2N+1)}| 
\, \prod_{l=1}^{2N+1} |\ha{v}_l(t, k_l)| \Big\|_{L_T^{\infty} l_s^2 } 
\lesssim L^2 \prod_{l=1}^{2N+1} \| v_l \|_{L_T^{\infty} H^s}
\end{align}
for any $\{ v_l \}_{l=1}^{2N+1} \subset C([-T, T]: H^s(\T))$. 
Hence, by the definition of $G_{f, L}$, Lemmas~\ref{lem_mle}--\ref{lem_nl10}, we have (\ref{mes2}). 
Next, we prove (\ref{mes4}). 
For $(N, j) \in J_1 \cup J_2 \cup J_3$, by the definition and (\ref{mes20}), it follows that 
\begin{align}
|\ti{L}_{j, f}^{(2N+1)} \Phi_{f}^{(2N+1)} \chi_{\le L}^{(2N+1)}- \ti{L}_{j, g}^{(2N+1)} \Phi_g^{(2N+1)} \chi_{\le L}^{(2N+1)}  | 
\lesssim  |\la_5| |E_1(f)- E_1(g)| L^2.  \label{mes23}
\end{align}
By (\ref{mes200}), (\ref{mes23}) and Lemma~\ref{lem_go}, we have 
\begin{align*}
\big\|  \La_f^{(2N+1)} ( \ti{L}_{j, f}^{(2N+1)} \Phi_f^{(2N+1)} \chi_{\le L}^{(2N+1)}, \ha{v} )
- \La_g^{(2N+1)} ( \ti{L}_{j, g}^{(2N+1)} \Phi_g^{(2N+1)} \chi_{\le L}^{(2N+1)}, \ha{v})  \big\|_{L_T^{\infty} l_s^2} \le C_*.
\end{align*}
Therefore, by the definition of $G_{f, L}$ and Lemmas~\ref{lem_mle}--\ref{lem_nl10}, we obtain \eqref{mes4}. 
\end{proof}

\section{proof of Theorem~\ref{thm_LWP}}

In this section, we give the proof of Theorem~\ref{thm_LWP}. 
By Lemma~\ref{lem_E1} below, it is verified that the following conservation law holds:
\begin{equation} \label{cle1}
E_1(u)(t)=E_1 (\vp)  \hspace{0,5cm} (t \in [-T, T])
\end{equation}
when $u \in C([-T,T]: H^{1}(\T))$ is a solution to \eqref{4NLS1} or \eqref{4NLS2} with \eqref{initial} and $\la_5= \la_2+ \la_4$. 
Thus, the unconditional local well-posedness for \eqref{4NLS1} with \eqref{initial} is equivalent to 
that for \eqref{4NLS2} with \eqref{initial} under the assumption $s \ge 1$ and $\la_5=\la_2+ \la_4$ or $\la_5=0$. 
\begin{lem} \label{lem_E1}
Let $s \ge 1$, $\la_5=\la_2+\la_4$, $\vp \in H^s(\T)$ 
and $u \in C([-T,T]: H^{s} (\mathbb{T}))$ be a solution to \eqref{4NLS1} or \eqref{4NLS2} with \eqref{initial}. 
Then it follows $\| u (t) \|_{L^2} = \| \varphi \|_{L^2}$ for any $t \in [-T,T]$. 
\end{lem} 
\begin{proof}
First, we suppose that $u$ is a solution to \eqref{4NLS1}--\eqref{initial}. 
For $N \in \N$, the projection operator $P_{\le N}$ is defined as 
\begin{equation*}
P_{\le N} u(t,x):= \sum_{|k| \le N} e^{i k x} \ha{u} (t, k). 
\end{equation*}
Then, $P_{\le N} u$ belongs to $C^1([-T, T]: H^{\infty} (\T))$ since 
\begin{align*}
& \p_t P_{\le N} u = i \p_x^4 P_{\le N} u + i \frac{3}{8} \la_1 P_{\le N}(|u|^4 u) +i (\la_5-\la_2) P_{\le N}( \bar{u} (\p_x u)^2  ) \notag \\
& \hspace{0.5cm} +i(2\la_4+\la_5- \la_3) P_{\le N} ( u |\p_x u|^2 ) 
 - i \la_4 \p_x P_{\le N} (u^2 \p_x \bar{u})   - i \la_5 \p_x P_{\le N}  ( |u|^2 \p_x u) 
\end{align*}
which is in $C([-T,T] : H^{\infty}(\T))$. 
Thus, we calculate $\frac{d}{dt} \| P_{\le N} u (t) \|_{L^2}^2 $ in the classical sense as follows: 
\begin{align} \label{cle11}
& \frac{d}{dt} \| P_{\le N} u(t) \|_{L^2}^2= 2 \text{Re} \int_{\T} \p_t P_{\le N}u \, P_{\le N} \bar{u} \, dx \notag \\
& = -\frac{3}{4} \la_1\, \text{Im}  \int_{\T} P_{\le N} (|u|^4 u) \, P_{\le N} \bar{u} \, dx 
+ 2 (\la_2-\la_5) \, \text{Im} \int_{\T} P_{\le N} (\bar{u} (\p_x u)^2) \, P_{\le N} \bar{u} \, dx \notag \\
& \hspace{0.3cm} + 2 (\la_3-2\la_4-\la_5) \, \text{Im} \int_{\T} P_{\le N}  (u |\p_x u|^2) \, P_{\le N} \bar{u} \, dx \notag \\
& \hspace{0.3cm} + 2 \la_4 \, \text{Im} \int_{\T} \p_x  P_{\le N} (u^2 \p_x \bar{u}) \, P_{\le N} \bar{u} \, dx 
+ 2 \la_5 \, \text{Im} \int_{\T} \p_x P_{\le N} (|u|^2 \p_x u) \, P_{\le N} \bar{u} \, dx \notag \\
& =: I_1+I_2+I_3+I_4+I_5. 
\end{align} 
A simple calculation yields that  
\begin{align*}
I_4= &  -2 \la_4 \, \text{Im} \int_{\T} P_{\le N}( u^2 \p_x \bar{u}) \, \p_x P_{\le N} \p_x \bar{u} \, dx \\
& =  -2 \la_4 \, \text{Im} \int_{\T} u^2 (\p_x \bar{u})^2 \, dx  
+  2 \la_4 \, \text{Im} \int_{\T} u^2 \p_x \bar{u} \, \p_x P_{>N} \bar{u} \, dx.   
\end{align*}
Similarly, it follows that 
\begin{align} 
& I_1+I_2+I_3+I_4+I_5= 2 (\la_5-\la_2-\la_4) \, \text{Im} \int_{\T} u^2 (\p_x \bar{u})^2 \, dx \notag \\
& \hspace{0.3cm} + 2(2 \la_4+\la_5-\la_3) \, \text{Im} \int_{\T} u P_{>N} \bar{u} |\p_x u|^2 \, dx 
+ 2(\la_5 -\la_2) \,  \text{Im} \int_{\T} \bar{u} P_{>N} \bar{u} (\p_x u)^2 \, dx  \notag  \\
& \hspace{0.3cm} + 2\la_4 \, \text{Im} \int_{\T} u^2 \p_x \bar{u} \p_x P_{>N} \bar{u}  \, dx 
+ 2 \la_5 \,  \text{Im} \int_{\T} |u|^2 \p_x u \p_x P_{>N} \bar{u}   \, dx  \notag  \\
& \hspace{0.3cm} + \frac{3}{4} \la_1 \, \text{Im} \int_{\T} |u|^4 u \, P_{>N} \bar{u} \, dx. \label{cle12}
\end{align}
By the assumption $\la_5= \la_2+ \la_4$, the fist term of the right hand side of \eqref{cle12} is equal to $0$. 
Thus, by the H\"{o}lder and the Sobolev inequalities, we have 
\begin{equation*}
|I_1+I_2+I_3+I_4+I_5| \lesssim (\| u \|_{L_T^{\infty} H^1}^3 + \| u \|_{L_T^{\infty} H^1}^5) \| P_{>N} u \|_{L_T^{\infty} H^1}. 
\end{equation*} 
Integrating \eqref{cle11} over $[0, t]$ for $t \in [-T, T]$, we have 
\begin{equation*}
\Big| \| P_{\le N} u(t) \|_{L^2}^2 - \| P_{\le N} \vp \|_{L^2}^2  \Big| 
\lesssim T (\| u \|_{L_T^{\infty} H^1}^3 + \| u \|_{L_T^{\infty} H^1}^5) \| P_{>N} u \|_{L_T^{\infty} H^1}. 
\end{equation*}
By uniform $H^1$-continuity on $[-T, T]$, the right hand side goes to $0$ as $N \to \infty$.  
By taking $N \to \infty$, $P_{\le N} u(t) \to u(t)$ for $t$ fixed and $P_{\le N} \vp \to \vp$ in $L^2(\T)$. 
Thus, we get 
\begin{equation*}
\Big| \| u(t)  \|_{L^2}^2 - \| \vp \|_{L^2}^2  \Big|=0
\end{equation*}
for any $t \in [-T, T]$. 

Next, we suppose that $u$ is a solution to \eqref{4NLS2}--\eqref{initial}. 
We notice that \eqref{4NLS2} is equivalent to
\begin{align*} 
& \p_t u -i \p_x^4 u + i \la_5 (E_1(\vp) - E_1(u)) \p_x^2 u \\
& \hspace{0.3cm} = 
i \frac{3}{8} \la_1 |u|^4 u -i \la_2 \bar{u} (\p_x u)^2- i \la_3 u |\p_x u|^2 -i \la_4 u^2  \p_x^2 \bar{u}-i \la_5 |u|^2 \p_x^2 u. 
\end{align*} 
Since $ \la_5 (E_1(\vp)-E_1(u) ) \text{Re}[ i \int_{\T} \p_x^2 P_{\le N} u \, P_{\le N} \bar{u} \, dx] =0$, 
we have the same result as above in the same manner. 
\end{proof}

Here, we define the translation operators $\mathcal{T}_{v}$ and $\mathcal{T}_{v}^{-1}$ by 
\begin{align*}
\mathcal{T}_v u(t, x)& := u \Big( t,  x+ \int_0^t [(\la_3-2 \la_2) E_2(v) ] (t') \, dt'  \Big), \\
\mathcal{T}_{v}^{-1} u(t,x)& :=u \Big(t, x- \int_0^t [(\la_3- 2 \la_2) E_2(v)] (t') \, dt'  \Big) 
\end{align*}
for $v \in C([-T, T]: H^{1/2}(\T))$. 
In a similar manner as Proposition 8.1 in \cite{KTT}, the following lemma holds. 
\begin{lem} \label{lem_homeo}
Let $s \ge 1/2$. Then, a map $\mathcal{T}: u \mapsto \mathcal{T}_u u$ is a homeomorphism on 
$C([-T,T]: H^s(\T))$.
\end{lem}
We notice that $\mathcal{T}^{-1}u = \mathcal{T}_{u}^{-1} u $ holds for $u \in C([-T, T]: H^s(\T))$ with $s \ge 1/2$. 
Thus, the unconditional local well-posedness for \eqref{4NLS2} with \eqref{initial} is equivalent to 
that for \eqref{4NLS3} with \eqref{initial} in $C([-T, T]: H^s(\T))$ for $s \ge 1$. 
Therefore, Theorem~\ref{thm_LWP} is equivalent to Proposition~\ref{prop_LWP} below and we only show it. 

\begin{prop} \label{prop_LWP}
Let $s \geq 1$ and $\la_5=\la_2+\la_4$ or $\la_5=0$. Then, for any $\varphi \in H^s (\mathbb{T})$, 
there exists $T=T(\| \varphi \|_{H^{s}}) >0$ and a unique solution $u \in C([-T,T]: H^{s} (\T))$ to 
\eqref{4NLS3} with \eqref{initial}. 
Moreover, the solution map, $H^s(\T) \ni \varphi \mapsto u \in C([-T,T]: H^s (\T))$ is continuous.  
\end{prop}

As we proved in Proposition \ref{prop_NF2}, for any solution $u \in C([-T,T]:H^s(\T))$ to \eqref{4NLS3}, 
$\ha{v}=e^{-i t\phi_\vp(k)}\ha{u}$ satisfies \eqref{NF21}.
By the standard iteration argument with Proposition \ref{prop_main1}, we can prove the existence of solutions of \eqref{NF21}.
However, this argument is useless to show the existence of solutions to \eqref{4NLS3} 
since we do not know whether $u:=U_{\vp}(t) v$ satisfies \eqref{4NLS3} or not when $\ha{v}$ satisfies \eqref{NF21}.
To overcome this difficulty, we use the existence of the solution to \eqref{4NLS1}--\eqref{initial} for sufficiently smooth initial data.
By Theorem 1.1 in \cite{Se4}, we have the following result. 
\begin{prop} \label{prop_existence}
Let $m \in \N$ be sufficiently large.  Then, \eqref{4NLS1}--\eqref{initial} is locally well-posed in $H^m(\T)$ on $[-T,T]$ without any condition on $\la_1, \la_2, \la_3, \la_4$ and $\la_5$. The existence time $T=T(\|\vp\|_{H^m})>0$.
\end{prop}
By Corollary~\ref{cor_mainest} and Proposition~\ref{prop_existence}, we show Proposition~\ref{prop_LWP}. 
For details, see the proof of Proposition 8.4 in \cite{KTT2}.

\end{document}